\documentclass[11pt]{amsart}

\usepackage{amsmath,amssymb,amsthm}
\usepackage{upgreek}
\usepackage{amsaddr}
\usepackage{mathtools}
\usepackage{enumitem}
\usepackage{geometry}
\usepackage{tikz}
\usepackage{longtable}
\usepackage{booktabs}
\usepackage{array}
\usepackage{MnSymbol}
\usetikzlibrary{arrows.meta,calc,patterns,decorations.pathmorphing}
\usepackage{graphicx}
\usepackage{subcaption}
\usepackage{hyperref}

\newtheorem{theorem}{Theorem}[section]
\newtheorem{lemma}[theorem]{Lemma}
\newtheorem{proposition}[theorem]{Proposition}
\newtheorem{corollary}[theorem]{Corollary}
\newtheorem{conjecture}[theorem]{Conjecture}

\theoremstyle{definition}
\newtheorem{definition}[theorem]{Definition}
\newtheorem{example}[theorem]{Example}

\theoremstyle{remark}
\newtheorem{remark}[theorem]{Remark}

\newcommand{\RP}{\mathbb{RP}}
\newcommand{\R}{\mathbb{R}}
\newcommand{\T}{T}
\newcommand{\That}{\widehat T}
\newcommand{\Area}{A}
\newcommand{\sArea}{A^*}
\newcommand{\dist}{\operatorname{dist}}
\newcommand{\diam}{\operatorname{diam}}
\newcommand{\conv}{\operatorname{conv}}
\newcommand{\Span}{\operatorname{span}}
\newcommand{\Spec}{\operatorname{Spec}}
\newcommand{\tr}{\operatorname{tr}}
\newcommand{\constpi}{\uppi}
\newcommand{\conste}{\mathrm{e}}
\newcommand{\consti}{\mathrm{i}}
\newcommand{\doi}[1]{\href{https://doi.org/#1}{\nolinkurl{doi:#1}}}
\newcommand{\GL}{\operatorname{GL}}
\newcommand{\sn}{\operatorname{sn}}
\newcommand{\cn}{\operatorname{cn}}
\newcommand{\dn}{\operatorname{dn}}
\newcommand{\cd}{\operatorname{cd}}
\renewcommand{\d}{\mathrm{d}}

\numberwithin{equation}{section}

\title[Area-Normalized Pentagram Map]
{Area-Normalized Pentagram Map Dynamics:\\
Spectral Flattening and Elliptic Asymptotics}

\author{Micha\l{} Zwierzy\'nski$^{\pentagram}$}

\email{$^{\pentagram}$Michal.Zwierzynski@pw.edu.pl}
\email{ORCID: 0000-0002-9627-1563}

\address{Warsaw University of Technology\\
Faculty of Mathematics and Information Science\\
ul. Koszykowa 75\\
00-662 Warsaw, Poland}



\begin{document}

\begin{abstract}
The~pentagram map sends a~polygon to the~intersections of consecutive short diagonals. We study its shape dynamics after translating and positively rescaling each iterate to restore unsigned area and barycenter. For polygons whose dynamics is generated, after passing to a~finite iterate and cyclic relabelling, by one projectivity, we prove a~spectral dichotomy. A~dominant real projective line yields flattening and unbounded diameter, whereas a~dominant real eigenvalue with a~subdominant non-real pair produces asymptotic motion on concentric homothetic ellipses.

We apply this framework to pentagons and hexagons via Glick's operator and to Poncelet polygons via the~Darboux--Schwartz projectivity. We obtain spectral diagrams and a~Jacobi-function formula for the~Poncelet return spectrum. Consequently, we recover Schwartz's long-and-thin result for strictly convex non-projectively-regular pentagons, prove a~line-or-ellipse dichotomy for convex hexagons under explicit nondegeneracy assumptions, and establish flattening for strictly convex non-projectively-regular Poncelet polygons.
\end{abstract}

\maketitle

\section*{Declaration of interests}

\noindent The author declares that he has no known competing financial interests or personal relationships that could have appeared to influence the work reported in this paper.

\section{Introduction}

\noindent The~pentagram map was introduced by Schwartz as a~transformation of
polygons in the~projective plane \cite{SchwartzPentagram}.  It sends a~
polygon to the~polygon formed by intersections of consecutive short
diagonals.  In affine coordinates, for a~labeled polygon
$P=(p_1,\ldots,p_n)$, it is given by
$$
\T(P)_i
=
(p_ip_{i+2})\cap(p_{i+1}p_{i+3}),
$$
with indices taken modulo $n$ (see
Figure~\ref{fig:pentagramIterationsIntroA}).  Since its introduction,
the~pentagram
map has become a~basic object in discrete projective geometry.  It is
one of the~central examples of a~discrete integrable system.  Ovsienko, Schwartz, and Tabachnikov described its Poisson geometry and
monodromy invariants on spaces of twisted polygons \cite{OST} and
proved Liouville--Arnold integrability on the~moduli space of closed
polygons \cite{OSTClosed}.  Soloviev subsequently gave a~Lax
representation and an~algebraic-geometric description through a~
spectral curve and its Jacobian \cite{Soloviev}; see also
\cite{Weinreich} for an~algebraic treatment over general algebraically
closed fields.
Complementary algebraic and combinatorial descriptions identify the~
pentagram map with cluster $Y$-dynamics and with transformations of
directed networks
\cite{GekhtmanShapiroTabachnikovVainshtein, GlickYPatterns}.
A~refactorization approach in Poisson--Lie groups of
pseudo-difference operators provides a~unified framework for the~
classical map and its known integrable higher-dimensional
generalizations \cite{IzosimovRefactorization}.  Foundational
higher-dimensional extensions and their integrability were developed in
\cite{KhesinSolovievHigher}; long-diagonal maps later unified the~known
integrable cases \cite{IzosimovKhesinLongDiagonal}.  Further
generalizations can be found in
\cite{FelipeMariBeffa,KhesinSolovievDented,MariBeffaAGD,Ovenhouse,WangQNets}.
More recently, these and further variants were incorporated into a~
common framework of pentagram maps over rings
\cite{HandIzosimov}.
The~map is also closely connected with Poncelet polygons and
projective invariants
\cite{IzosimovPoncelet,SchwartzPoncelet}; related rigidity for
deep-diagonal dynamics was proved for centrally symmetric octagons in
\cite{SchwartzRigidityOctagons}.

\begin{figure}[htbp]
    \centering

    \begin{subfigure}[t]{0.44\textwidth}
        \centering
        \includegraphics[width=\textwidth]{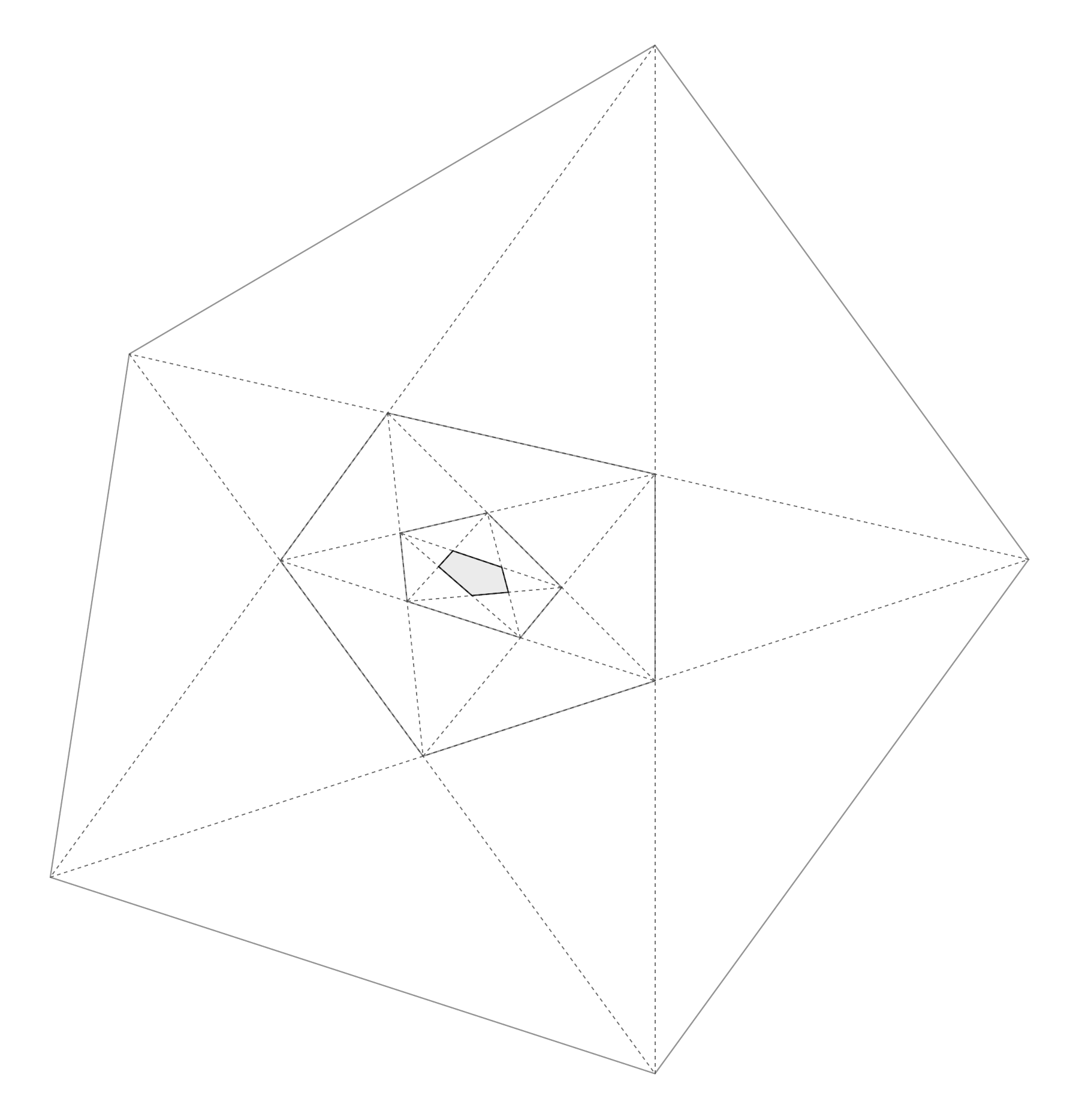}
        \caption{$3$ iterations of the~classical pentagram map}
        \label{fig:pentagramIterationsIntroA}
    \end{subfigure}
    \hfill
    \begin{subfigure}[t]{0.44\textwidth}
        \centering
        \includegraphics[width=\textwidth]{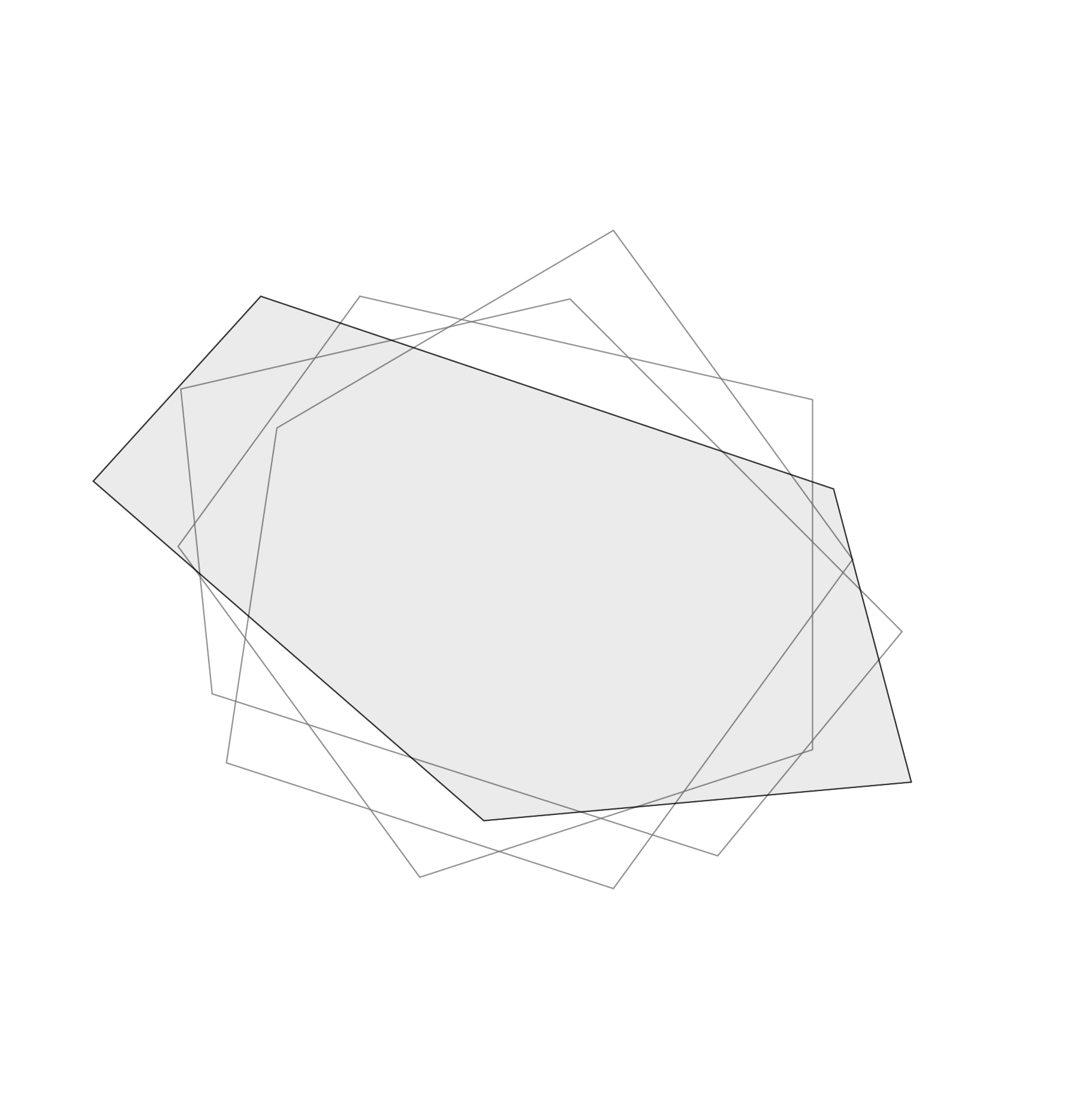}
        \caption{$3$ iterations of the~area-normalized pentagram map}
        \label{fig:pentagramIterationsIntroB}
    \end{subfigure}

    \vspace{0.5cm}

    \begin{subfigure}[t]{0.48\textwidth}
        \centering
        \includegraphics[width=\textwidth]{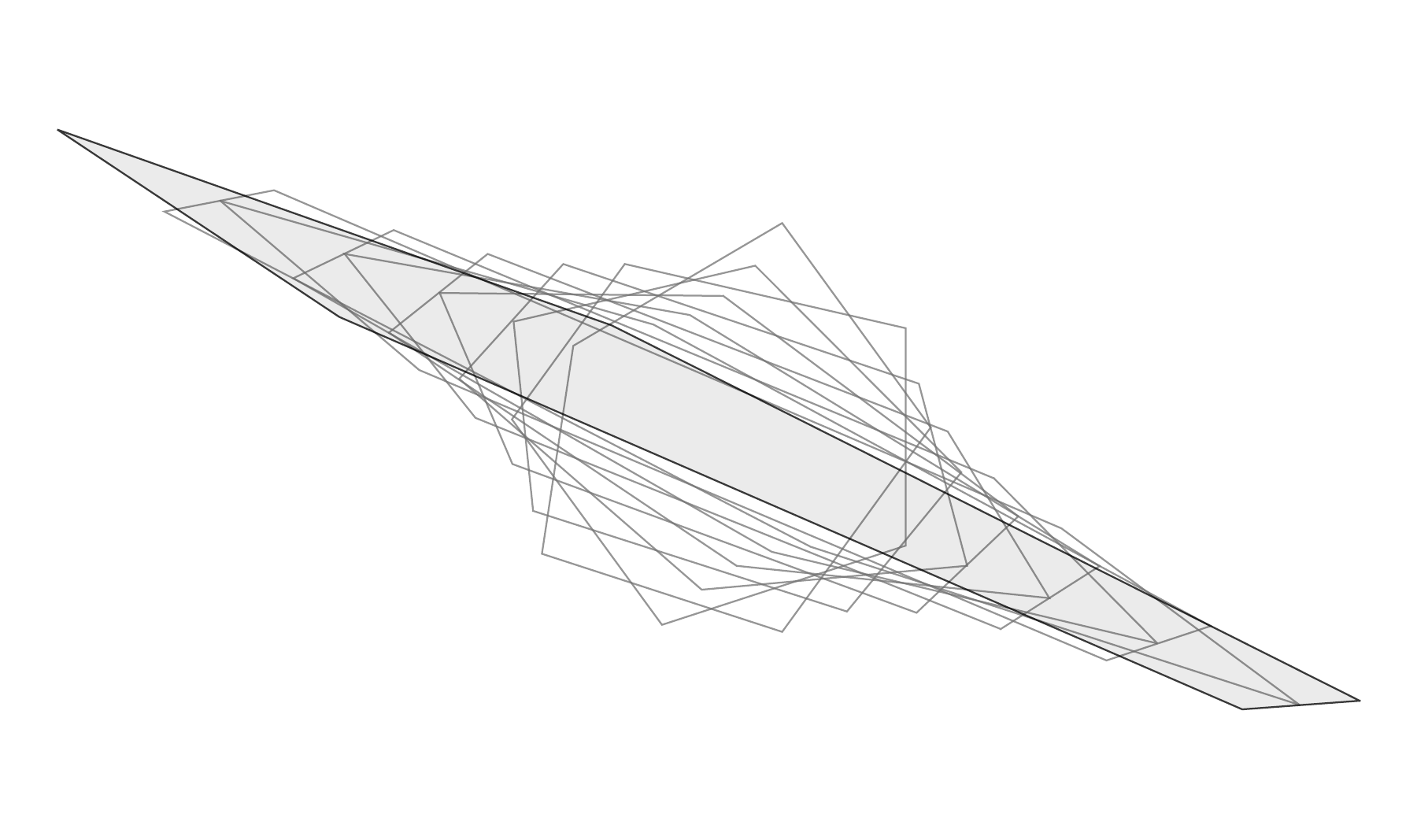}
        \caption{$10$ iterations of the~area-normalized pentagram map}
        \label{fig:pentagramIterationsIntroC}
    \end{subfigure}
    \hfill
    \begin{subfigure}[t]{0.48\textwidth}
        \centering
        \includegraphics[width=\textwidth]{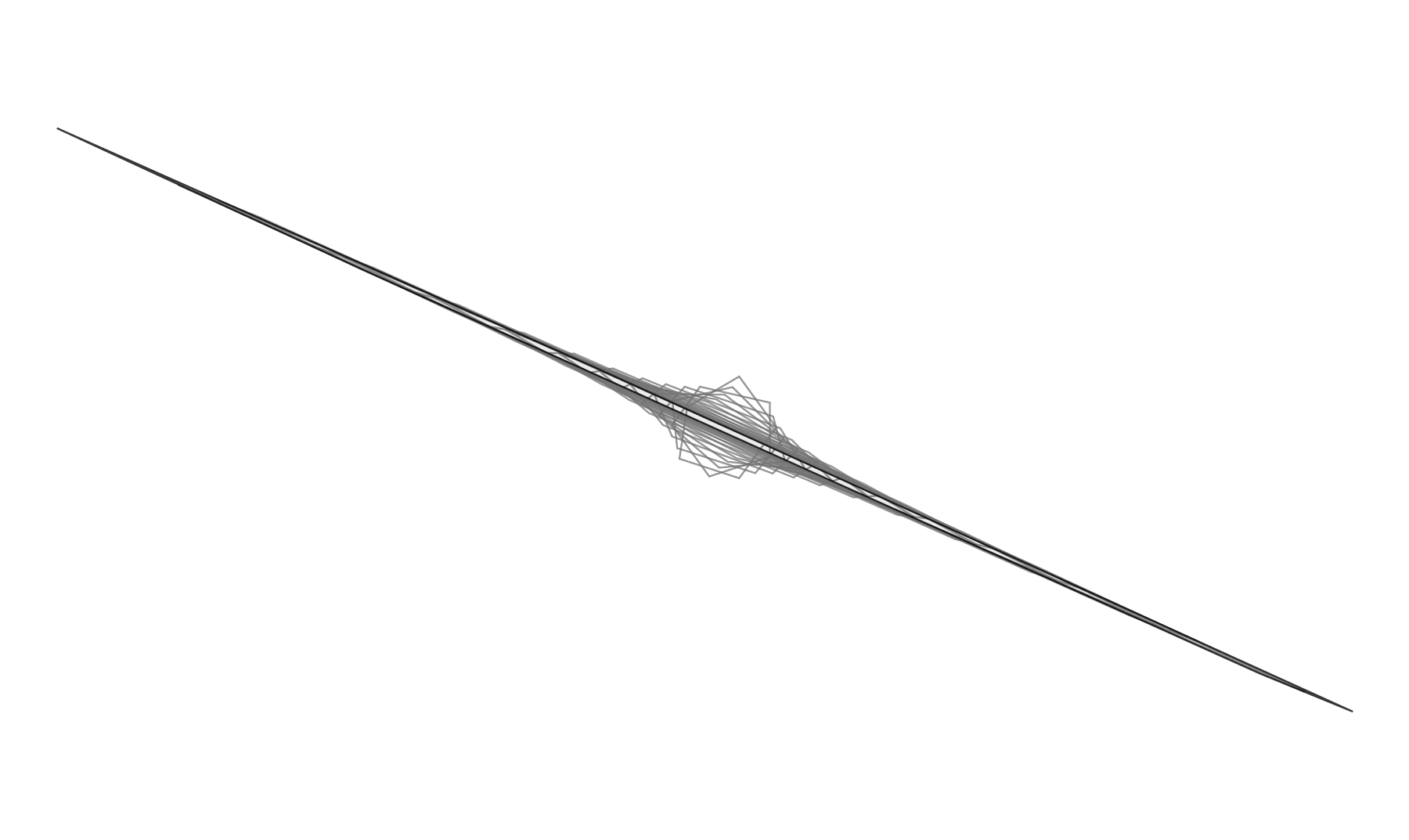}
        \caption{$20$ iterations of the~area-normalized pentagram map}
        \label{fig:pentagramIterationsIntroD}
    \end{subfigure}

    \caption{Iterations of the~classical and area-normalized maps for a~fixed pentagon}
    \label{fig:pentagramIterationsIntro}
\end{figure}

Schwartz also proved recurrence on the~moduli space of convex
polygons \cite{SchwartzRecurrent} and exhibited special finite-time
linear degenerations for axis-aligned polygons
\cite{SchwartzDiscreteMonodromy}.  In a~related setting, Schwartz
studied collapse toward a~limit point for infinite polygonal paths
called pentagram spirals and discussed self-similar asymptotic shapes
in special cases \cite{SchwartzSpirals}.  These phenomena are distinct from the~
asymptotic Euclidean shapes studied here.  We keep the~same projective
orbit, but choose a~Euclidean representative after each step and analyze
its behavior under area normalization.

A~classical theorem of Schwartz says that, for a~convex polygon, the~
ordinary pentagram iterates shrink exponentially to a~point
\cite{SchwartzPentagram}.  Importantly for the~present paper, Schwartz
also studied the~normalized asymptotic shape in the~pentagonal case.  He
showed that the~differential at the~collapse point of the~projectivity
carrying a~convex pentagon to its pentagram image is real-diagonalizable,
and concluded that, outside the~projectively regular case (that is,
outside the~projective class of a~regular pentagon),
constant-area rescalings become exponentially long and thin
\cite[Theorem~2.2]{SchwartzPentagram}.  This result is the~direct
motivation for our work: we ask how far the~pentagonal phenomenon
extends, which spectral data select the~limiting direction, and what
replaces flattening when the~first centered spectral term is genuinely
two-dimensional.

The~purpose of this paper is to place Schwartz's normalized
pentagonal result in a~spectral framework that also covers hexagons,
Poncelet polygons, and, conditionally, any orbit generated by an~iterate
of one projectivity.  Unlike Schwartz's normalized-shape theorem, which
concerns convex pentagons, the~flattening criterion proved in the~following sections is
formulated in terms of local projective spectral dynamics and does not
itself require convexity.  It therefore also applies to nonconvex
labeled pentagons and hexagons whenever their
forward iterates remain defined, have nonzero unsigned area, and satisfy the~stated
spectral and centered two-scale nondegeneracy assumptions.

After each pentagram step we apply a~homothety and a~translation so that
the~unsigned area and the~barycenter are restored to fixed values.  This does not
change the~underlying projective orbit, because a~homothety and a~
translation are projective transformations.  It only chooses a~Euclidean
gauge in which the~collapsing sequence can be seen under a~changing
microscope.  We then ask whether the~resulting shapes have asymptotic
geometric structure: do they become flat, do they oscillate, or do they
approach some other limiting family?

Schwartz's long-and-thin behavior for pentagons persists in a~broader
spectral setting: the~unnormalized iterates may collapse
projectively, while the~area-normalized polygons converge transversely
to a~line (see Figure~\ref{fig:pentagramIterationsIntro}, especially
panels~\subref{fig:pentagramIterationsIntroB}--\subref{fig:pentagramIterationsIntroD}).
Our numerical experiments also
reveal a~complementary possibility.  When a~real eigenvalue is dominant
and the~next spectral block is a~complex-conjugate pair, one observes
bounded elliptic oscillations rather than flattening.

To make this precise, we define an~area-normalized pentagram map
$\That$.  Let
$$
C(P)=\frac1n\sum_{i=1}^n p_i
$$
be the~vertex barycenter of a~polygon $P$ with vertices
$p_1,p_2,\ldots,p_n$, and let
$$
\sArea(P)
=
\frac12\sum_{i=1}^n \det(p_i,p_{i+1})
$$
be the~oriented area of $P$.  We write
$
\Area(P)=|\sArea(P)|
$
for the~absolute value of the~oriented area of $P$.
Whenever $\Area(P)>0$ and $\Area(\T(P))>0$, set
\begin{align}
\label{eq:area-normalized-map}
\That(P)_i
&=
C(P)
+
\sqrt{\frac{\Area(P)}{\Area(\T(P))}}
\bigl(\T(P)_i-C(\T(P))\bigr).
\end{align}
Thus $\That(P)$ is obtained from $\T(P)$ by a~positive homothety and a~
translation chosen so that
$\Area(\That(P))=\Area(P)$,
$C(\That(P))=C(P)$.

The~first observation is that iterating $\That$ does not introduce a~new
projective dynamics.  By affine naturality, a~direct induction shows
that, whenever the~relevant iterates have nonzero unsigned area, $\That^k(P)$ is
obtained from the~classical iterate $\T^k(P)$ by a~single homothety and
a~translation restoring the~initial unsigned area and barycenter.  This makes
the~asymptotic shape of $\That^k(P)$ accessible once one understands
the~anisotropic collapse of $\T^k(P)$.

The~second observation is that Glick's operator $L_P$ is the~natural
spectral object for this problem.  We recall its definition and the~
properties needed here in Section~\ref{sec:glick-spectral-line}; at this
point it is enough to know that it is a~projectively natural linear
operator associated with the~polygon $P$ \cite{GlickLimit}.  For
pentagons and hexagons, the~projectivity governing the~relevant
pentagram step is not arbitrary: it is
\begin{align}
\label{eq:glick-projective-generator}
G_P&=L_P-3I.
\end{align}
With the~labeling convention used here, where
$\Sigma_s(P)_i=p_{i+s}$, one has
$\T(P)=\Sigma_{-1}(G_P(P))$
for pentagons and
$\T^2(P)=G_P(P)$
for hexagons.  Cyclic relabelling does not affect any of the~Euclidean
setwise conclusions considered below.  Thus pentagons and hexagons fit
the~same spectral framework.

For Poncelet polygons, Schwartz's Darboux-type theorem
\cite{SchwartzPoncelet} implies that
$\T^2(P)$ is projectively equivalent to $P$.  The~corresponding
Darboux--Schwartz projectivity commutes with Glick's operator.  Hence,
when $L_P$ has simple spectrum, the~same eigenpoints appear again.
In Subsection~\ref{subsec:poncelet-spectrum} we sharpen this input by
putting a~Poncelet family into Jacobi normal form and computing the~
return projectivity explicitly.  For every strictly convex Poncelet
polygon that is not projectively regular, this return projectivity has
three positive, real, pairwise distinct eigenvalues.  Thus the~general
flattening criterion becomes unconditional in this class.

For pentagons and hexagons we obtain explicit spectral diagrams.
Theorem~\ref{thm:intrinsic-pentagonal-spectrum} identifies a~single
intrinsic invariant $\kappa(P)=-\det(G_P)$ that determines the~entire
pentagonal characteristic polynomial and separates all real and complex
spectral regimes.  Its balanced bracket formula makes projective
invariance explicit, while Lemma~\ref{lem:pentagonal-normal-form} gives
the~corresponding rational expression in a~projective chart and proves
the~strictly convex bound.  The~remaining generic spectral types occur
only on the~nonconvex locus; see
Figure~\ref{fig:pentagon-spectral-regions}.
Lemma~\ref{lem:centered-spectral-invariants} gives intrinsic
trace--determinant and bracket formulas for the~centered spectral
coefficients of $L_P$.  Proposition~\ref{prop:hexagon-spectral-diagram}
specializes them to two projective coefficients $\mathcal S$ and
$\mathcal R$ for every hexagon and describes the~complete cubic
spectral diagram.  Corollary~\ref{cor:convex-hexagon-spectral-restriction}
shows that strict convexity forces $\mathcal S,\mathcal R>0$; in this
positive chamber a~non-real conjugate pair is always subdominant.
The~Poncelet return spectrum gives a~third, one-parameter diagram:
Corollary~\ref{cor:convex-poncelet-spectrum} places its strictly convex
locus entirely in the~positive real simple-spectrum chamber.
These results combine with the~general flattening and elliptic criteria
to recover Schwartz's pentagonal long-and-thin conclusion spectrally
for every strictly convex pentagon outside the~projectively regular
class (Corollary~\ref{cor:convex-pentagon-flattening}) and
to prove the~nondegenerate line-or-ellipse dichotomy of
Theorem~\ref{thm:convex-hexagon-dichotomy}.

The~general asymptotic analysis also yields finer information.  In the~
real two-scale regime, Corollary~\ref{cor:hyperbolic-vertex-asymptotics}
identifies the~leading vertex asymptotics: whenever both leading
coefficients are nonzero, the~corresponding compensated subsequence is
asymptotic to a~hyperbola whose asymptotes are the~longitudinal and
transverse spectral lines.  If the~projective return occurs only after
$\T^m$, the~orbit splits into $m$ residue classes, each governed by the~
same return projectivity.  Exact nonconvex heptagonal and octagonal
examples have minimal return times $m=3$ and $m=4$; their computed
normalized orbits illustrate, respectively, the~elliptic and flattening
regimes.  Theorem~
\ref{thm:principal-lines-convergence} also shows that, in the~
flattening regime, the~principal lines of the~normalized polygons
converge in direction to the~selected spectral line.

For general $n\geq7$, the~conserved operator $L_P$ no longer generates
the~orbit.  The~spectral calculations and examples nevertheless lead
to two conjectural line-selection regimes: distinct real spectra select
one of the~two spectral tangent lines through the~projective collapse
point, while a~dominant non-real pair selects its canonical real
invariant line.  These
observations are formulated in
Conjecture~\ref{conj:tangent-line-selection}.  In contrast, when the~
real eigenvalue is dominant and the~remaining pair is non-real, no
comparable universal pattern has emerged outside the~projectively
periodic setting.

Section~\ref{sec:normalized-map} defines the~normalization and its
forward domain, and Section~\ref{sec:glick-spectral-line} develops the~
spectral invariants.  Section~\ref{sec:two-scale-criterion} proves the~
flattening results and computes the~Poncelet, pentagonal, and hexagonal
spectra.  Section~\ref{sec:elliptic-oscillatory} treats the~
complementary elliptic regime, Section~\ref{sec:examples-evidence}
separates exact returns from numerical evidence, and
Section~\ref{sec:conjectures} states the~remaining conjectures.

Symbolic and numerical computations, including the~spectral
calculations and the~generation of the~figures, were carried out using
Wolfram Mathematica~\cite{Mathematica}.

\section{The~area-normalized pentagram map}
\label{sec:normalized-map}

\noindent Let $P=(p_1,\ldots,p_n)$ be a~labeled polygon in $\R^2$.  We always
assume that all indices are read modulo $n$.

The~pentagram map is defined whenever the~lines
$p_ip_{i+2}$ and $p_{i+1}p_{i+3}$ are well defined and not parallel.
Then
$$
\T(P)_i
=
(p_ip_{i+2})\cap(p_{i+1}p_{i+3}).
$$

The~map is projectively natural: if $F$ is a~projective transformation
for which both sides are defined, then $\T(F(P))=F(\T(P))$.
In particular, it is affinely natural.

The~area-normalized map $\That$ is the~map defined in
\eqref{eq:area-normalized-map}; throughout, its domain is restricted to
polygons for which $\T(P)$ is defined and both the~current and next
unsigned areas are nonzero.

Let
$$
Q_k=\T^k(P)
\qquad\text{and}\qquad
P_k=\That^k(P).
$$
Whenever all the~iterates involved are defined and have nonzero
unsigned area,
affine naturality and induction give
\begin{align}
\label{eq:normalization-classical-iterates}
P_k
&=
C(P)
+
\sqrt{\frac{\Area(P)}{\Area(Q_k)}}
\bigl(Q_k-C(Q_k)\bigr).
\end{align}
Indeed, at each step the preceding affine normalization can be
commuted through the~pentagram map, while the~successive area-scaling
factors telescope.  Thus the~normalized dynamics is simply the~
classical dynamics viewed under a~changing Euclidean microscope.

We also record here the~elementary genericity statement concerning the~
domain of the~normalized map.  Placing it at this point separates the~
question of whether the~orbit is defined from the~later spectral
assumptions.

\begin{proposition}[Generic forward domain]
\label{prop:generic-forward-domain}
Fix $n\geq5$ and put $\mathcal P_n=(\R^2)^n$.  For $N\geq0$, let
$\mathcal U^{(N)}\subset\mathcal P_n$ be the~set of labeled polygons
for which $\T^k(P)$ is defined and has nonzero oriented area for every
$0\leq k\leq N$.  Then $\mathcal U^{(N)}$ is open and dense.  In every
affine coordinate chart it is the~common nonvanishing set of finitely
many nonzero polynomials, and hence its complement is a~finite union of
proper real-algebraic subsets.

Consequently, the set
$$
\mathcal U^\infty
=
\bigcap_{N=0}^{\infty}\mathcal U^{(N)}
$$
is residual and dense in $\mathcal P_n$.
\end{proposition}

\begin{proof}
The~oriented area is a~nonzero polynomial in the~vertex coordinates,
and the~condition that the~two short diagonals defining a~vertex of
$\T(P)$ are not parallel is the~nonvanishing of a~determinant.  Since
$\T$ is rational on its one-step domain, the~conditions that the~next
iterate be defined and have nonzero oriented area are again
nonvanishing conditions for rational functions.  Repeating this up to
time $N$ and clearing all previously occurring denominators produces a~
finite list of polynomial nonvanishing conditions.

None of these polynomials is identically zero: a~regular convex
$n$-gon has all forward pentagram iterates defined and of nonzero area.
Thus $\mathcal U^{(N)}$ is open and dense.  Finally,
$\mathcal P_n$ is a~complete metric space, so the~Baire category theorem
implies that the~countable intersection $\mathcal U^\infty$ is residual
and dense.
\end{proof}

No openness is asserted for $\mathcal U^\infty$.  Although each
finite-time domain is open, exceptional algebraic sets arising at later
and later iterates may accumulate.  Finite-time degenerations of special
axis-aligned polygons illustrate the~kind of boundary behavior that
can occur \cite{SchwartzDiscreteMonodromy}.

\section{Glick's operator, centered invariants, and the~spectral line}
\label{sec:glick-spectral-line}

\noindent We now recall Glick's operator \cite{GlickLimit}.  Let
$P=(p_1,\ldots,p_n)$ be a~labeled projective polygon such that every
consecutive triple $p_{i-1},p_i,p_{i+1}$ is noncollinear, and choose
nonzero homogeneous lifts $\widetilde p_i\in\R^3$.
Define
\begin{align}
\label{eq:glick-operator}
L_P(v)
&=
nv-
\sum_{j=1}^n
\frac{\det\left(\widetilde p_{j-1},v,\widetilde p_{j+1}\right)}
{\det\left(\widetilde p_{j-1},\widetilde p_j,\widetilde p_{j+1}\right)}
\widetilde p_j.
\end{align}
This definition is projective: it is independent of the~chosen lifts.
Aboud and Izosimov later interpreted $L_P$, up to addition of a~scalar
operator, as the~infinitesimal monodromy associated with the~scaling
deformation of a~closed polygon \cite{AboudIzosimov}.

The~following results of Glick are the~basic input for our work.  We use
only the~parts of Glick's theory which are needed for the~spectral
arguments below.

\begin{proposition}[Glick \cite{GlickLimit}]
\label{prop:glick}
Let $P$ lie in the~domain just specified.  Each assertion involving
pentagram images or iterates is made on the~additional common domain on
which those images or iterates are defined.  Then $L_P$ has the~
following properties.
\begin{enumerate}[label=\textup{(\roman*)}]
\item If $\psi\in \GL(3,\R)$ represents a~projective transformation,
then $L_{\psi(P)}=\psi L_P\psi^{-1}.$
In particular, the~characteristic polynomial of $L_P$ is a~projective
invariant.
\item The~operator is conserved by the~pentagram map: $L_{\T(P)}=L_P.$
\item If $P$ is convex and $\T^k(P)\to X,$
then the~homogeneous lift of $X$ is an~eigenvector of $L_P$.
\item With the~convention $\Sigma_s(P)_i=p_{i+s}$ used here,
$L_P-3I$ sends a~pentagon $P$ to $\Sigma_1(\T(P))$ and sends a~hexagon
$P$ to $\T^2(P)$.  If $L_P-3I$ is invertible, these linear identities
realize the~corresponding projective equivalences.
\item The~trace of $L_P$ is $2n$.
\end{enumerate}
\end{proposition}

The~trace identity is Proposition~2.2 of \cite{GlickLimit}.
Items~\textup{(i)}--\textup{(iii)} are, respectively,
Proposition~2.3, Theorem~3.1, and Proposition~1.2 of that paper; the~
pentagon and hexagon identities in \textup{(iv)} are recorded in its
introduction.

\begin{remark}
The~last statement means that, for pentagons and hexagons, the~
projectivity controlling the~relevant pentagram iterate is canonical.
It is not obtained by choosing four point correspondences; it is
provided directly by $L_P$.
\end{remark}

We next isolate two projective invariants that encode the~characteristic
polynomial of $L_P$ for every number of vertices.  They will later
specialize to the~one-parameter pentagonal diagram and the~
two-parameter hexagonal diagram.

\begin{lemma}[Centered spectral invariants and bracket formulas]
\label{lem:centered-spectral-invariants}
Let $P=(p_1,\ldots,p_n)$ be a~labeled polygon in the~domain of
Glick's operator, and choose arbitrary nonzero homogeneous lifts
$\widetilde p_i\in\R^3$.  Put
$$
[ijk]
=
\det(\widetilde p_i,\widetilde p_j,\widetilde p_k)
$$
and, with indices understood modulo $n$,
$$
c_{ij}
=
\frac{[i-1,j,i+1]}{[i-1,i,i+1]}.
$$
Define the~centered operator and its two spectral coefficients by
$$
\widehat L_P
=
L_P-\frac{2n}{3}I,
\qquad
\mathcal S_n(P)
=
\frac12\tr(\widehat L_P^2),
\qquad
\mathcal R_n(P)
=
\det(\widehat L_P).
$$
Then $\mathcal S_n$ and $\mathcal R_n$ are projective invariants and
\begin{align}
\label{eq:centered-characteristic-polynomial}
\chi_{L_P}(t)
&=
\left(t-\frac{2n}{3}\right)^3
-
\mathcal S_n(P)
\left(t-\frac{2n}{3}\right)
-
\mathcal R_n(P).
\end{align}
Moreover,
$$
\mathcal S_n(P)
=
\frac12
\left(
\sum_{i,j=1}^n c_{ij}c_{ji}
-
\frac{n^2}{3}
\right),
\quad\text{and}\quad
\mathcal R_n(P)
=
-\frac{2n^3}{27}
+
\frac n3
\sum_{i,j=1}^n c_{ij}c_{ji}
-
\frac13
\sum_{i,j,k=1}^n
c_{ij}c_{jk}c_{ki}.
$$
In particular, the~spectrum of $L_P$ is completely determined by the~
pair
$$
\bigl(\mathcal S_n(P),\mathcal R_n(P)\bigr).
$$
\end{lemma}

\begin{proof}
For each $i$, let
$$
\varphi_i(v)
=
\frac{
\det(\widetilde p_{i-1},v,\widetilde p_{i+1})
}{
[i-1,i,i+1]
}.
$$
If $\mathcal B_P$ denotes the~rank-one sum
$$
\mathcal B_P(v)
=
\sum_{i=1}^n
\varphi_i(v)\widetilde p_i,
$$
then Glick's operator \eqref{eq:glick-operator} becomes $L_P=nI-\mathcal B_P.$
Since $\varphi_i(\widetilde p_j)=c_{ij}$, the~trace identities for
products of rank-one operators give
$$
\tr(\mathcal B_P^2)
=
\sum_{i,j=1}^n c_{ij}c_{ji}
\quad\text{and}\quad
\tr(\mathcal B_P^3)
=
\sum_{i,j,k=1}^n c_{ij}c_{jk}c_{ki}.
$$
Also, $c_{ii}=1$, so $\tr \mathcal B_P=n$, consistently with
$\tr L_P=2n$.

Now
$\widehat L_P=\frac n3 I-\mathcal B_P$.
Expanding the~trace of its square yields $
\tr(\widehat L_P^2)
=
\tr(\mathcal B_P^2)-\frac{n^2}{3},$
which proves the~formula for $\mathcal S_n$.  Since
$\tr\widehat L_P=0$, Newton's identities for a~$3\times3$ matrix give $\det(\widehat L_P)
=
\frac13\tr(\widehat L_P^3).$
Expanding the~cube and using $\tr \mathcal B_P=n$ gives
$$
\det(\widehat L_P)
=
-\frac{2n^3}{27}
+
\frac n3\tr(\mathcal B_P^2)
-
\frac13\tr(\mathcal B_P^3),
$$
which is the~stated formula for $\mathcal R_n$.

For every traceless $3\times3$ matrix $A$,
$$
\det(zI-A)
=
z^3-\frac12\tr(A^2)z-\det A.
$$
Applying this to $A=\widehat L_P$ proves the~characteristic-polynomial
formula.  Finally, projective covariance of $L_P$ implies covariance of
$\widehat L_P$ by conjugation, so both
$\tr(\widehat L_P^2)$ and $\det(\widehat L_P)$ are projectively
invariant.
\end{proof}

\begin{remark}[Area ratios, balanced brackets, and cross-ratios]
\label{rem:bracket-interpretation}
For the~standard affine lifts
$\widetilde p_i=(x_i,y_i,1)^{\top}$, one has
$$
[ijk]
=
\det(p_j-p_i,p_k-p_i)
=
2\sArea(p_ip_jp_k).
$$
Thus a~bracket is twice the~oriented area of the~corresponding
triangle.  A~single bracket is not projectively invariant, but a~
balanced quotient of brackets can be.  Indeed, under a~change of lifts
$\widetilde p_i\mapsto s_i\widetilde p_i$,
$$
c_{ij}\mapsto\frac{s_j}{s_i}c_{ij}.
$$
Consequently the~closed-cycle products
$m_{ij}=c_{ij}c_{ji}$,
$\Gamma_{ijk}=c_{ij}c_{jk}c_{ki}$
are independent of all choices of lifts and are projective invariants.
Explicitly,
$$
m_{ij}
=
\frac{
[i-1,j,i+1]\,[j-1,i,j+1]
}{
[i-1,i,i+1]\,[j-1,j,j+1]
}
\quad\text{and}\quad
\Gamma_{ijk}
=
\frac{
[i-1,j,i+1]\,
[j-1,k,j+1]\,
[k-1,i,k+1]
}{
[i-1,i,i+1]\,
[j-1,j,j+1]\,
[k-1,k,k+1]
}.
$$
These are balanced bracket quotients; ordinary cross-ratios of
collinear points admit the~same kind of bracket representation.  In the~
standard corner coordinates for the~pentagram map, they become rational
functions of the~usual cross-ratio invariants \cite{OST}.
\end{remark}

For the~rest of the~paper, whenever $n=5$ or $n=6$, we use the~notation
from \eqref{eq:glick-projective-generator} and write
$$
G_P=L_P-3I.
$$
The~eigenpoints of $G_P$ and $L_P$ are the~same, although their
eigenvalues are shifted.  If $L_P v_i=\lambda_i v_i,$ then $G_P v_i=(\lambda_i-3)v_i.$
Thus the~relevant ordering is by $|\mu_i|=|\lambda_i-3|,$
not necessarily by $|\lambda_i|$.

\begin{definition}[Dominant real spectral line]
\label{def:dominant-real-spectral-line}
Let $G$ be a~real projectivity, represented by a~real invertible
linear map on $\R^3$.  A~projective line $\ell_G\subset\RP^2$ is
called a~\emph{dominant real spectral line} of $G$ if it is the~
projectivization $\ell_G=\mathbb P(E)$
of a~real two-dimensional $G$-invariant subspace
$E\subset\R^3$ which dominates the~complementary spectral direction.
In this paper we use the~following two cases.

\begin{enumerate}[label=\textup{(\roman*)}]
\item $G$ has three real eigenpoints $X_1,X_2,X_3$ with eigenvalues $|\mu_1|>|\mu_2|>|\mu_3|.$
Then
$$
E=\Span\{v_1,v_2\},
\qquad
\ell_G=X_1X_2,
$$
where $X_i=[v_i]$.

\item $G$ has a~spectrally dominant complex-conjugate pair
$\mu,\overline{\mu}$ and one real eigenvalue $\nu$, with $|\mu|>|\nu|.$
If $v\in\mathbb C^3$ is an~eigenvector for $\mu$, then
$$
E=\Span_{\R}\{\operatorname{Re}v,\operatorname{Im}v\},
\qquad
\ell_G=\mathbb P(E).
$$
This is a~real projective line, although it is not determined by two
real eigenpoints.
\end{enumerate}

If $\ell_G$ is not the~line at infinity, let $u$ denote the~direction
of its affine part.  Then, for a~polygon $P$, we
write
$$
\mathcal L(P,G)=C(P)+\R u.
$$
For pentagons and hexagons, where $G=L_P-3I$, this line will also be
called the~eigenline.
\end{definition}

In the~real-eigenvalue case, the~line is the~line through the~two
leading real eigenpoints.  In the~complex case, the~word ``dominant''
is essential: the~complex pair must have larger modulus than the~
remaining real eigenvalue.  Only then does the~real invariant plane
spanned by the~real and imaginary parts of a~complex eigenvector define
the~leading projective line.  Theorem~\ref{thm:spectral-flattening} says that
area-normalization magnifies such a~dominant real projective line
whenever the~centered two-scale asymptotic is nondegenerate.

\section{Projective integrability versus Euclidean normalization}

\noindent The~ordinary pentagram map is usually studied on projective
moduli, where its integrability has geometric, algebraic-geometric,
cluster, network, and refactorization descriptions
\cite{GekhtmanShapiroTabachnikovVainshtein,GlickYPatterns,
IzosimovRefactorization,OST,OSTClosed,Soloviev,Weinreich}.  These formulations concern the~projective
dynamics; the~normalization studied here asks a~different, Euclidean
question.

Our area-normalized map does not change the~projective class of the~
ordinary pentagram iterate.  Indeed, the~passage from $\T(P)$ to
$\That(P)$ is an~affine homothety composed with a~translation, hence a~
projective transformation.  By projective naturality of $\T$, it follows
inductively that
$[\That^k(P)]=[\T^k(P)]$
for every $k$ for which the~iterates are defined, where $[Q]$ denotes
the~projective-equivalence class of a~labeled polygon $Q$.
Consequently, the~two
sequences
$$
P,\T(P),\T^2(P),\ldots
\quad\text{and}\quad
P,\That(P),\That^2(P),\ldots
$$
define the~same trajectory in projective moduli space.  The~flattening
phenomenon studied in this paper is therefore not a~projective
convergence statement.  It is a~statement about a~particular Euclidean
choice of representative along a~projectively natural orbit.

This distinction is especially relevant for $n\geq7$.  In this range,
one should not expect a~generic orbit to be generated by powers of a~
single fixed projectivity; projectively periodic behavior, such as the~
cases considered later in this paper, occurs only on special families.
This is consistent with the~integrable-system picture, where a~suitable
power of the~pentagram map acts by translation on invariant tori or,
algebraically, on Abelian varieties.  In some cases the~map itself
permutes several connected components of a~common level set, producing
a~staircase-type dynamics between invariant tori
\cite{KhesinSolovievNonIntegrability,OST,OSTClosed,Soloviev,Weinreich}.
Thus the~experimental conjectures for $n\geq7$ below should be viewed
as statements about the~Euclidean gauge of a~projectively integrable
orbit, not as statements of projective convergence generated by one
fixed projectivity.

The~corner-coordinate formulas of
Ovsienko--Schwartz--Tabachnikov make the~exceptional sets in
Proposition~\ref{prop:generic-forward-domain} explicit: the~one-step
denominators are $1-x_i y_i$, and, after clearing denominators, every
fixed-time failure locus is algebraic \cite{OST}.  Glick's cluster
$Y$-pattern description and the~directed-network realization provide
complementary algebraic descriptions of the~same rational dynamics
\cite{GekhtmanShapiroTabachnikovVainshtein,GlickYPatterns}.

This also separates our phenomenon from Schwartz's finite-time collapse
for special axis-aligned polygons \cite{SchwartzDiscreteMonodromy}.  In
that result the~ordinary pentagram dynamics reaches a~configuration
whose two parity classes lie on two lines after finitely many steps.  In
contrast, our flat limits are asymptotic, use all vertices of the~
area-normalized polygon, and depend on a~Euclidean normalization which
is not visible on projective moduli space.

Finally, conics enter the~integrable picture in a~natural way.
Schwartz and Tabachnikov proved that the~monodromy invariants satisfy
$E_k=O_k$ for polygons inscribed in a~conic
\cite{SchwartzTabachnikovInscribed}.  Related projective-geometric
results include Tabachnikov's extension of Kasner's commutation theorem
to Poncelet polygons \cite{TabachnikovKasner}.  For the~inverse
pentagram map, Izosimov proved that the~convex polygons whose whole
inverse orbit remains convex form a~codimension-two algebraic subset;
equivalently, this persistence occurs precisely when Glick's operator is
affine as a~projective transformation \cite{IzosimovSides}.

These results provide useful context.  The~general Poncelet input used
below is the~projective-periodicity theorem recalled in
Lemma~\ref{lem:projective-periodicity-input}; for the~strictly convex
subclass we additionally use the~standard Jacobi normal form of a~
Poncelet pencil to compute the~return spectrum explicitly in
Subsection~\ref{subsec:poncelet-spectrum}.  The~last result also
reinforces a~point important here: convexity is not intrinsic to the~
area-normalized mechanism.  Our general flattening theorem can apply to
nonconvex labeled pentagons and hexagons whenever its explicit
forward-orbit, spectral, and two-scale hypotheses hold.

\section{Centered two-scale flattening and projectively periodic polygons}
\label{sec:two-scale-criterion}

\subsection{The~Euclidean two-scale criterion}
\label{subsec:euclidean-two-scale}

\noindent We first isolate the~Euclidean mechanism responsible for
area-normalized flattening.  It is independent of the~special features
of pentagons, hexagons, or Poncelet polygons and does \emph{not} require
the~unnormalized polygons to converge to a~finite point.  The~only
essential input is an~anisotropic two-scale expansion after subtracting
the~barycenter, together with a~nonvanishing leading mixed-area term.

The~model to keep in mind is a~long thin parallelogram.  Its length in
one direction is of order $|\alpha_k|$, its width in a~transverse
direction is of order $|\beta_k|$, and
$$
\frac{|\beta_k|}{|\alpha_k|}\to0.
$$
Its unsigned area is then of order $|\alpha_k\beta_k|$.  Rescaling it
to fixed unsigned area multiplies the~longitudinal and transverse
scales, respectively,
by
$$
\frac{|\alpha_k|}{\sqrt{|\alpha_k\beta_k|}}
=
\sqrt{\frac{|\alpha_k|}{|\beta_k|}}
\to\infty,
\qquad
\frac{|\beta_k|}{\sqrt{|\alpha_k\beta_k|}}
=
\sqrt{\frac{|\beta_k|}{|\alpha_k|}}
\to0.
$$
The~following definition and proposition formalize this mechanism once
and for all.

\begin{definition}[Centered two-scale asymptotic]\label{def:two-scale}
Let $(Q_k)$ be a~sequence of labeled polygons with nonzero area, and
fix an~affine direction $u$.  We say that $(Q_k)$ has a~\emph{centered
two-scale asymptotic in the~direction $u$} if there exist a~vector $v$
transverse to $u$, nonzero real sequences $\alpha_k,\beta_k$, and real
numbers $a_i,b_i$ such that, on writing
\begin{align}
\label{eq:QkiminusCqk}
q_{k,i}-C(Q_k)&=x_{k,i}u+y_{k,i}v,
\end{align}
one has
$$
\frac{x_{k,i}}{\alpha_k}\to a_i,
\qquad
\frac{y_{k,i}}{\beta_k}\to b_i
\quad\text{for all }i,
\qquad
\frac{|\beta_k|}{|\alpha_k|}\to0.
$$
Since the~number of vertices is finite, the~coordinate convergences are
uniform in $i$.  Centering gives
$\sum_i a_i=\sum_i b_i=0$.

The~asymptotic is called \emph{area-nondegenerate} if its leading
mixed-area coefficient
\begin{align}
\label{eq:leading-mixed-area}
B&=
\frac12\sum_i
\det(a_i u+b_i v,\,a_{i+1}u+b_{i+1}v)
\end{align}
is nonzero.  Equivalently, $B$ is the~first nonzero coefficient of the~
oriented area at the~scale $\alpha_k\beta_k$.  In particular,
area-nondegeneracy implies that the~centered longitudinal coefficients
$a_i$ are not all equal.  When $Q_k=G^k(Q)$ for a~projectivity $G$, we
also say that $Q$ has the~corresponding centered two-scale asymptotic
under $G$.
\end{definition}

\begin{proposition}[Centered two-scale flattening criterion]\label{prop:centered-two-scale}
Assume that $(Q_k)$ has an~area-nondegenerate centered two-scale
asymptotic in the~direction $u$ in the~sense of
Definition~\ref{def:two-scale}.  Let $P_k$ be obtained from $Q_k$ by
translating the~barycenter to a~fixed point $C_0$ and rescaling to a~
fixed unsigned area $A_0>0$.  Then
$$
\max_i \dist(p_{k,i},\,C_0+\R u)
=
O\!\left(\sqrt{\frac{|\beta_k|}{|\alpha_k|}}\right)
\to0.
$$
Moreover, if
$D_u=\max_{p,q}|a_p-a_q|\,\|u\|>0$,
then
$$
\diam(P_k)
\sim
D_u\sqrt{\frac{A_0}{|B|}}
\sqrt{\frac{|\alpha_k|}{|\beta_k|}}
\to\infty.
$$
\end{proposition}

\begin{proof}
Let $A_0>0$ be the~prescribed unsigned area after normalization.  Then
$$
p_{k,i}
=
C_0+\rho_k\bigl(q_{k,i}-C(Q_k)\bigr),
\qquad
\rho_k=
\sqrt{\frac{A_0}{\Area(Q_k)}}.
$$
We first compute the~scale of $\Area(Q_k)$.  Oriented area is unchanged
by translation, so
$$
\sArea(Q_k)
=
\frac12\sum_i
\det\bigl(q_{k,i}-C(Q_k),\,q_{k,i+1}-C(Q_k)\bigr).
$$
Using \eqref{eq:QkiminusCqk} and the~identities
$\det(u,u)=\det(v,v)=0$,
$\det(v,u)=-\det(u,v)$,
we get
$$
\det(x_{k,i}u+y_{k,i}v,\,x_{k,i+1}u+y_{k,i+1}v)
=
(x_{k,i}y_{k,i+1}-y_{k,i}x_{k,i+1})\det(u,v).
$$
Equivalently, only the~mixed longitudinal--transverse terms contribute
to the~area.  Since
$x_{k,i}=\alpha_k(a_i+o(1))$,
$y_{k,i}=\beta_k(b_i+o(1))$
uniformly in $i$, it follows that $
\sArea(Q_k)
=
\alpha_k\beta_k B+o(|\alpha_k\beta_k|).$
Because the~coefficient $B$ in \eqref{eq:leading-mixed-area} is nonzero,
this gives
$$
\Area(Q_k)
=
|\alpha_k\beta_k|\,|B|+o(|\alpha_k\beta_k|)
\quad\text{and hence}\quad
\rho_k
\sim
\sqrt{\frac{A_0}{|B|}}\,
\frac1{\sqrt{|\alpha_k\beta_k|}}.
$$

We now estimate the~distance to the~line $C_0+\R u$.  The~component
$\rho_k x_{k,i}u$ lies on this line, while the~transverse component is
$\rho_k y_{k,i}v$.  Since $u$ and $v$ are linearly independent,
$\dist(v,\R u)$ is a~fixed positive constant.  Therefore $\dist(p_{k,i},C_0+\R u)
\leq
\rho_k |y_{k,i}|\,\dist(v,\R u).$
Using $y_{k,i}=O(\beta_k)$ uniformly in $i$, we obtain
$$
\max_i\dist(p_{k,i},C_0+\R u)
=
O\!\left(
\rho_k|\beta_k|
\right)
=
O\!\left(
\sqrt{\frac{|\beta_k|}{|\alpha_k|}}
\right)
\to0.
$$
This proves transverse convergence to $C_0+\R u$.

For the~diameter, uniformly over the~finitely many pairs of vertices,
$$
\frac{p_{k,p}-p_{k,q}}{\rho_k\alpha_k}
=
(a_p-a_q+o(1))u
+
\frac{\beta_k}{\alpha_k}(b_p-b_q+o(1))v
\rightarrow
(a_p-a_q)u.
$$
Area-nondegeneracy implies that the~numbers $a_i$ are not all equal,
so $D_u=\max_{p,q}|a_p-a_q|\,\|u\|$ is positive.  Taking the~maximum
over $p,q$ therefore gives
$$
\frac{\diam(P_k)}{|\rho_k\alpha_k|}\to D_u.
$$
Together with
$$
|\rho_k\alpha_k|
\sim
\sqrt{\frac{A_0}{|B|}}
\sqrt{\frac{|\alpha_k|}{|\beta_k|}},
$$
this proves the~stated diameter asymptotic and, in particular,
$\diam(P_k)\to\infty$.
\end{proof}

Thus flattening is not essentially a~statement about convergence of the~
ordinary pentagram iterates to a~finite point.  It is a~statement about
anisotropy after centering.  The~ordinary iterates may converge to a~
point, escape in the~affine chart, or oscillate in the~longitudinal
coordinate; the~area-normalized iterates still flatten whenever the~
centered two-scale condition holds.

\subsection{Projective periodicity and spectral flattening}
\label{subsec:projective-periodic-flattening}

\noindent We now combine Proposition~\ref{prop:centered-two-scale} with the~
case in which a~positive iterate of the~pentagram map is realized by a~
single projectivity.  The~three classes most important here are
pentagons, hexagons, and Poncelet polygons.  For these standard classes
the~required exponent is $m=1$ or $m=2$, but the~conditional statement
is naturally formulated for an~arbitrary integer $m\geq1$.
Examples~\ref{ex:period-three-heptagon} and
\ref{ex:period-four-octagon} provide exact heptagonal and octagonal
examples with minimal return times $m=3$ and $m=4$, respectively.  For a~
recent general study of collapse in projectively periodic polygonal
dynamics and of the~corresponding infinitesimal-monodromy formalism, see
\cite{StieglerCollapse}.  That
work concerns collapse itself, whereas our focus is the~Euclidean shape
remaining after area normalization.

For a~labeled polygon $P$, its \emph{minimal projective return time}
is the~least integer $m\geq1$ for which there exist a~projectivity $G$
and a~cyclic shift $\Sigma_s$ such that
$\T^m(P)=\Sigma_s(G(P))$.  The~term is used only when all objects in
this identity are defined and such an~integer exists.

\begin{definition}[Poncelet polygon]\label{def:poncelet-polygon}
A~labeled polygon $P=(p_1,\ldots,p_n)$ in the~projective plane is called
a~Poncelet polygon if there exist two nondegenerate conics
$\mathcal C_{\mathrm{out}}$ and $\mathcal C_{\mathrm{in}}$ such that all
vertices $p_i$ lie on $\mathcal C_{\mathrm{out}}$ and every side
$p_ip_{i+1}$ is tangent to $\mathcal C_{\mathrm{in}}$.  Equivalently,
$P$ is a~closed orbit of the~Poncelet construction for the~pair
$(\mathcal C_{\mathrm{out}},\mathcal C_{\mathrm{in}})$.
For a~closed orbit, we write $q/n\in(0,1/2)$ for its primitive rotation
number after choosing an~orientation, where $\gcd(q,n)=1$.
\end{definition}

\begin{definition}[Projectively regular Poncelet polygon]
\label{def:projectively-regular}
A~Poncelet $n$-gon of primitive rotation number $q/n$ is called
\emph{projectively regular} if, with its cyclic labeling, it is
projectively equivalent to the~regular star polygon
$$
\left(
\left(
\cos\frac{2\constpi qj}{n},
\sin\frac{2\constpi qj}{n}
\right)
\right)_{j=0}^{n-1}.
$$
For a~strictly convex Poncelet polygon one has $q=1$ after choosing the~
positive cyclic orientation, so this means projective equivalence to
the~ordinary regular $n$-gon.
\end{definition}

\begin{lemma}[Projective-periodicity input]\label{lem:projective-periodicity-input}
Let $P$ be a~labeled polygon, and assume in each part that every
displayed pentagram iterate and Glick operator is defined.  In
parts~\textup{(i)}
and~\textup{(ii)}, assume additionally that $L_P-3I$ is invertible.
\begin{enumerate}[label=\textup{(\roman*)}]
\item If $P$ is a~pentagon, then
$\T(P)=\Sigma_{-1}(G(P))$, where
$G=L_P-3I$.
\item If $P$ is a~hexagon, then
$\T^2(P)=G(P)$, where
$G=L_P-3I$.
\item If $P$ is a~Poncelet $n$-gon in the~sense of
Definition~\ref{def:poncelet-polygon}, with $n\geq5$, then there is a~
forward Darboux--Schwartz projectivity $G$ and a~cyclic relabelling
$\Sigma_s$ such that $\T^2(P)=\Sigma_s(G(P)).$
Here $s$ depends on the~convention used to label the~vertices.  Moreover,
this projectivity commutes with Glick's operator: $GL_P=L_PG.$
Consequently, if $L_P$ has simple spectrum, then $G$ and $L_P$ have
common eigenpoints.
\end{enumerate}
\end{lemma}

\begin{proof}
The~pentagonal projectivity already appears in Schwartz's original
analysis \cite[Section~2]{SchwartzPentagram}.  Glick identifies
$L_P-3I$ as the~canonical projectivity realizing the~relevant relation
for both pentagons and hexagons \cite{GlickLimit}.  The~third
statement is Schwartz's Darboux--Schwartz theorem for Poncelet polygons
\cite{SchwartzPoncelet}; Schwartz states a~projectivity carrying
$\T^2(P)$ back to $P$, and here $G$ denotes its inverse.  The~cyclic
shift in the~Poncelet statement records the~chosen labeling; in the~
positive Jacobi convention of
Subsection~\ref{subsec:poncelet-spectrum}, it is $\Sigma_3$.

It remains only to prove the~commutation statement in the~Poncelet
case.  Using conservation, the~equality $L_{\Sigma_s(R)}=L_R$ which follows
immediately from the~cyclic sum in \eqref{eq:glick-operator}, and
projective covariance of Glick's operator, we obtain
$$
L_P
=
L_{\T^2(P)}
=
L_{\Sigma_s(G(P))}
=
L_{G(P)}
=
GL_PG^{-1}.
$$
Thus $GL_P=L_PG.$
If $L_P$ has simple spectrum, any projectivity commuting with it
preserves its one-dimensional eigenspaces, so the~eigenpoints coincide.
\end{proof}

The~next elementary lemma uses only the~definition of the~pentagram map:
cyclic relabelling commutes with taking consecutive short diagonals.

\begin{lemma}[Cyclic relabelling]\label{lem:cyclic-relabel}
Let $\Sigma_s$ denote the~cyclic relabelling
$(\Sigma_s(P))_i=p_{i+s}$, with indices taken modulo the~number of
vertices.  Suppose that, for some integer $m\geq1$ and some projectivity
$G$, one has $\T^m(P)=\Sigma_s(G(P)).$
Then, for every residue class $r=0,\ldots,m-1$ and every $q\geq0$,
$$
\T^{qm+r}(P)=\Sigma_{qs}\bigl(G^q(\T^r(P))\bigr).
$$
In particular, every conclusion of the~spectral flattening and elliptic
oscillation results that is invariant under cyclic relabelling of the~
vertices remains valid with this shifted relation.
\end{lemma}

\begin{proof}
The~pentagram map commutes with cyclic relabelling and is projectively
natural.  Hence, if the~shifted relation holds for $P$, then for every
$r$ it also holds for $\T^r(P)$:
$$
\T^m(\T^r(P))
=
\T^r(\T^m(P))
=
\T^r(\Sigma_s(G(P)))
=
\Sigma_s(G(\T^r(P))).
$$
Iterating this identity gives the~displayed formula.  Since $\Sigma_s$
only changes the~starting label of the~same polygonal cycle, it preserves
the~barycenter, signed area, diameter, and the~Euclidean set traced by the~
vertices.  Therefore it does not affect the~normalized limiting shapes.
\end{proof}

The~next theorem is conditional.  It separates the~projective input from
the~Euclidean normalization mechanism already proved in
Proposition~\ref{prop:centered-two-scale}.  It is not a~statement about
arbitrary $n$-gons: one must first have a~relation
$\T^m(P)=\Sigma_s(G(P))$, with $s=0$ allowed, and then verify the~
area-nondegenerate centered two-scale hypothesis of
Definition~\ref{def:two-scale} in every residue class.  The~verified
sources of the~projective relation used in this paper are pentagons,
hexagons, and Poncelet polygons.  For general $n\geq7$, predictions based
only on $L_P$ are treated later as experimental or conjectural.

The~spectral assumption below singles out the~geometrically distinguished
direction $u$.  The~actual flattening mechanism is supplied by the~
centered two-scale hypothesis and Proposition~\ref{prop:centered-two-scale}.

\begin{theorem}[Spectral flattening theorem]\label{thm:spectral-flattening}
Let $P$ be a~labeled polygon with $\Area(P)>0$.  Suppose that the~
following assumptions hold.
\begin{enumerate}[label=\textup{(\roman*)}]
\item\label{ass:flattening-forward-domain}
All forward pentagram iterates of $P$ are defined and have nonzero
area:
$\Area(\T^k(P))\ne0$
for all $k\geq0$.
\item\label{ass:flattening-projective-periodicity}
There exist an~integer $m\geq1$, a~projectivity $G$, and a~cyclic
relabelling $\Sigma_s$ (with $s=0$ allowed) such that $\T^m(P)=\Sigma_s(G(P)).$
\item\label{ass:flattening-spectrum}
The~projectivity $G$ has a~dominant real spectral line
$\ell_G$ in the~sense of
Definition~\ref{def:dominant-real-spectral-line}.  Assume that
$\ell_G$ has a~finite affine direction $u$.
\item\label{ass:flattening-two-scale}
For every residue class $r=0,\ldots,m-1$, the~polygon $Q^{(r)}=\T^r(P)$
has an~area-nondegenerate centered two-scale asymptotic under $G$ in
the~direction $u$.
\end{enumerate}
Then the~area-normalized iterates
$
P_k=\That^k(P)
$
flatten to the~affine line
$
\mathcal L(P,G)=C(P)+\R u.
$
More precisely,
$$
\max_i\dist(p_{k,i},\mathcal L(P,G))\to0
\quad\text{and}\quad
\diam(P_k)\to\infty.
$$
\end{theorem}

\begin{proof}
For each residue class $r=0,\ldots,m-1$, put
$Q^{(r)}=\T^r(P)$ and 
$R_j^{(r)}=G^j(Q^{(r)})$.
By Lemma~\ref{lem:cyclic-relabel},
$$
\T^{jm+r}(P)
=
\Sigma_{js}\bigl(R_j^{(r)}\bigr).
$$
Cyclic relabelling preserves the~vertex set, barycenter, area, diameter,
and distances from affine lines.  Equation~
\eqref{eq:normalization-classical-iterates} therefore shows that the~
area-normalized subsequence $(P_{jm+r})_{j\geq0}$ is obtained from
$(R_j^{(r)})_{j\geq0}$ by translating its barycenter to $C(P)$ and
rescaling its area to the~fixed value $\Area(P)$.

By condition~\ref{ass:flattening-two-scale}, the~sequence
$(R_j^{(r)})_{j\geq0}$ has an~area-nondegenerate centered two-scale
asymptotic in the~direction $u$.  Proposition~
\ref{prop:centered-two-scale}, applied with $C_0=C(P)$ and
$A_0=\Area(P)$, gives
$$
\max_i\dist(p_{jm+r,i},\,C(P)+\R u)\to0
\quad\text{and}\quad
\diam(P_{jm+r})\to\infty.
$$
The~direction $u$ is common to all residue classes by conditions
\ref{ass:flattening-spectrum} and~\ref{ass:flattening-two-scale}.
Since there are only finitely many residue classes, the~same two
conclusions hold for the~full sequence $(P_k)$.  This proves the~claim.
\end{proof}

\begin{corollary}[Hyperbolic vertex asymptotics]
\label{cor:hyperbolic-vertex-asymptotics}
Assume the~hypotheses of Theorem~\ref{thm:spectral-flattening}.  Fix a~
residue class $r\in\{0,\ldots,m-1\}$ and undo the~cyclic relabelling by
putting
$\widetilde P_j^{(r)}
=
\Sigma_{-js}\bigl(\That^{jm+r}(P)\bigr)$.
Use the~coordinates $(X,Y)$ in the~basis $(u,v_r)$ supplied by the~
centered two-scale asymptotic of $Q^{(r)}=\T^r(P)$, and write
$$
\widetilde p_{j,i}^{(r)}-C(P)
=
X_{j,i}^{(r)}u+Y_{j,i}^{(r)}v_r.
$$
Along every subsequence on which
$\sigma_j^{(r)}=\operatorname{sgn}
\left(\alpha_j^{(r)}\beta_j^{(r)}\right)$ has a~constant value
$\sigma\in\{-1,1\}$, one has
$$
X_{j,i}^{(r)}Y_{j,i}^{(r)}
\rightarrow
\sigma
\frac{\Area(P)}{|B_r|}
a_i^{(r)}b_i^{(r)}.
$$
Consequently, whenever $a_i^{(r)}b_i^{(r)}\ne0$, that vertex
subsequence is asymptotic to a~hyperbola whose asymptotes are
$$
C(P)+\R u
\qquad\text{and}\qquad
C(P)+\R v_r.
$$
If one of the~two leading coefficients vanishes, the~limiting product
is zero; the~first-order conic then degenerates to the~union of the~two
coordinate asymptotes, and no nondegenerate hyperbola is asserted for
that vertex.

In particular, when the~two scales come from a~real simple spectrum,
the~sign sequence is periodic with period at most two.  Thus each
residue class splits into at most two asymptotic hyperbolic families.
In an~affine eigenchart for the~return projectivity, after undoing the~
cyclic relabelling, the~nondegenerate trajectories lie on the~
corresponding hyperbolas exactly.
\end{corollary}

\begin{proof}
By Definition~\ref{def:two-scale}, the~centered coordinates of
$R_j^{(r)}=G^j(Q^{(r)})$ have the~two-scale form
$$
x_{j,i}^{(r)}
=
\alpha_j^{(r)}a_i^{(r)}+o(\alpha_j^{(r)}),
\qquad
y_{j,i}^{(r)}
=
\beta_j^{(r)}b_i^{(r)}+o(\beta_j^{(r)}).
$$
The~normalization formula and the~area computation in the~proof of
Proposition~\ref{prop:centered-two-scale} therefore give
$$
X_{j,i}^{(r)}=\rho_j^{(r)}x_{j,i}^{(r)},
\qquad
Y_{j,i}^{(r)}=\rho_j^{(r)}y_{j,i}^{(r)},
\qquad
\bigl(\rho_j^{(r)}\bigr)^2
\sim
\frac{\Area(P)}{|B_r|\,|\alpha_j^{(r)}\beta_j^{(r)}|}.
$$
Multiplication gives the~stated limit.  The~description of the~
asymptotes is immediate from the~coordinate axes.  For a~real simple
spectrum the~two scales are signed geometric sequences, so the~sign of
their product is constant or alternating.  In an~affine eigenchart the~
return projectivity is diagonal, the~area scaling is exact, and hence so is the~
hyperbola equation.
\end{proof}

\begin{remark}
Condition~\ref{ass:flattening-two-scale} of Theorem~\ref{thm:spectral-flattening} is precisely
this definition applied to every residue class of a~projectively
periodic orbit.  It is a~nondegeneracy and transversality assumption,
not an~additional projective-periodicity statement.  Once a~relation
$\T^m(P)=\Sigma_s(G(P))$ is known, condition~\ref{ass:flattening-two-scale} asks that the~centered polygons
$G^j(\T^r(P))$ really have one leading scale along the~predicted line
and one smaller transverse scale, and that these two leading components
produce nonzero area to first order.

In the~real-eigenvalue case, when the~dominant eigenpoint of $G$ is
finite in the~chosen affine chart, such a~centered two-scale asymptotic
is the~usual local spectral expansion near this attracting point,
provided the~leading area coefficient is nonzero.  Thus for a~fixed
projectivity with a~simple spectral gap this condition is expected to
hold outside an~algebraic exceptional set: it can fail only if the~
polygon has no leading component in one of the~relevant eigendirections,
or if the~first mixed area coefficient $B_Q$ vanishes.  The~definition
above is more flexible in two ways.  It allows the~dominant real
eigenpoint to lie on the~line at infinity, and it also allows a~
spectrally dominant complex-conjugate pair.  In the~latter case there is
no real attracting eigenpoint; instead, the~ordinary iterates may rotate
inside the~real invariant plane $E$, while the~component transverse to
$\mathbb P(E)$ is dominated.

Whenever a~dominant non-real pair nevertheless satisfies the~fixed
centered two-scale hypothesis, Corollary~\ref{cor:hyperbolic-vertex-asymptotics}
applies without change.  A~genuinely rotating leading block usually
does not have fixed coefficients $a_i,b_i$.  Passing to subsequences on
which the~rotation phase converges then gives phase-dependent
asymptotic hyperbolas, provided the~corresponding phase-dependent
leading mixed area coefficient is nonzero.  A~rational rotation angle
produces only finitely many such families, whereas an~irrational angle
may produce a~one-parameter family; no single hyperbola is expected
without additional phase locking.

It is important not to confuse this case with the~opposite spectral
configuration in which one real eigenvalue has largest modulus and the~
remaining two eigenvalues form a~non-real conjugate pair.  Then the~
ordinary iterates may converge projectively to the~real eigenpoint, but
the~centered first-order term lives in a~two-dimensional rotating
subspace of equal modulus.  Generically this does not select a~fixed
real line after area normalization.
\end{remark}

\begin{definition}[Principal line]
Let $P=(p_1,\ldots,p_n)$
be a~labeled polygon in $\R^2$, and let
$C(P)=\frac1n\sum_{i=1}^n p_i$
be its barycenter.  We define the~covariance matrix of $P$ by
$$
    \mathsf{Cov}(P)
    =
    \frac1n\sum_{i=1}^n
    (p_i-C(P))(p_i-C(P))^{\top}.
$$
Assume that the~largest eigenvalue of $\mathsf{Cov}(P)$ is simple, and
let $u_P$ be a~corresponding unit eigenvector.  The~\emph{principal
line} of $P$ is the~affine line
$$
    \mathfrak p(P)=C(P)+\R u_P.
$$
This definition is independent of the~choice of the~sign of $u_P$.
Equivalently, $\mathfrak p(P)$ is the~unique line through the~barycenter
which minimizes the~mean squared distance from the~vertices of $P$.
\end{definition}

For an~affine line $\ell$, we denote its direction in $\RP^1$ by
$\operatorname{dir}\ell$.

\begin{theorem}[Convergence of principal lines]
\label{thm:principal-lines-convergence}
Assume that the~hypotheses of Theorem~\ref{thm:spectral-flattening} hold.  Let $P_k$ denote the~
area-normalized pentagram iterates of $P$.

Then, for all sufficiently large $k$, the~principal line
$\mathfrak p(P_k)$ is well defined, and
$\operatorname{dir}\mathfrak p(P_k)\to\R u$
in $\RP^1$.
Equivalently, the~current principal lines converge in direction to
the~predicted spectral line.
\end{theorem}

\begin{proof}
We first note that the~principal direction is not affected by the~
area normalization.  Indeed, translating a~polygon does not change
its covariance matrix, while applying a~homothety with ratio $\rho$
multiplies the~covariance matrix by $\rho^2$.  Therefore the~
principal eigenspaces of the~covariance matrix are the~same before
and after area normalization.

Fix a~residue class $r\in\{0,\ldots,m-1\}$ and put
$Q^{(r)}=\T^r(P)$.  By
condition~\ref{ass:flattening-projective-periodicity} of
Theorem~\ref{thm:spectral-flattening} and
Lemma~\ref{lem:cyclic-relabel}, we have
$$
    \T^{jm+r}(P)
    =
    \Sigma_{js}\bigl(G^j(Q^{(r)})\bigr).
$$
The~cyclic relabelling does not change the~covariance matrix.  Hence it
is enough to study the~principal direction of the~ordinary projective
sequence $G^j(Q^{(r)})$.

For simplicity of notation, we suppress the~superscript $r$.  By
condition~\ref{ass:flattening-two-scale}, the~centered vertices satisfy
$$
    z_{j,i}:=G^j(Q)_i-C(G^j(Q))
    =
    \alpha_j a_i u
    +
    \beta_j b_i v
    +
    \varepsilon_{j,i},
$$
where $v$ is transverse to $u$,
$$
    \sum_i a_i=\sum_i b_i=0,
    \qquad
    \frac{|\beta_j|}{|\alpha_j|}
    \to  0,
    \qquad
    \max_i\|\varepsilon_{j,i}\|
    =
    o(|\alpha_j|).
$$
Let $\Sigma_j$ be the~covariance matrix of $G^j(Q)$.  Then
$$
    \Sigma_j
    =
    \frac1n
    \sum_{i=1}^n
    z_{j,i}z_{j,i}^{\top}.
$$
Using the~above expansion, we obtain
\begin{align}
\label{eq:SigmaiFormulaalfaj}
    \Sigma_j
    & =
    \alpha_j^2 A\, uu^{\top}
    +
    \alpha_j\beta_j B\,(uv^{\top}+vu^{\top})
    +
    \beta_j^2 D\, vv^{\top}
    +
    o(\alpha_j^2),
\end{align}
where
$$
    A
    =
    \frac1n
    \sum_{i=1}^n a_i^2,
    \qquad
    B
    =
    \frac1n
    \sum_{i=1}^n
    a_i b_i,
    \quad\text{and}\quad
    D
    =
    \frac1n
    \sum_{i=1}^n b_i^2.
$$
Here $A>0$ follows from area-nondegeneracy: if all centered
longitudinal coefficients vanished, then the~leading mixed-area
coefficient would vanish as well.  Since $|\beta_j/\alpha_j|\to0$,
dividing equation \eqref{eq:SigmaiFormulaalfaj} by $\alpha_j^2$ gives
$$
    \frac{1}{\alpha_j^2}\Sigma_j
    \to
    A\,uu^{\top}.
$$
The~limiting matrix $A\,uu^{\top}$ is symmetric and has a~simple largest
eigenspace equal to $\R u$.  By continuity of simple
eigenspaces for symmetric matrices, the~principal eigenspace of
$\Sigma_j$ converges to $\R u$.

Thus, in each residue class,
$$
\operatorname{dir}\mathfrak p(P_{jm+r})
\to
\R u.
$$
Since the~same dominant direction $u$ occurs in all residue classes
in Theorem~\ref{thm:spectral-flattening}, the~whole sequence
$\operatorname{dir}\mathfrak p(P_k)$ converges to $\R u$.
This proves the~claim.
\end{proof}

\newcommand{\figpanel}[2]{%
    \raisebox{-0.5\height}{%
        \includegraphics[width=#1\textwidth]{#2}%
    }%
}
\begin{figure}[htbp]
    \centering

    \figpanel{0.24}{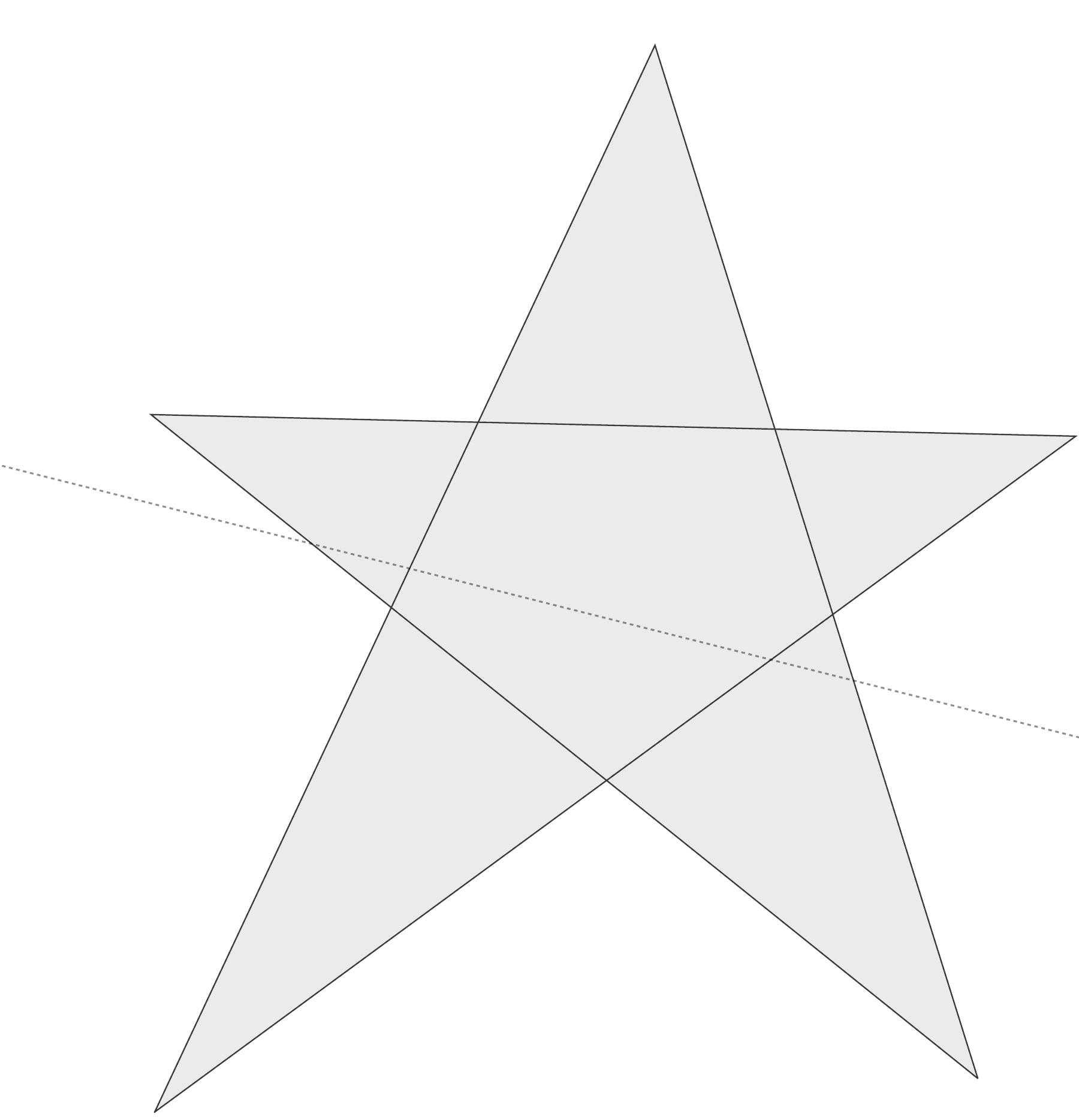}
    \hfill
    \figpanel{0.24}{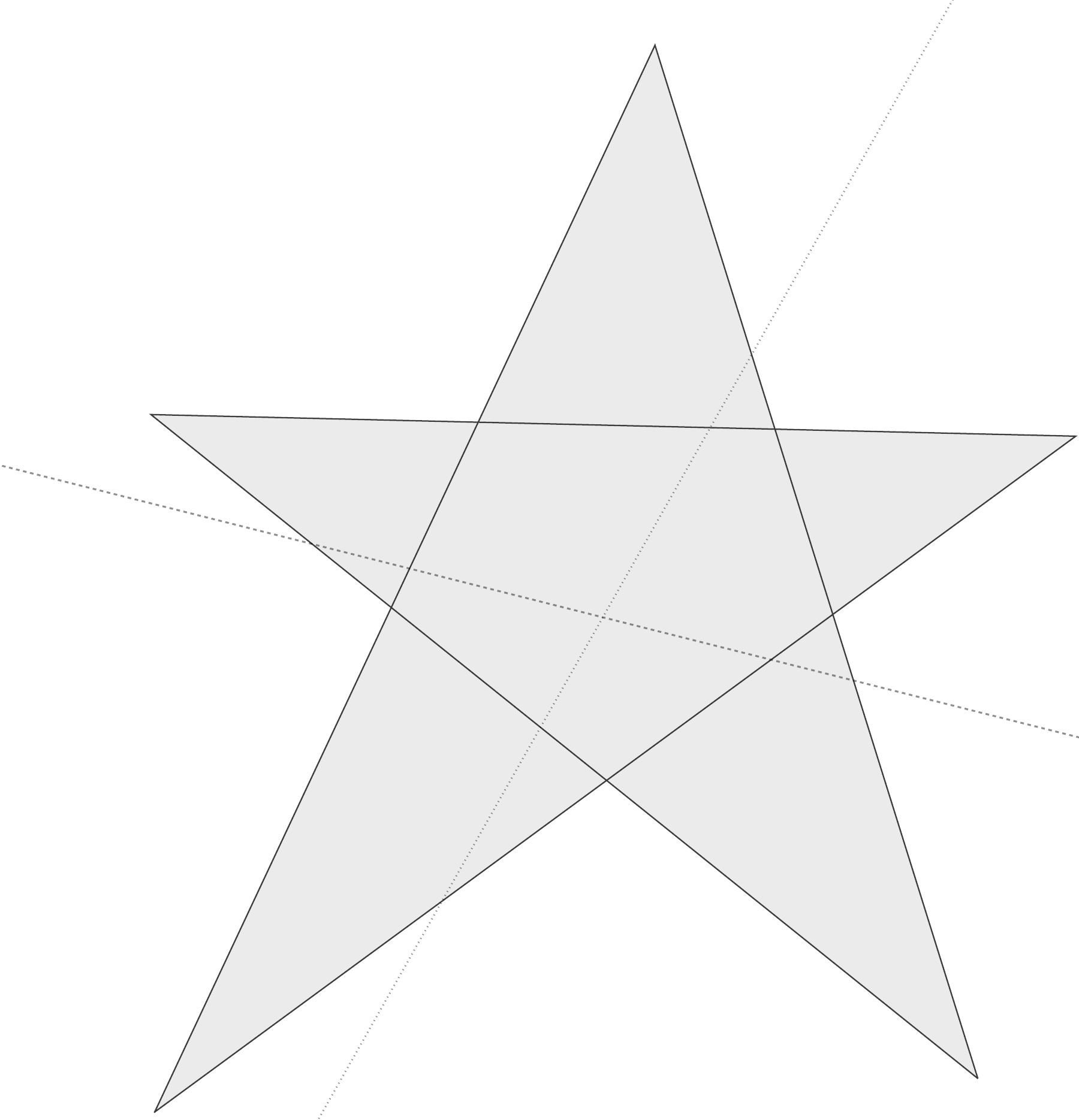}
    \hfill
    \figpanel{0.24}{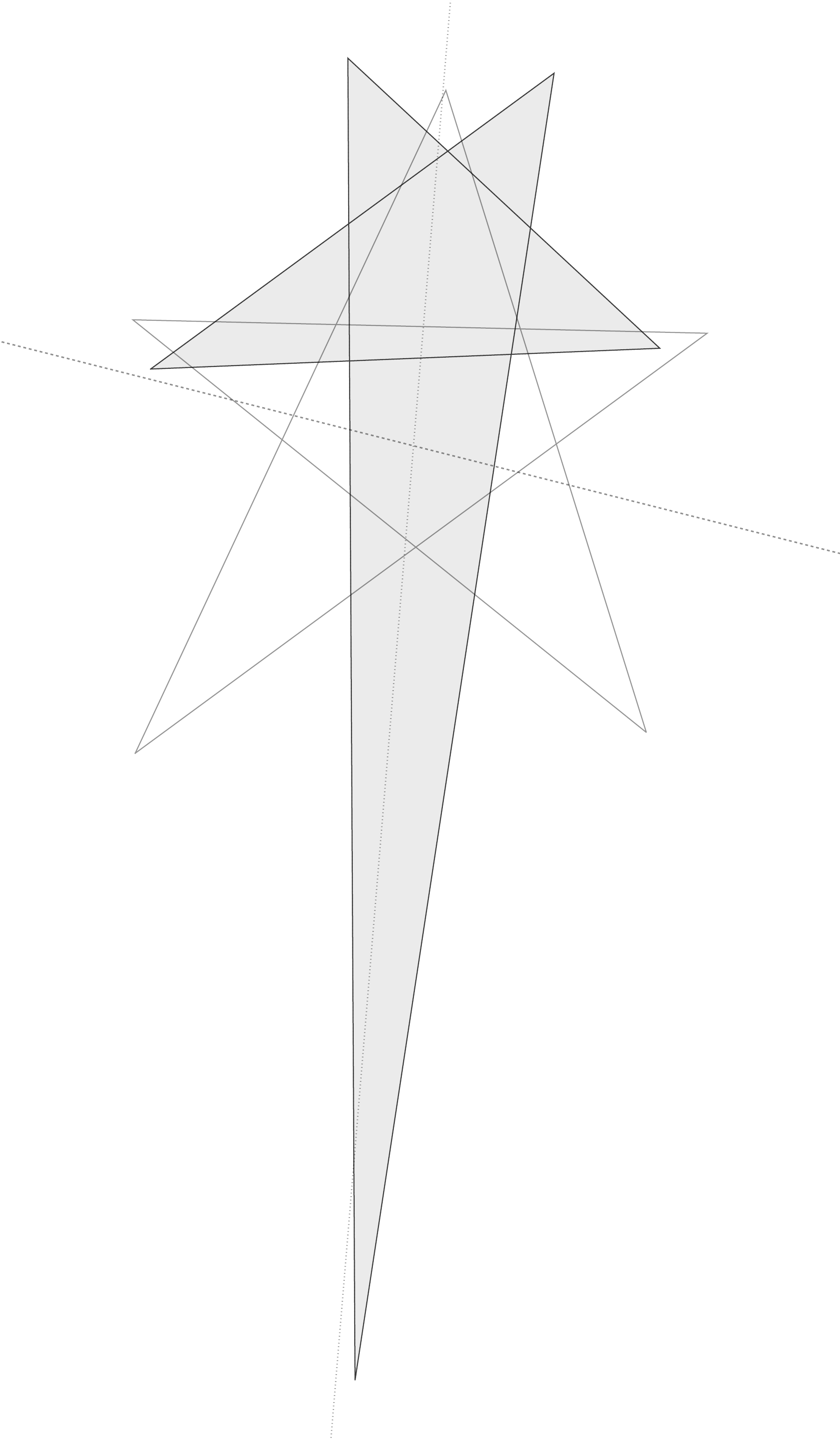}
    \hfill
    \figpanel{0.24}{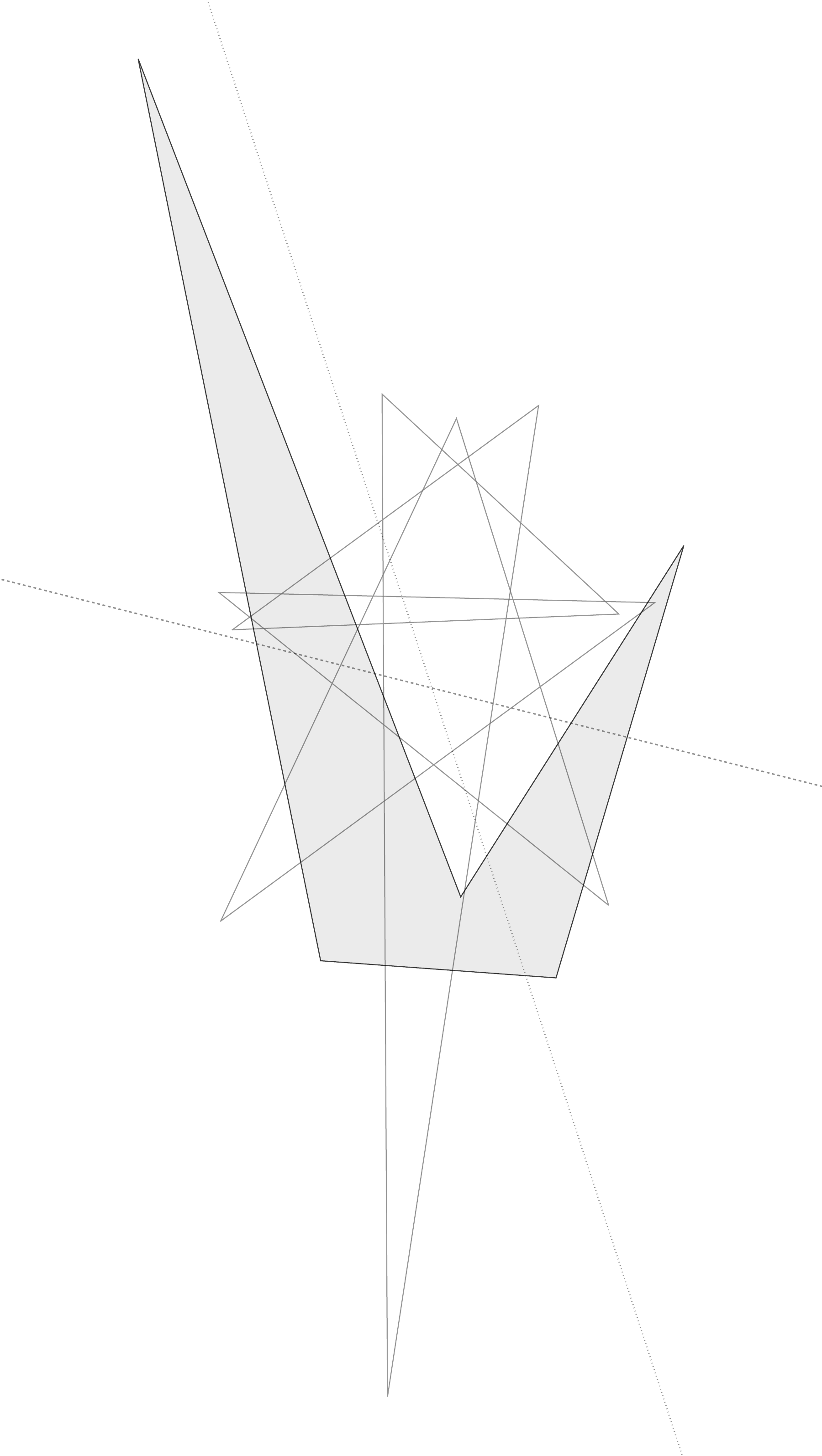}

    \vspace{0.25cm}

    \figpanel{0.19}{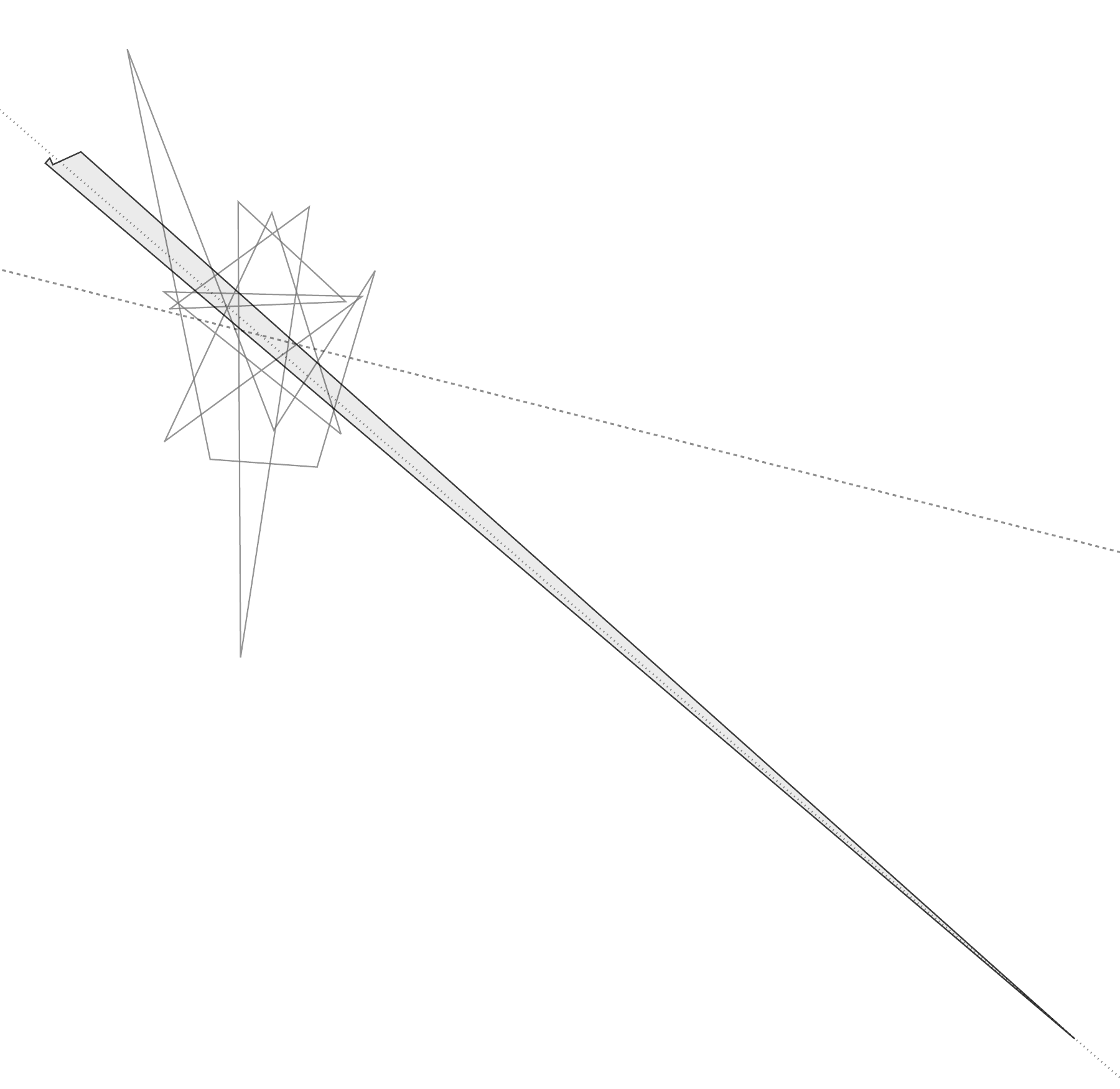}
    \hfill
    \figpanel{0.19}{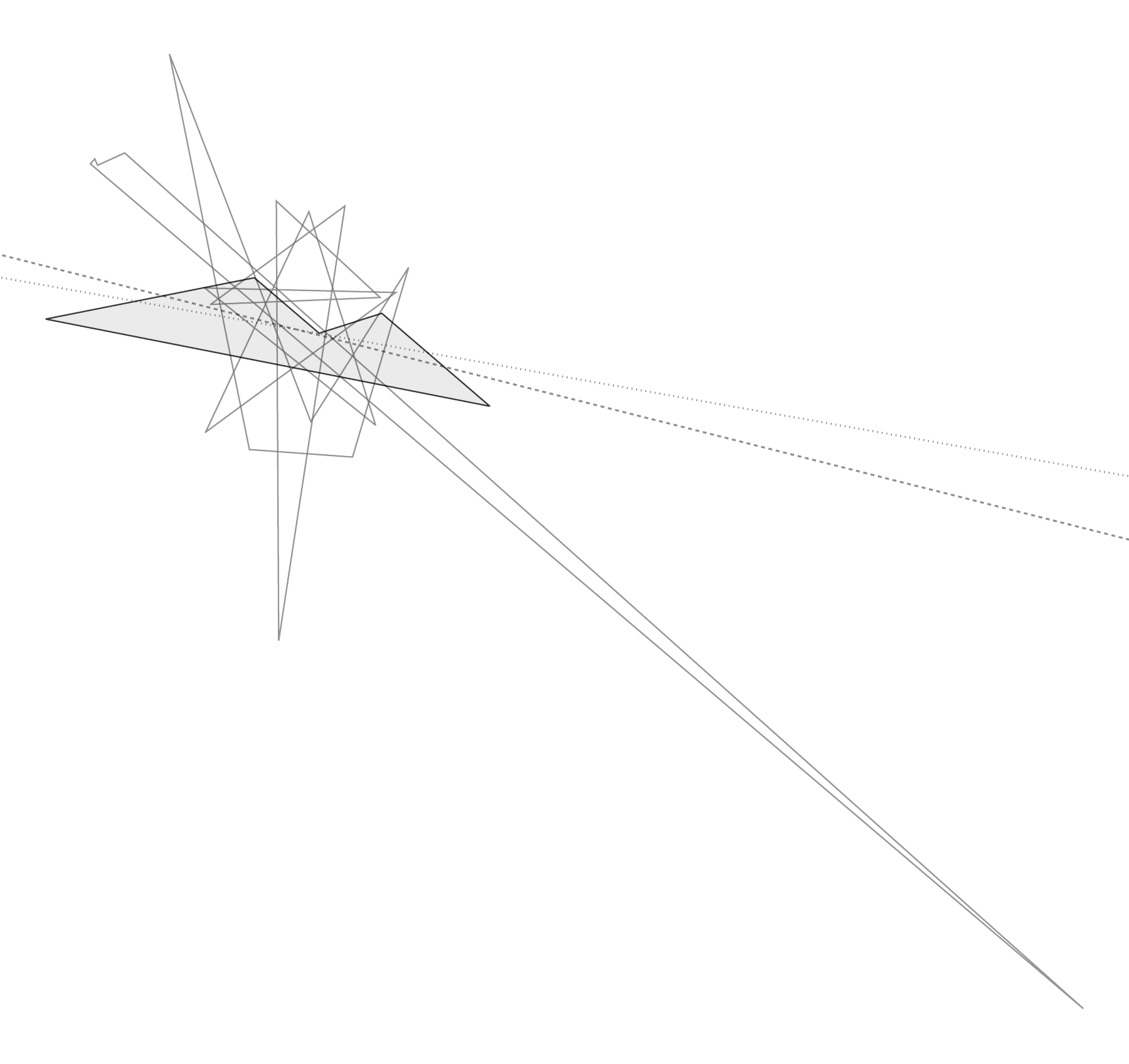}
    \hfill
    \figpanel{0.19}{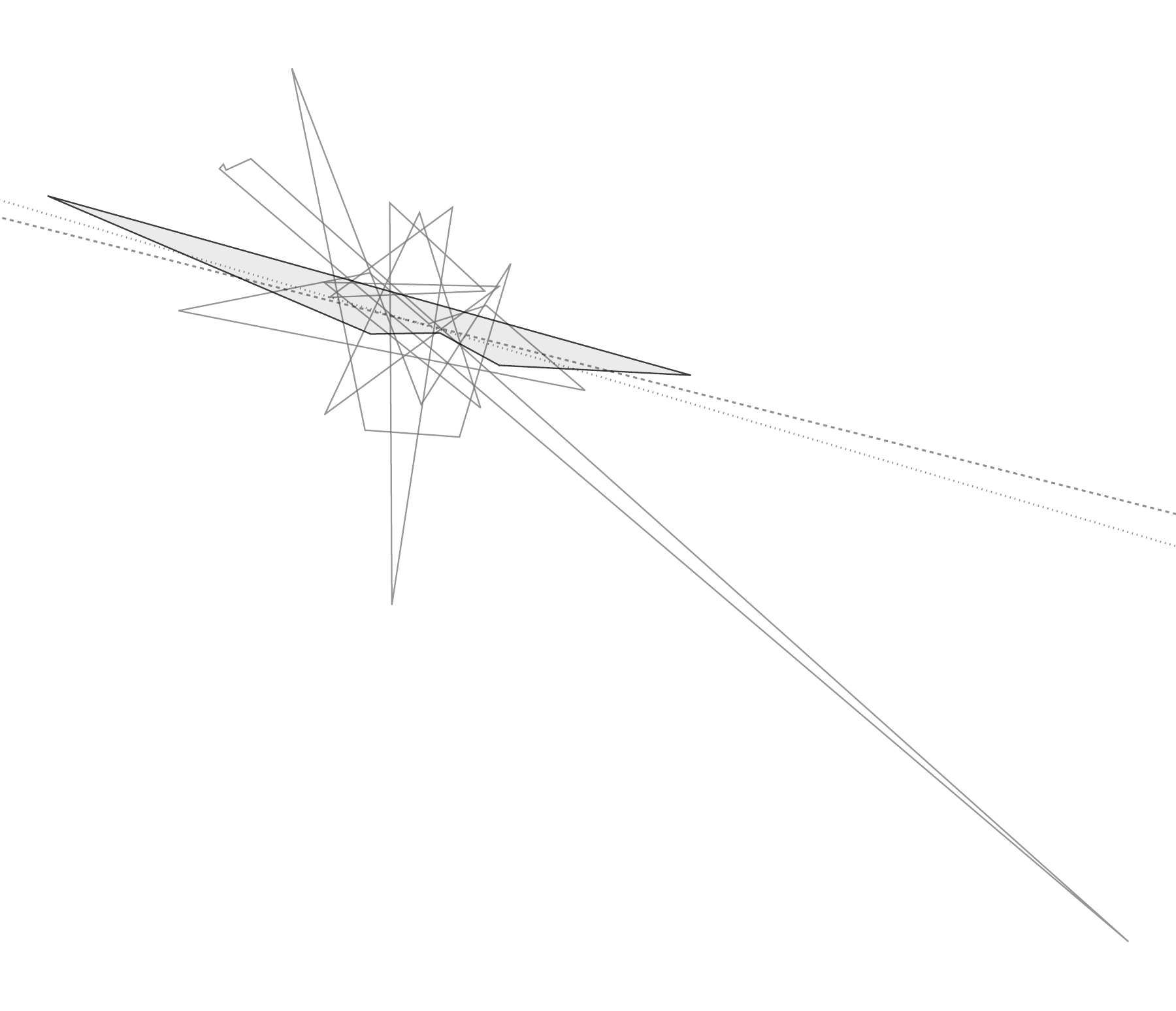}
    \hfill
    \figpanel{0.19}{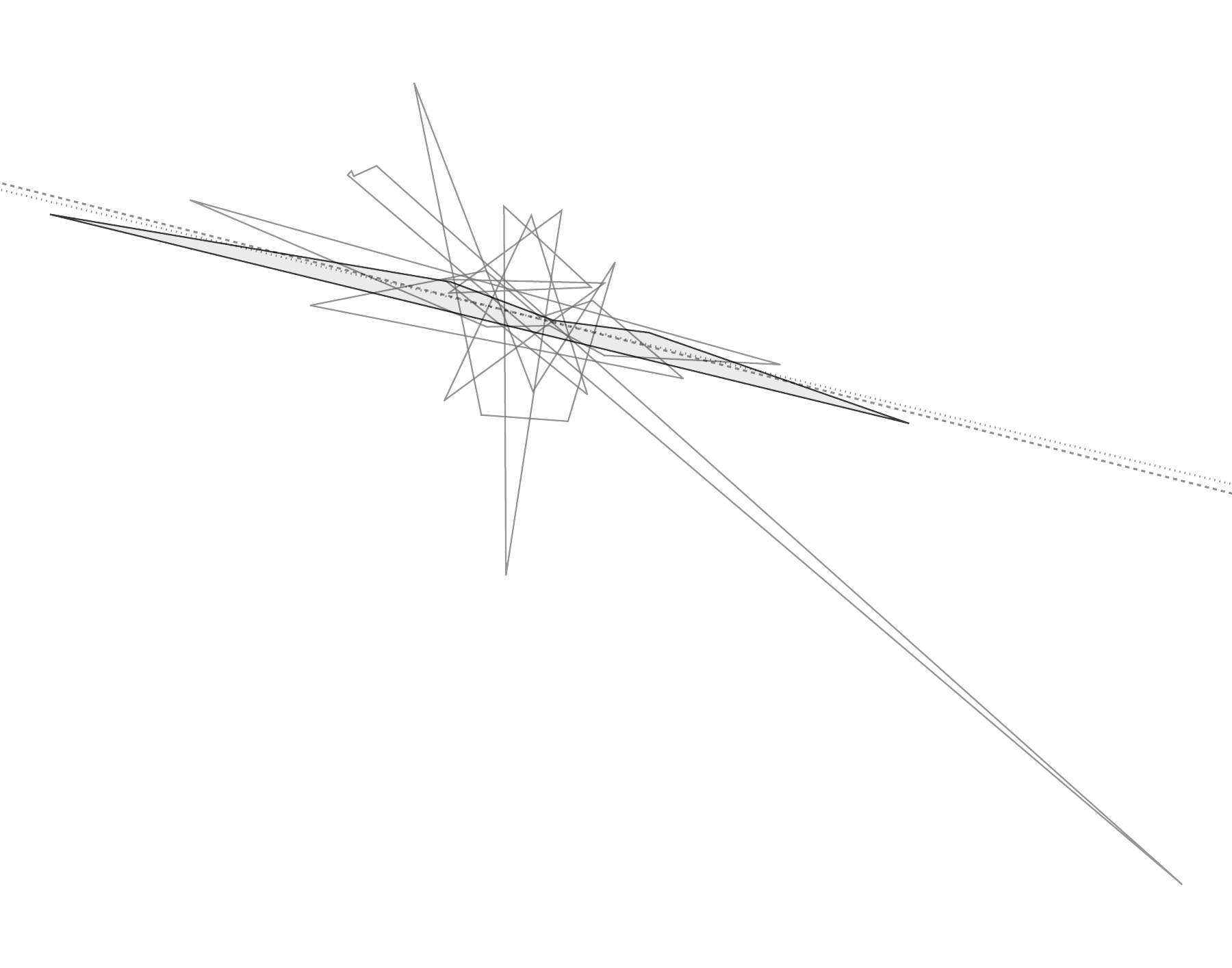}
    \hfill
    \figpanel{0.19}{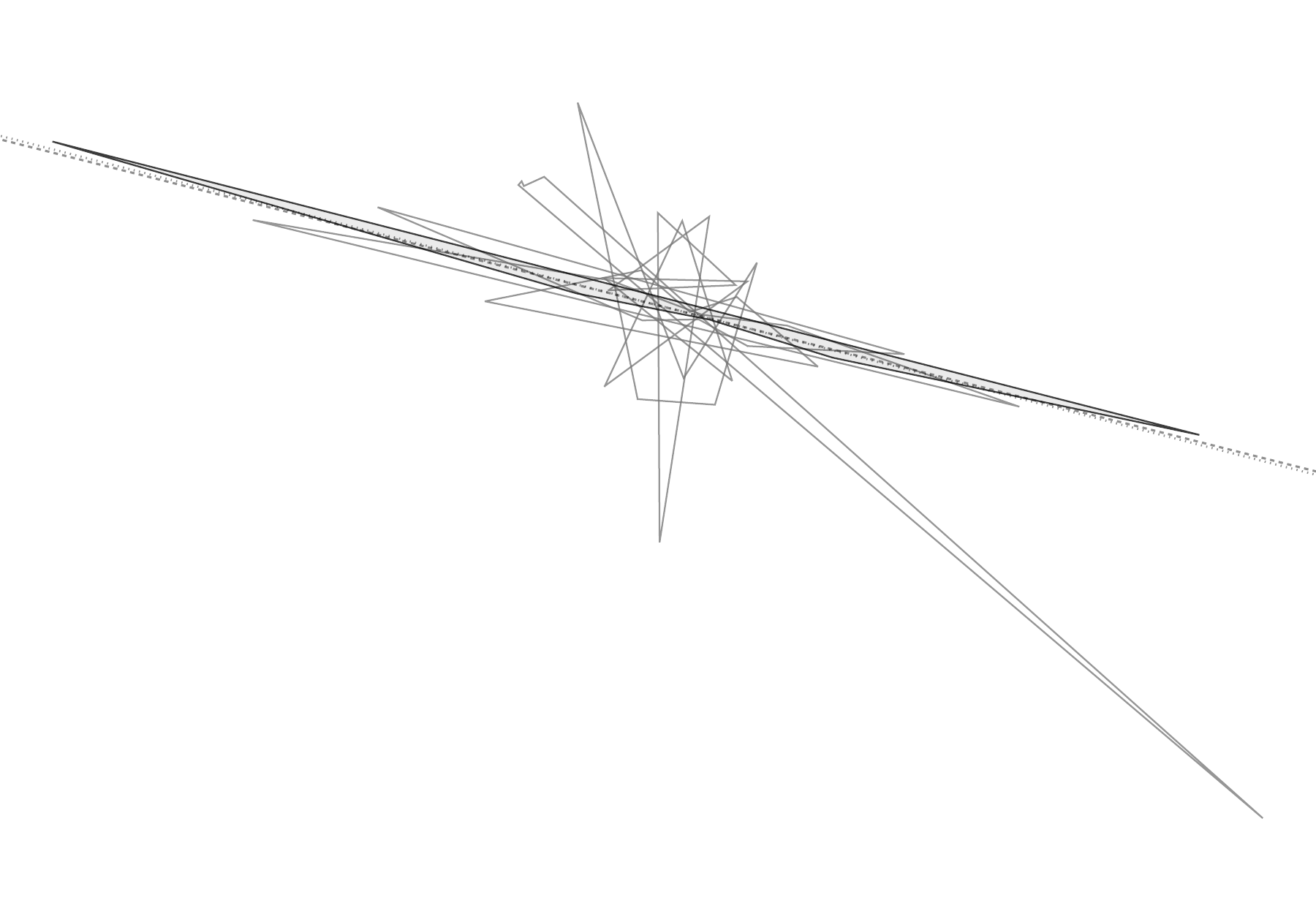}

    \vspace{0.25cm}

    \figpanel{0.48}{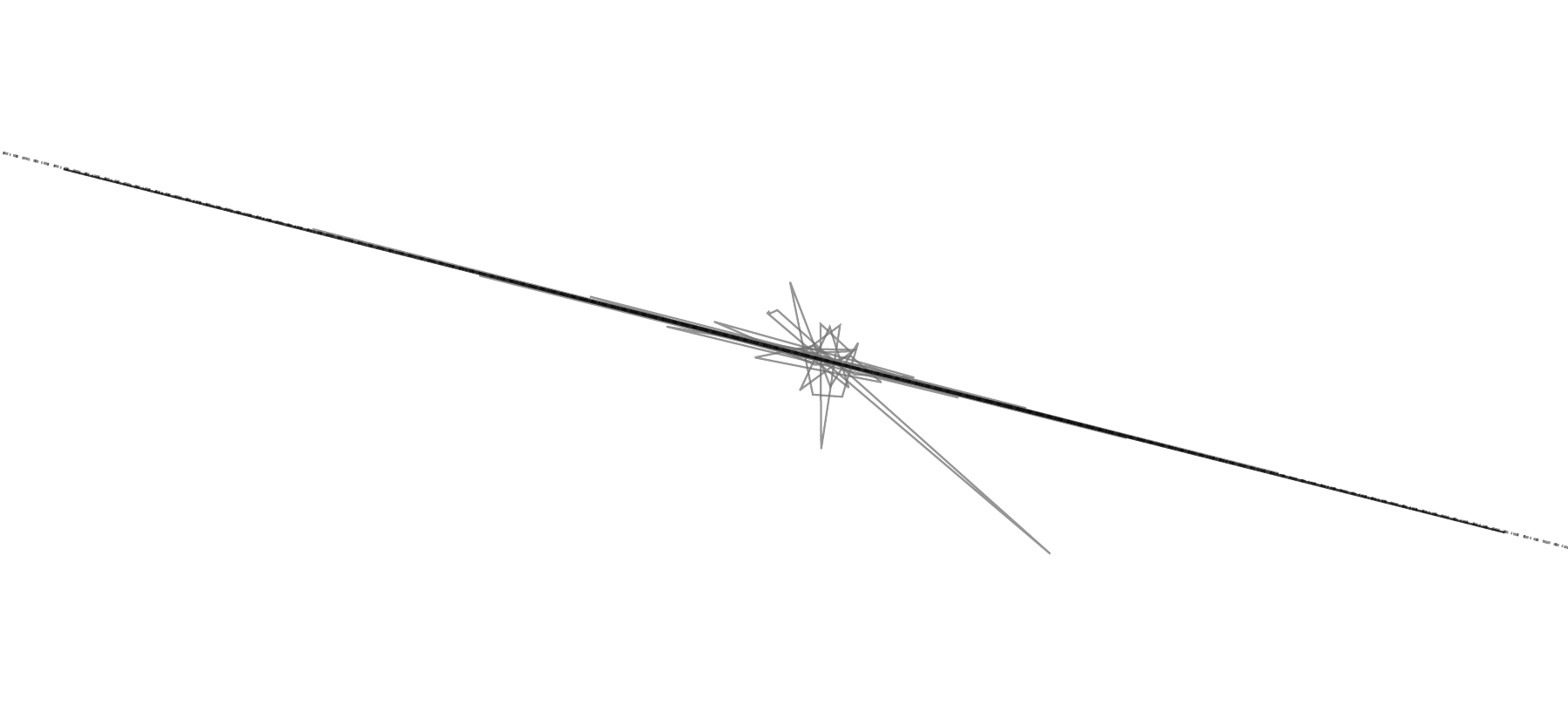}
    \hfill
    \figpanel{0.48}{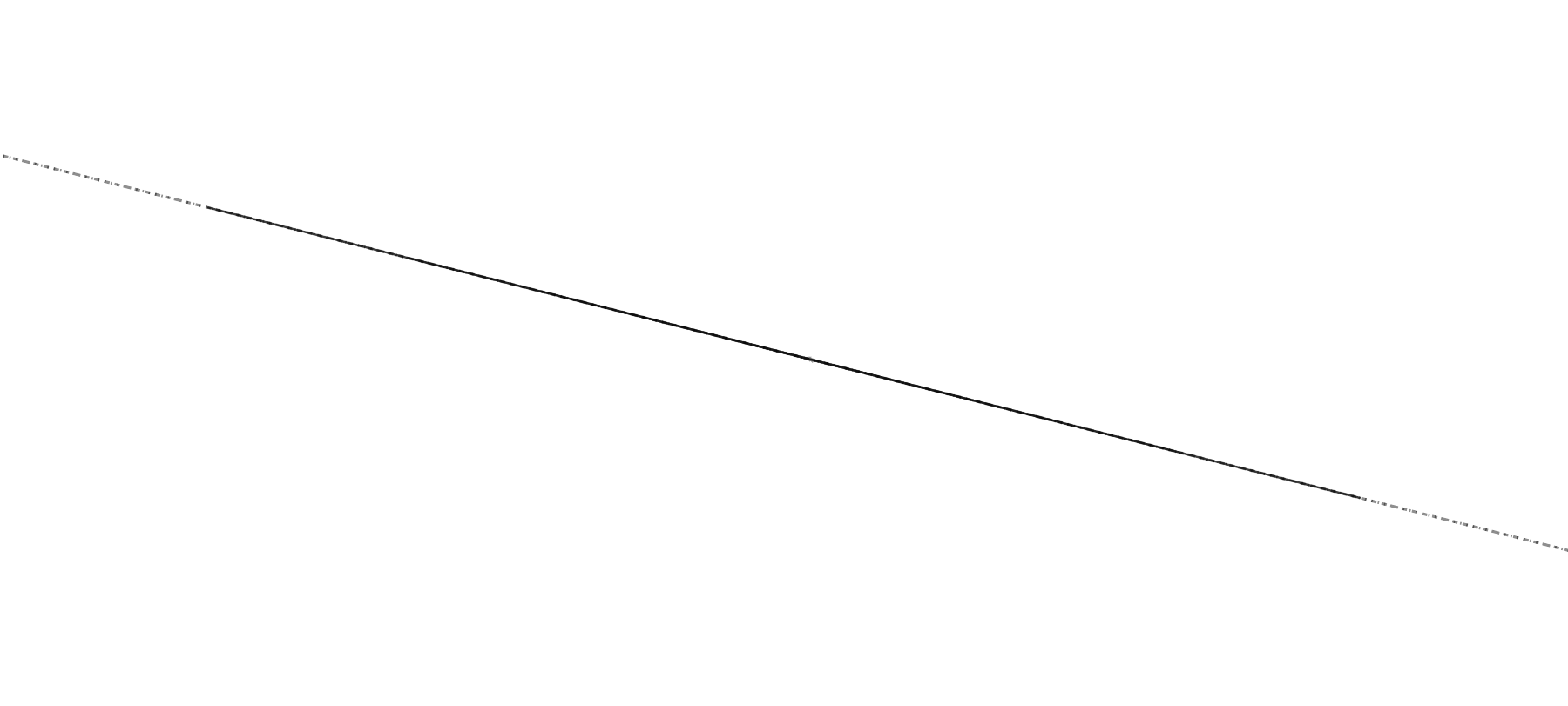}

    \caption{Spectral flattening of a~pentagram-shaped pentagon}
    \label{fig:pentagramPrincipalLines}
\end{figure}

We now illustrate this result on a~concrete example, rather fittingly
given by a~polygon in the~shape of a~pentagram.
Figure~\ref{fig:pentagramPrincipalLines} begins with the~initial polygon
and its predicted spectral line.  The~remaining panels show steps
$0,1,2,\ldots,7$, $10$, and $20$ of the~area-normalized map, together
with the~current principal line.  The~initial stages are visually
irregular, but a~clear asymptotic pattern emerges: the~principal lines
approach the~spectral line, while the~polygons become progressively flatter
and increasingly concentrated near the~same line.

\subsection{The~Poncelet return spectrum}
\label{subsec:poncelet-spectrum}

Schwartz's Darboux-type theorem asserts, in its labeling convention,
that one projectivity carries $\T^2(P)$ back to $P$ for every polygon
in a~fixed Poncelet family \cite[Theorem~1.1]{SchwartzPoncelet}.  With
the~convention adopted here, its inverse gives the~forward relation in
Lemma~\ref{lem:projective-periodicity-input}; in the~Jacobi labeling
used below, the~corresponding cyclic shift is $\Sigma_3$.  We now
identify this forward projectivity explicitly and compute its spectrum.

To obtain a~computable normal form, we use the~results of Lomel\'i and
Meiss.  Their Lemma~19 and Remark~14 show that a~pencil of nested
ellipses is projectively equivalent to a~standard pencil whose outer
conic is the~unit circle
\cite[Lemma~19 and Remark~14]{LomeliMeissPoncelet}.  Their standard
covering map and conjugacy identify every Poncelet map in this pencil
with a~translation in a~Jacobi parameter
\cite[Definition~4, Lemma~20, and Theorem~2]
{LomeliMeissPoncelet}.  The~Jacobi modulus is the~projective
eccentricity of the~pencil.

We use the~standard Jacobi elliptic functions
$\sn(u,k)$, $\cn(u,k)$, and $\dn(u,k)$ with modulus $k\in[0,1)$,
and write
$$
K(k)
=
\int_0^{\constpi/2}
\frac{\d\theta}{\sqrt{1-k^2\sin^2\theta}}
$$
for the~complete elliptic integral of the~first kind.  For a~
noncircular pencil, let $\varepsilon\in(0,1)$ be its projective
eccentricity.  We also include the~circular limiting case
$\varepsilon=0$.

Lomel\'i and Meiss construct a~covering coordinate
$\theta\in\R$ on the~outer conic in which the~Poncelet map
becomes a~rigid translation.  More precisely, let
$\mathcal C_{\mathrm{out}}$ be the~outer ellipse and let
$\mathcal C_{\mathrm{in}}$ be an~inner member of the~associated
pencil.  Since this pair of conics is fixed throughout each
calculation, we suppress the~dependence of the~Poncelet map on the~
chosen member of the~pencil and write
$$
\mathsf P\colon
\mathcal C_{\mathrm{out}}
\rightarrow
\mathcal C_{\mathrm{out}}.
$$
Starting from a~point of $\mathcal C_{\mathrm{out}}$, one draws the~
appropriately oriented tangent to $\mathcal C_{\mathrm{in}}$, and
$\mathsf P$ assigns to the~initial point the~second intersection of
this tangent with $\mathcal C_{\mathrm{out}}$.

They introduce a~covering map
$\Pi_M\colon\R\rightarrow\mathcal C_{\mathrm{out}}$,
where $M$ describes the~projective change of coordinates reducing
the~pencil to its standard diagonal form.  In this covering
coordinate, the~Poncelet map admits a~lift
$\widetilde{\mathsf P}\colon\R\rightarrow\R$
satisfying
$\mathsf P\circ\Pi_M=
\Pi_M\circ\widetilde{\mathsf P}$.
Equivalently, for every $\theta\in\R$,
$\mathsf P\bigl(\Pi_M(\theta)\bigr)=
\Pi_M\bigl(\widetilde{\mathsf P}(\theta)\bigr)$.
The~lift is the~rigid translation
$\widetilde{\mathsf P}(\theta)
=\theta+\omega$,
where $\omega$ is the~rotation number of $\mathsf P$.  Consequently,
$$
\mathsf P\bigl(\Pi_M(\theta)\bigr)
=
\Pi_M(\theta+\omega).
$$
Here and below, $\omega$ is the~value of the~rotation-number function
denoted by $\rho$ in \cite{LomeliMeissPoncelet}; it should not be
confused with the~auxiliary angle denoted there by $\omega$.
Thus the~geometrically defined Poncelet construction becomes
addition of a~constant in the~covering coordinate
\cite[Section~1.4, Definition~4, Lemma~20, and Remark~17]
{LomeliMeissPoncelet}.

For the~standard pencil, the~projective change of coordinates may be
taken to be the~identity, and the~covering map can be written as
$$
\theta
\mapsto
[
\cn(4K(\varepsilon)\theta,\varepsilon):
\sn(4K(\varepsilon)\theta,\varepsilon):
1
].
$$
Its deck transformations are generated by
$\theta\mapsto\theta+1$
\cite[Remark~16]{LomeliMeissPoncelet}.

After interchanging the~two affine coordinates and choosing the~
orientation of the~parameter so that increasing parameter agrees
with the~cyclic labeling, let
$u=4K(\varepsilon)\theta$.
The~covering map then becomes
$$
p(u)
=
[\sn(u,\varepsilon):\cn(u,\varepsilon):1].
$$
In the~rescaled Jacobi parameter, the~Poncelet map therefore lifts
to
$u\mapsto u+h$, where
$h=4K(\varepsilon)\omega$,
or, equivalently,
$\mathsf P\bigl(p(u)\bigr)
=p(u+h)$.
Thus one step of the~Poncelet construction is represented by
translation through the~constant increment $h$.

Accordingly, for an~integer $n\geq5$ and a~primitive rational
rotation number
$$
\omega=\frac qn,
\qquad
0<q<\frac n2,
\qquad
\gcd(q,n)=1,
$$
define
\begin{align}
\label{eq:poncelet-jacobi-parametrization}
p_i(u_0)
&=
p(u_0+ih),
\qquad i\in\mathbb Z,
\\[-1mm]
P(u_0)
&=
\bigl(p_0(u_0),\ldots,p_{n-1}(u_0)\bigr).
\nonumber
\end{align}
After $n$ steps the~parameter has increased by
$nh=4K(\varepsilon)q$.
Since $\sn(\,\cdot\,,\varepsilon)$ and
$\cn(\,\cdot\,,\varepsilon)$ have period $4K(\varepsilon)$, this
gives
$p_{i+n}(u_0)=p_i(u_0)$.
Moreover, the~smallest positive integer $m$ for which
$m\omega\in\mathbb Z$ is $m=n$, because $q/n$ is written in lowest
terms.  Hence the~orbit has primitive period $n$.

For completeness, Lomel\'i and Meiss give an~explicit description of
the~member of the~standard pencil having a~prescribed rotation number.
In their notation, Remark~15 expresses the~corresponding inner conic
directly in terms of the~rotation number, while Theorem~4 solves the~
inverse parameter problem by identifying the~associated member of an~
arbitrary pencil
\cite[Remark~15 and Theorem~4]{LomeliMeissPoncelet}.
After interchanging the~two affine coordinates and using
$h/2=2K(\varepsilon)\omega$, their formula becomes
\begin{align}
\label{eq:poncelet-standard-caustic}
\dn^2\left(\frac h2,\varepsilon\right)x^2+y^2
&=
\cn^2\left(\frac h2,\varepsilon\right).
\end{align}
Consequently, the~segments joining consecutive points in
\eqref{eq:poncelet-jacobi-parametrization} are tangent to this fixed
inner conic.

In the~computations below, we suppress the~modulus and write
$$
\sn(u)=\sn(u,\varepsilon),
\qquad
\cn(u)=\cn(u,\varepsilon),
\qquad
\dn(u)=\dn(u,\varepsilon),
\qquad
\cd(u)=\frac{\cn(u)}{\dn(u)}.
$$
The~cited results provide the~projective normal form and the~translation
model.  The~chord equation, the~diagonal return projectivity, and its
spectrum are computed directly below.

\begin{lemma}[Jacobi chord equation]
\label{lem:jacobi-chord}
Let $u,v\in\R$ and assume that $p(u-v)\ne p(u+v)$.  Then the~line
joining these two points has equation
$$
\dn(v)\sn(u)x+\cn(u)y-\cn(v)z=0.
$$
\end{lemma}

\begin{proof}
The~Jacobi addition formulas give
$$
\dn(v)\sn(u)\sn(u\pm v)
+
\cn(u)\cn(u\pm v)
=
\cn(v).
$$
Hence both $p(u-v)$ and $p(u+v)$ satisfy the~displayed linear equation.
Since the~two points are distinct, this equation defines their unique
projective line.
\end{proof}

For the~values $v=h$ and $v=h/2$ used below, the~distinctness
hypothesis in Lemma~\ref{lem:jacobi-chord} is automatic.  Indeed, the~
deck transformations of $p$ are the~translations by
$4K(\varepsilon)\mathbb Z$, whereas neither $h$ nor $2h$ is a~multiple
of $4K(\varepsilon)$ when $0<\omega<1/2$.

Schwartz's theorem supplies the~return projectivity abstractly.  We now
use Lemma~\ref{lem:jacobi-chord} to identify its forward lift and to
determine its complete spectrum.

\begin{proposition}[Explicit Poncelet return spectrum]
\label{prop:poncelet-return-spectrum}
With the~notation above, put
$$
s_x
=
\frac{\cn(h)\dn(h/2)}{\dn(h)\cn(h/2)}
=
\frac{\cd(h)}{\cd(h/2)},
\qquad
s_y
=
\frac{\cn(h)}{\cn(h/2)},
$$
and let $S=\operatorname{diag}(s_x,s_y,1)$.  Then
$$
\T(P(u_0))_i
=
S\,p\!\left(u_0+\left(i+\frac32\right)h\right),
$$
and consequently
$$
\T^2(P(u_0))_i
=
S^2P(u_0)_{i+3}.
$$
Thus, in this labeling convention, the~forward
Darboux--Schwartz projectivity has a~lift
\begin{align}
\label{eq:poncelet-return-projectivity}
G
&=
S^2
=
\operatorname{diag}(a,b,1),
\end{align}
where
\begin{align}
\label{eq:poncelet-return-eigenvalues}
a
=
\left(\frac{\cd(h)}{\cd(h/2)}\right)^2,
\qquad
b
=
\left(\frac{\cn(h)}{\cn(h/2)}\right)^2.
\end{align}
In particular, $\Spec(G)=\{1,a,b\}.$

For $0<\varepsilon<1$ and the~primitive rotation numbers under
consideration, the~ordering depends only on $\omega$:
$$
\begin{array}{rcl}
0<\omega<1/3
&\Rightarrow&
1>a>b>0,
\\[1mm]
\omega=1/3
&\Rightarrow&
a=b=1,
\\[1mm]
1/3<\omega<1/2
&\Rightarrow&
b>a>1.
\end{array}
$$
The~middle case has reduced denominator $3$ and therefore does not
occur under the~assumptions $n\geq5$ and $\gcd(q,n)=1$.  Likewise, the~
degenerate value $\omega=1/4$, for which $a=b=0$, has reduced
denominator $4$ and is excluded.  Thus the~return spectrum is positive
and simple for $0<\varepsilon<1$.  At the~projectively regular boundary
$\varepsilon=0$, the~two non-unit eigenvalues coincide.  They are
subdominant when $0<\omega<1/3$ and dominant when
$1/3<\omega<1/2$.
\end{proposition}

\begin{proof}
Write $u_i=u_0+ih$.  The~short diagonal $p_ip_{i+2}$ has parameter
center $u_i+h$ and half-span $h$, while $p_{i+1}p_{i+3}$ has center
$u_i+2h$ and the~same half-span.  Put $v=u_i+3h/2$.  By
Lemma~\ref{lem:jacobi-chord}, these two diagonals have equations
$$
\dn(h)\sn(v-h/2)x+\cn(v-h/2)y=\cn(h),
$$
$$
\dn(h)\sn(v+h/2)x+\cn(v+h/2)y=\cn(h).
$$
Substituting $[s_x\sn(v):s_y\cn(v):1]$
into either equation reduces it to the~identity in the~proof of
Lemma~\ref{lem:jacobi-chord}, now with half-span $h/2$.  This proves the~
one-step formula.  Projective naturality and induction give
$$
\T^m(P(u_0))_i
=
S^m p\!\left(u_0+\left(i+\frac{3m}{2}\right)h\right).
$$
The~case $m=2$ gives the~claimed return relation.  Since
$nh=4K(\varepsilon)q$, the~indices in this relation are read modulo
$n$.

It remains to determine the~ordering.  Put
$$
u=\frac h2=2K\omega\in(0,K)
\quad\text{and}\quad
w=\min\{2u,2K-2u\}\in(0,K].
$$
The~symmetries
$\cn(2K-t)=-\cn(t)$,
$\dn(2K-t)=\dn(t)$, and
$\cd(2K-t)=-\cd(t)$
imply
$$
|s_y|=\frac{\cn(w)}{\cn(u)},
\qquad
|s_x|=\frac{\cd(w)}{\cd(u)},
\qquad
\frac{|s_x|}{|s_y|}=\frac{\dn(u)}{\dn(w)}.
$$
For $0<\varepsilon<1$, the~functions $\cn$, $\dn$, and $\cd$ are
positive and strictly decreasing on $(0,K)$.  Indeed,
$$
\cn'(t)=-\sn(t)\dn(t),
\qquad
\dn'(t)=-\varepsilon^2\sn(t)\cn(t),
\qquad
\cd'(t)=-\frac{(1-\varepsilon^2)\sn(t)}{\dn(t)^2}.
$$
If $0<u<2K/3$ and $u\ne K/2$, then $w>u$, and hence $0<|s_y|<|s_x|<1.$
At $u=K/2$ one has $s_x=s_y=0$; this is precisely the~excluded value
$\omega=1/4$.  At $u=2K/3$, the~symmetry identities give
$s_x=s_y=-1$.  Finally, if $2K/3<u<K$, then $w<u$, and therefore $|s_y|>|s_x|>1.$
Squaring proves the~three asserted cases.

When $\varepsilon=0$, one has $\dn\equiv1$ and $\cd=\cn$, so $a=b$;
the~sampled vertices form the~corresponding regular star polygon.
Conversely, for $0<\varepsilon<1$ the~return spectrum is simple.  The~
cyclically aligned return projectivity is unique, since it is determined
by the~images of any four vertices in general position.  Projective
equivalence therefore conjugates the~return projectivity.  Since the~
return projectivity of a~regular star polygon has a~double eigenvalue,
the~noncircular case cannot be projectively regular.
\end{proof}

The~star case $n=7$, $q=2$ has rotation number
$
\omega=\frac27<\frac13.
$
Hence Proposition~\ref{prop:poncelet-return-spectrum} gives
$
1>a>b>0.
$
Figure~\ref{fig:poncelet-heptagon-hyperbolas} illustrates the~
corresponding area-normalized flattening and the~hyperbolic organization
of the~compensated vertex subsequences.  The~left panel shows the~initial
star polygon, the~outer ellipse, the~inner caustic, and the~dominant
spectral line.  The~middle panel adds the~second spectral asymptote and
the~vertex hyperbolas, while the~right panel overlays the~first $50$
area-normalized iterates and the~marked vertex traces.  The~two spectral
lines form the~asymptotes of the~hyperbolas; in the~affine Jacobi model
the~displayed hyperbolas are exact.

\begin{figure}[htbp]
    \centering

    \includegraphics[width=0.32\textwidth]
        {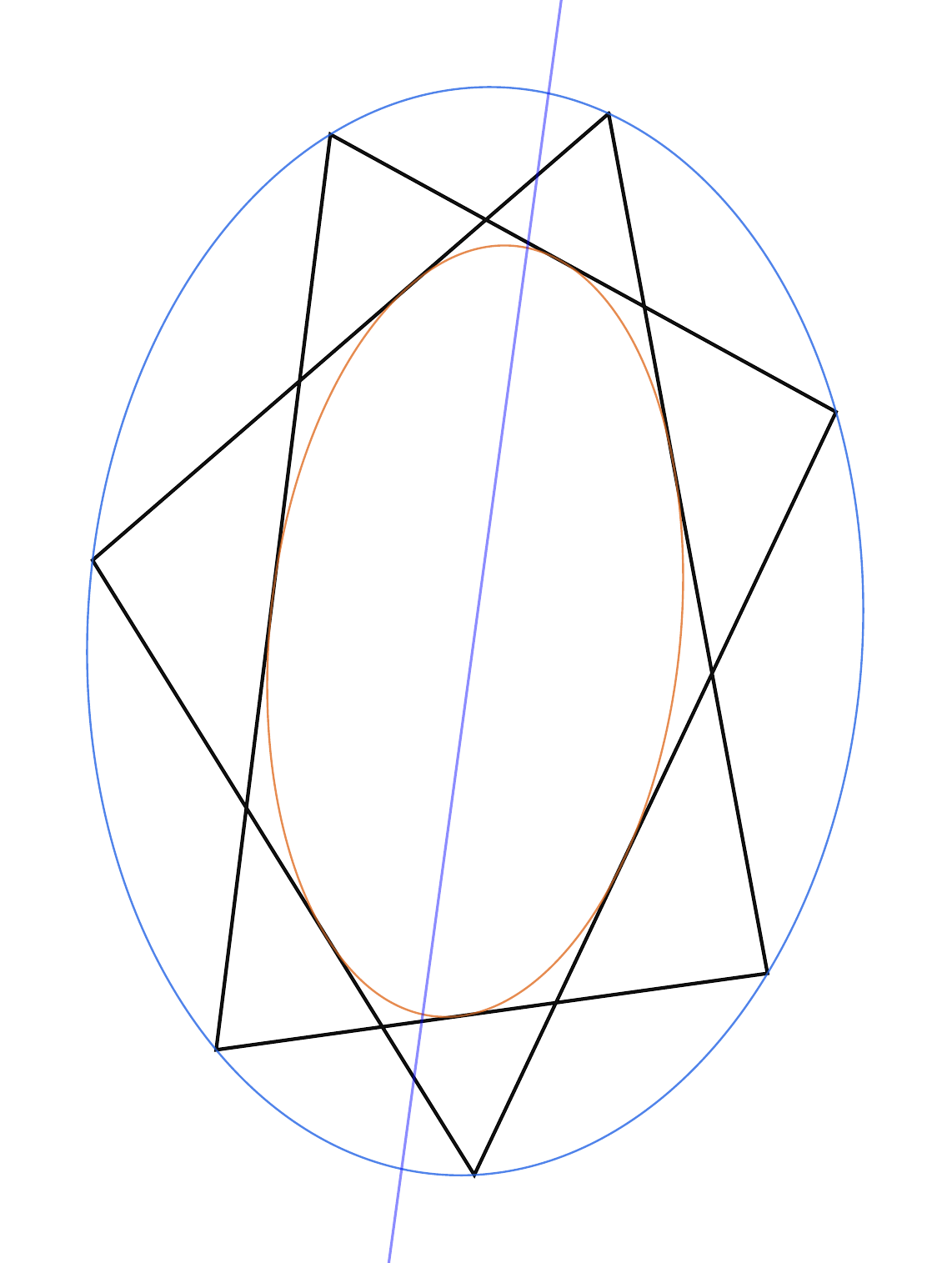}
    \hfill
    \includegraphics[width=0.32\textwidth]
        {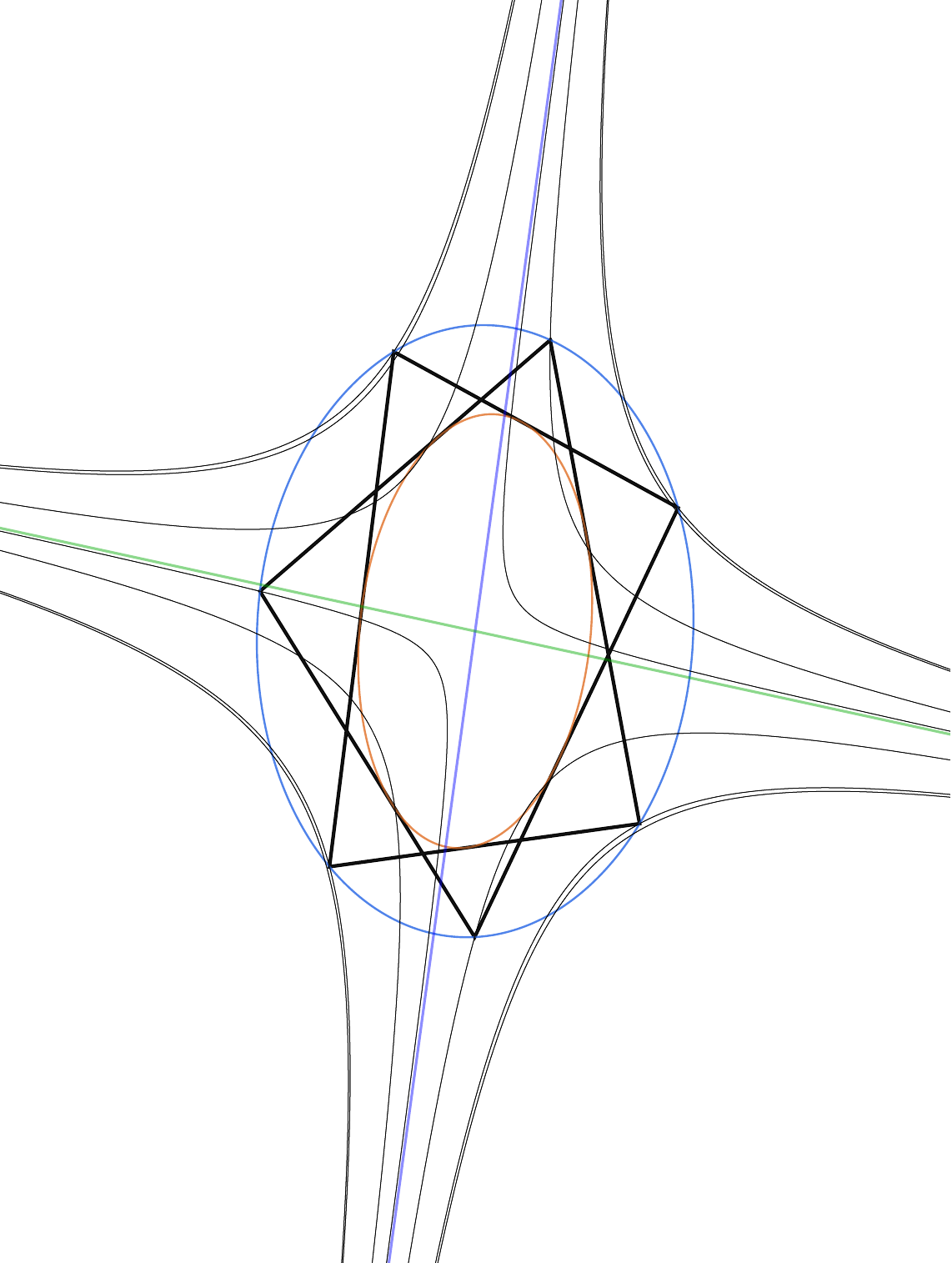}
    \hfill
    \includegraphics[width=0.32\textwidth]
        {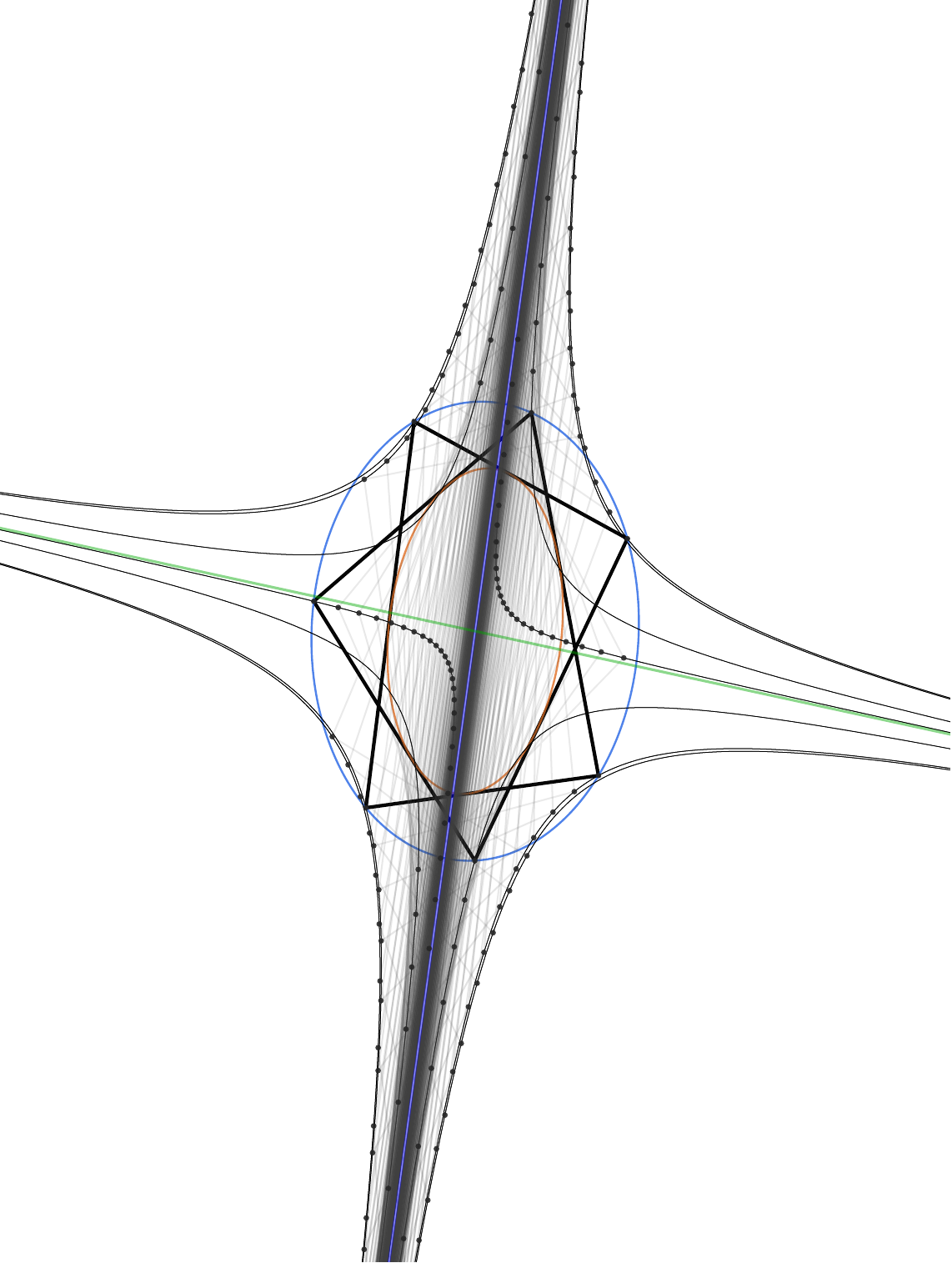}

    \caption{Hyperbolic vertex trajectories for a~Poncelet heptagon
    with rotation number $\omega=2/7$}
    \label{fig:poncelet-heptagon-hyperbolas}
\end{figure}

\begin{corollary}[Strictly convex Poncelet spectrum and flattening]
\label{cor:convex-poncelet-spectrum}
Let $P$ be a~strictly convex Poncelet $n$-gon, $n\geq5$.  If $P$ is not
projectively regular in the~sense of
Definition~\ref{def:projectively-regular}, then the~cyclically aligned
Darboux--Schwartz projectivity has three positive, pairwise distinct
eigenvalues which, after a~common normalization, satisfy
$1>a>b>0$.  Define $\mathfrak s=a+b$ and $\mathfrak d=ab$.  Then
$$
\chi_G(t)
=
(t-1)(t^2-\mathfrak s t+\mathfrak d)
\quad\text{and}\quad
\operatorname{disc}(\chi_G)
=
(\mathfrak s^2-4\mathfrak d)
(1-\mathfrak s+\mathfrak d)^2.
$$
The~strictly convex Poncelet locus therefore satisfies
$\mathfrak d>0$,
$\mathfrak s^2>4\mathfrak d$, and
$1-\mathfrak s+\mathfrak d>0$,
so it lies entirely in the~positive real simple-spectrum chamber below
$\mathfrak d=\mathfrak s^2/4$.
Moreover, the~area-normalized iterates flatten to the~spectral line and
their diameter tends to infinity.  The~compensated vertex subsequences
have the~hyperbolic asymptotics of
Corollary~\ref{cor:hyperbolic-vertex-asymptotics}; in the~Jacobi normal
form the~corresponding hyperbolas are exact.
\end{corollary}

\begin{proof}
After reversing the~cyclic labeling if necessary, the~vertices of a~
strictly convex Poncelet polygon occur in their positive cyclic order on
the~outer conic.  The~Poncelet map therefore advances by one vertex, so
its primitive rotation number is
$\omega=\frac1n<\frac13$.
Since $P$ is not projectively regular, the~Jacobi modulus is nonzero.
Equations~\eqref{eq:poncelet-return-projectivity}
and~\eqref{eq:poncelet-return-eigenvalues}, together with
Proposition~\ref{prop:poncelet-return-spectrum}, give $1>a>b>0$.
The~characteristic polynomial and discriminant are immediate, since
$\mathfrak s^2-4\mathfrak d=(a-b)^2$ and
$1-\mathfrak s+\mathfrak d=(1-a)(1-b)$.

The~pentagram map sends a~strictly convex polygon to a~strictly convex
polygon contained in its interior \cite{SchwartzPentagram}.  Hence all
forward iterates are defined and have nonzero oriented area.  In the~
Jacobi chart, let
$R_0=P(u_0)$ and $R_1=\T(P(u_0))$.  For each $r\in\{0,1\}$, the~
subsequence $\T^{2k+r}(P)$ is, up to cyclic relabelling, $G^k(R_r)$.
Thus, if the~vertices of $R_r$ have affine coordinates
$(x_{r,j},y_{r,j})$, the~corresponding return iterates have coordinates
$(a^k x_{r,j},b^k y_{r,j})$.
After centering, this is an~exact two-scale expansion with scales
$a^k$ and $b^k$, and its leading mixed area coefficient is the~nonzero
oriented area of $R_r$.

It remains to check that this nondegeneracy survives the~projective
conjugacy to the~given pair of nested ellipses.  In affine coordinates
adapted to the~attracting eigenpoint and the~dominant invariant line,
the~conjugacy has the~local form
$$
\Psi(x,y)
=
\left(
\frac{\alpha x+\beta y}{1+cx+dy},
\frac{\gamma y}{1+cx+dy}
\right),
\qquad
\alpha\gamma\ne0.
$$
The~absence of a~pure $x$-term in the~second component expresses the~
invariance of the~dominant line.  Since
$$
\mathrm{D}\Psi(0)
=
\begin{pmatrix}
\alpha&\beta\\
0&\gamma
\end{pmatrix},
$$
Taylor expansion at the~attracting eigenpoint shows that the~two scales
remain $a^k$ and $b^k$, while the~leading mixed area coefficient is
multiplied by $\det\mathrm{D}\Psi(0)=\alpha\gamma$.  It therefore remains
nonzero.  All hypotheses of Theorem~\ref{thm:spectral-flattening} now
hold.  The~last assertion follows from
Corollary~\ref{cor:hyperbolic-vertex-asymptotics}.
\end{proof}

\begin{figure}[htbp]
    \centering
    \includegraphics[width=0.72\textwidth]{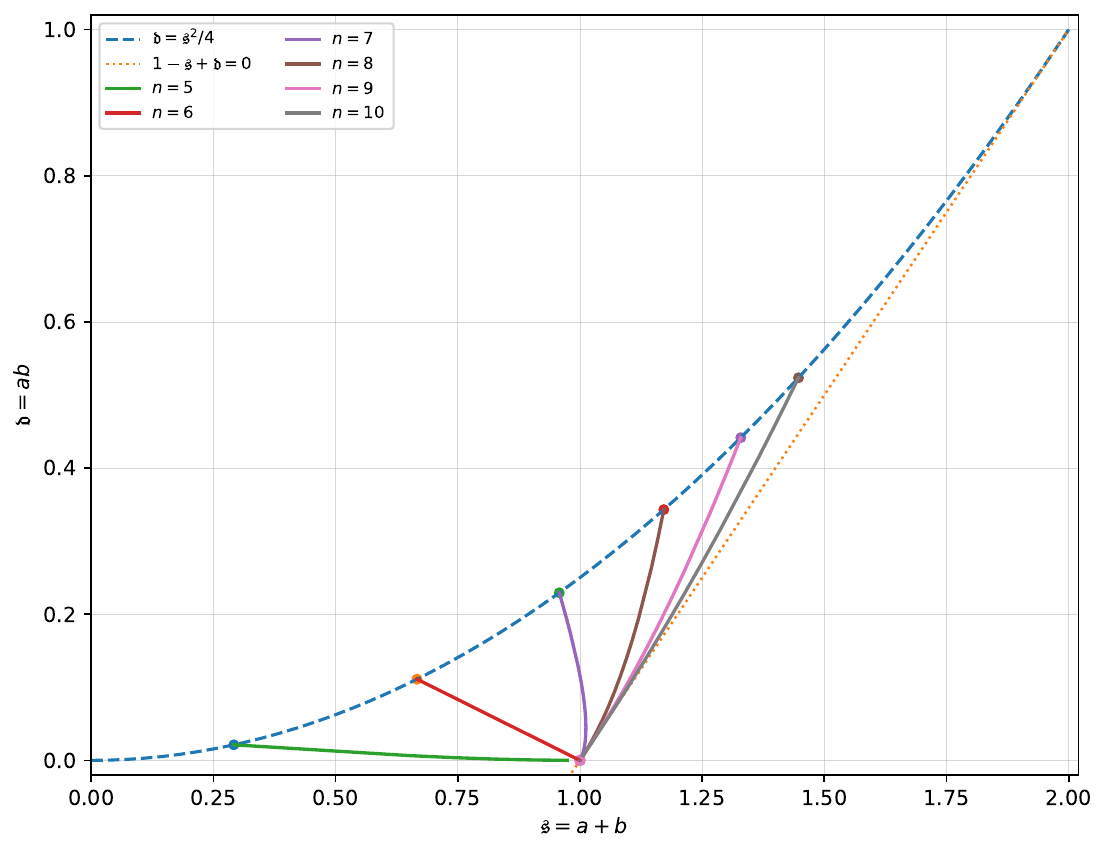}
    \caption{Strictly convex Poncelet spectral curves for
    $n=5,\ldots,10$}
    \label{fig:poncelet-spectral-curves}
\end{figure}

Figure~\ref{fig:poncelet-spectral-curves} plots the~explicit strictly
convex Poncelet spectral curves for $n=5,\ldots,10$ in the~
$(\mathfrak s,\mathfrak d)$-plane.  They lie in the~positive real
simple-spectrum chamber below the~repeated-root parabola
$\mathfrak d=\mathfrak s^2/4$, and their marked endpoints on that
parabola correspond to the~projectively regular cases
$\varepsilon=0$.

At the~circular endpoint $\varepsilon=0$ one has
$$
a=b
=
\left(
\frac{\cos(2\constpi/n)}{\cos(\constpi/n)}
\right)^2
$$
for the~strictly convex rotation number $1/n$.  This is the~
projectively regular boundary of the~diagram.  It does not flatten;
its common asymptotic ellipse is described later in
Proposition~\ref{prop:projectively-regular-poncelet}.

\subsection{Spectral diagrams for pentagons and hexagons}

We conclude this section with explicit algebraic descriptions of the~
spectral regions for pentagons and hexagons.  The~pentagonal diagram is
one-dimensional after projective normalization: one scalar $\kappa$
determines the~entire characteristic polynomial.  For hexagons, the~two
centered coefficients from Lemma~\ref{lem:centered-spectral-invariants}
are generally independent, and the~transition between real and non-real
spectra is the~standard cusp discriminant of a~depressed cubic.  Strict
convexity confines the~invariants to the~positive chamber of this
two-parameter diagram.

\begin{lemma}[Pentagonal normal form]
\label{lem:pentagonal-normal-form}
Let
$P(u,v)=
\bigl((0,0),(1,0),(1,1),(0,1),(u,v)\bigr)$
be a~labeled pentagon with $uv(1-v)\ne0$, so that Glick's operator is
defined.  Put
$$
\kappa(P)
=
\frac{(1-u)(u-v)(u+v-1)}{uv(1-v)}
\quad\text{and}\quad
\kappa_- = \frac{-11-5\sqrt{5}}{2}.
$$
Then
$\chi_{G_P}(t)=t^3-t^2+\kappa(P)t+\kappa(P)$, where
$G_P=L_P-3I$.

If $P(u,v)$ is strictly convex and labeled in cyclic order, then
$u<0$,
$0<v<1$,
and
$\kappa(P)\leq\kappa_-$.
Equality occurs only at
$(u,v)
=
\left(\frac{1-\sqrt5}{4},\frac12\right)$.
\end{lemma}

\begin{proof}
Using the~standard affine lifts
$$
\begin{aligned}
\widetilde p_1&=(0,0,1)^{\top},
&
\widetilde p_2&=(1,0,1)^{\top},
&
\widetilde p_3&=(1,1,1)^{\top},
&
\widetilde p_4&=(0,1,1)^{\top},
&
\widetilde p_5&=(u,v,1)^{\top},
\end{aligned}
$$
and substituting them into the~definition
$$
L_P(w)
=
5w-
\sum_{j=1}^5
\frac{\det(\widetilde p_{j-1},w,\widetilde p_{j+1})}
     {\det(\widetilde p_{j-1},\widetilde p_j,\widetilde p_{j+1})}
\widetilde p_j,
$$
we obtain, with respect to the~standard basis of $\R^3$,
$$
G_P=L_P-3I
=
\begin{pmatrix}
-1
&
0
&
1
\\[2mm]
-\dfrac{v}{u}
&
\dfrac{u-v}{1-v}
&
\dfrac{1-u}{1-v}
\\[3mm]
-\dfrac{1}{u}
&
\dfrac{(1-u)(1-2v)}{v(1-v)}
&
\dfrac{2-u-v}{1-v}
\end{pmatrix}.
$$
Writing $\kappa=\kappa(P)$, we obtain
$$
\begin{aligned}
\chi_{G_P}(t)
&=\det(tI-G_P)
=
t^3-t^2
+
\frac{(1-u)(u-v)(u+v-1)}
     {uv(1-v)}\,t
+
\frac{(1-u)(u-v)(u+v-1)}
     {uv(1-v)}\\
&=
t^3-t^2+\kappa t+\kappa.
\end{aligned}
$$
Strict convexity gives the~normal form
$$
p_5=(-x,y),
\qquad
x>0,
\quad
0<y<1.
$$
Thus $u=-x$ and $v=y$, and
$$
\kappa
=
-\left(
\frac{(x+1)^2}{y(1-y)}
+
\frac{x+1}{x}
\right).
$$
Since $y(1-y)\leq1/4$,
$$
-\kappa
\geq
4(x+1)^2+\frac{x+1}{x}.
$$
Moreover,
$$
4(x+1)^2+\frac{x+1}{x}
-
\frac{11+5\sqrt5}{2}
=
\frac{4}{x}
\left(x-\frac{\sqrt5-1}{4}\right)^2
\left(x+\frac{3+\sqrt5}{2}\right)
\geq0.
$$
Hence $\kappa\leq\kappa_-$, with equality only at the~stated point.
\end{proof}

\begin{theorem}[Intrinsic pentagonal spectral classification]
\label{thm:intrinsic-pentagonal-spectrum}
Let $P=(p_1,\ldots,p_5)$ be a~labeled pentagon in the~domain of
Glick's operator.  For arbitrary nonzero homogeneous lifts, put
$$
[ijk]
=
\det(\widetilde p_i,\widetilde p_j,\widetilde p_k).
$$
Then
$$
\kappa(P)
=
\frac{
[235][315][425][124][134]
}{
[415][125][345][123][234]
}
$$
is well defined, independent of the~lifts, and invariant under
projective transformations and cyclic relabelling.  Moreover,
$\kappa(P)=-\det(G_P)$,
$G_P=L_P-3I$,
and
\begin{align}
\label{eq:pentagonal-characteristic-polynomial}
\chi_{G_P}(t)
&=
t^3-t^2+\kappa(P)t+\kappa(P).
\end{align}
Thus $\kappa(P)$ completely determines the~characteristic polynomial
and the~spectrum of $G_P$.

In terms of the~centered invariants from
Lemma~\ref{lem:centered-spectral-invariants},
$$
\mathcal S_5(P)=\frac13-\kappa(P),
\qquad
\mathcal R_5(P)=\frac{2}{27}-\frac43\kappa(P),
$$
so the~pentagonal invariants satisfy the~linear relation $\mathcal R_5
=
\frac43\mathcal S_5-\frac{10}{27}$.

For brevity, write $\kappa=\kappa(P)$ and put $\kappa_\pm=\frac{-11\pm5\sqrt5}{2}.$
The~spectral type is summarized in
Table~\ref{tab:pentagonal-spectral-types}.

\begingroup
\small

\begin{longtable}[c]
{@{}p{0.27\linewidth}p{0.65\linewidth}@{}}

\caption{Spectral types of $G_P$ as a function of $\kappa(P)$}
\label{tab:pentagonal-spectral-types}
\\

\toprule
\text{Range of $\kappa(P)$}
&
\text{Spectral type of $G_P$}
\\
\midrule
\endfirsthead

\multicolumn{2}{c}
{\tablename~\thetable\ continued from the previous page}
\\
\addlinespace

\toprule
\text{Range of $\kappa(P)$}
&
\text{Spectral type of $G_P$}
\\
\midrule
\endhead

\midrule
\multicolumn{2}{r}
{\textit{Continued on the next page}}
\\
\endfoot

\bottomrule
\endlastfoot
$\kappa<\kappa_-$
&
Three real eigenvalues with distinct moduli
\\
$\kappa=\kappa_-$
&
One simple and one double real eigenvalue
\\
$\kappa_-<\kappa<0$
&
One dominant real eigenvalue and a non-real conjugate pair
\\
$\kappa=0$
&
$\Spec(G_P)=\{1,0,0\}$
\\
$0<\kappa<\kappa_+$
&
Three real eigenvalues with distinct moduli
\\
$\kappa=\kappa_+$
&
One simple and one double real eigenvalue
\\
$\kappa>\kappa_+$
&
A dominant non-real conjugate pair and one real eigenvalue
\\
\end{longtable}

\endgroup

At the~two repeated-spectrum values,
$$
\Spec(G_P)=\left\{\begin{array}{ll}
\{2+\sqrt5,-\varphi,-\varphi\}
&\text{for }\kappa=\kappa_-,\\
\{2-\sqrt5,\varphi^{-1},\varphi^{-1}\}
&\text{for }\kappa=\kappa_+,
\end{array}\right.$$
where $\varphi=(1+\sqrt5)/2$.

If $P$ is strictly convex and labeled in cyclic order, then
$\kappa(P)\leq\kappa_-$,
with equality if and only if $P$ is projectively equivalent, with its
cyclic labeling, to a~regular pentagon.  Consequently, every strictly
convex pentagon has real spectrum, and outside the~projectively regular
class its eigenvalues have pairwise distinct moduli.
\end{theorem}

\begin{proof}
The~denominator in the~bracket quotient is, up to sign, the~product of
the~five consecutive brackets occurring in the~definition of $L_P$;
hence it is nonzero on the~stated domain.  If $\widetilde p_i$ is
replaced by $s_i\widetilde p_i$, then every bracket containing $p_i$ is
multiplied by $s_i$.  Each vertex occurs three times in the~numerator
and three times in the~denominator, so all lift-dependent factors
cancel.  Under a~projective transformation represented by
$A\in\GL(3,\R)$, every bracket is multiplied by $\det A$; both the~
numerator and the~denominator contain five brackets, so these factors
also cancel.

First suppose that $p_1,\ldots,p_4$ form a~projective frame.  Apply a~
projective transformation sending them to
$p_1=(0,0)$,
$p_2=(1,0)$,
$p_3=(1,1)$,
$p_4=(0,1)$,
and write $p_5=(u,v)$.  For the~standard affine lifts,
$$
\begin{aligned}
[235]&=1-u,
&
[315]&=u-v,
&
[425]&=u+v-1,&
[415]&=u,
&
[125]&=v,
&
[345]&=1-v,
\end{aligned}
$$
while $[124]=[134]=[123]=[234]=1.$ Therefore the~bracket quotient equals
$$
\frac{(1-u)(u-v)(u+v-1)}{uv(1-v)}
=
\kappa(P).
$$
Lemma~\ref{lem:pentagonal-normal-form} now gives
\eqref{eq:pentagonal-characteristic-polynomial} and
$\kappa(P)=-\det(G_P)$.  These are rational identities in the~homogeneous
vertex coordinates.  After clearing the~consecutive-bracket
denominators, they hold on the~dense open locus on which the~first four
vertices form a~projective frame and therefore hold identically.  Thus
the~formulas extend to every pentagon in the~domain of $L_P$.  Since the~
cyclic sum in \eqref{eq:glick-operator} gives
$L_{\Sigma_s(P)}=L_P$, the~identity $\kappa(P)=-\det(G_P)$ also proves
cyclic invariance.

Substituting $t=z+\frac13$ in
\eqref{eq:pentagonal-characteristic-polynomial} gives
$$
\chi_{G_P}\left(z+\frac13\right)
=
z^3
-
\left(\frac13-\kappa(P)\right)z
-
\left(\frac{2}{27}-\frac43\kappa(P)\right).
$$
Since $\widehat L_P=G_P-\frac13I$ for $n=5$, comparison with
Lemma~\ref{lem:centered-spectral-invariants} proves the~remaining
formulas.

The~discriminant of \eqref{eq:pentagonal-characteristic-polynomial} is $\operatorname{disc}(\chi_{G_P})
=
-4\kappa(\kappa^2+11\kappa-1),$
and the~roots of $\kappa^2+11\kappa-1$ are precisely
$\kappa_-$ and $\kappa_+$.  This gives the~real/complex alternatives in
Table~\ref{tab:pentagonal-spectral-types}.

Suppose that the~cubic has three distinct real roots.  If two of them
had the~same modulus, they would be $\lambda$ and $-\lambda$.  Since
the~sum of all three roots is $1$, the~third root would be $1$, whereas $\chi_{G_P}(1)=2\kappa$.
Thus, away from $\kappa=0$, the~three real roots have distinct moduli.

Now suppose that the~spectrum consists of one real root $r$ and a~
non-real pair $z,\overline z$.  Vieta's relations give $r|z|^2=-\kappa$.
Since $\chi_{G_P}(r)=0$,
$$
\kappa=\frac{r^2(1-r)}{r+1},
\qquad
|z|^2=\frac{r(r-1)}{r+1},
\quad\text{and}\quad
r^2-|z|^2
=
\frac{r(r^2+1)}{r+1}.
$$
For $\kappa_-<\kappa<0$, one has $\chi_{G_P}(1)=2\kappa<0$, so the~
unique real root satisfies $r>1$ and is dominant.  For
$\kappa>\kappa_+$, the~inequalities
$\chi_{G_P}(-1)=-2<0<\chi_{G_P}(0)=\kappa$ place the~unique real root in
$(-1,0)$, so the~conjugate pair is dominant.  In particular, there is
no mixed real/non-real equal-modulus locus.  Direct factorization at
$\kappa_\pm$ gives the~two displayed repeated spectra.

It remains to identify the~equality case in the~convex bound.  Normalize
the~first four vertices projectively and apply
Lemma~\ref{lem:pentagonal-normal-form}.  Equality holds only for $p_5=
\left(\frac{1-\sqrt5}{4},\frac12\right).$
Writing $\varphi=(1+\sqrt5)/2$, the~projectivity represented by
$$
M=
\begin{pmatrix}
\varphi-1&-\varphi&1\\
\varphi&1&0\\
\varphi-1&1&1
\end{pmatrix}
$$
cyclically permutes the~five vertices of this normalized pentagon.
Thus its fifth power is projectively the~identity.  A~nontrivial real
projectivity of order five is projectively conjugate to a~planar
rotation through $2\constpi/5$ or $4\constpi/5$; cyclic convexity selects
the~first alternative.  Hence the~equality case is projectively regular.
Conversely, projective invariance shows that every cyclically labeled
regular pentagon has $\kappa=\kappa_-$.
\end{proof}

The~scalar $-\det(L_P-3I)$ is a~projective invariant for every $n$, but
it is not in general a~complete spectral invariant.  Indeed, if
$c_n=\frac{2n-9}{3}$,
then Lemma~\ref{lem:centered-spectral-invariants} gives
$$
-\det(L_P-3I)
=
-c_n^3+c_n\mathcal S_n(P)-\mathcal R_n(P).
$$
Thus this determinant records only one combination of the~two centered
coefficients.  Pentagons are exceptional because the~identity
$\mathcal R_5=\frac43\mathcal S_5-\frac{10}{27}$ reduces those
coefficients to one parameter; for hexagons they are generally
independent.

The~two repeated-spectrum curves and the~three generic spectral types
are shown in Figure~\ref{fig:pentagon-spectral-regions} for pentagon $\big((0,0),(1,0),(1,1),(0,1),(u,v)\big)$.  The~
normalization uses the~first four vertices as a~projective frame, so
the~diagram is a~chart on the~two-dimensional moduli space of labeled
pentagons.  White denotes three real eigenvalues with distinct moduli;
the~light-gray region has one dominant real eigenvalue and a~subdominant
non-real pair, while the~dark-gray region has a~dominant non-real pair.
The~solid black curves are $\kappa=\kappa_\pm$ and the~singular locus
$\kappa=0$, whereas the~dotted lines $u=0$, $v=0$, and $v=1$ mark domain
degeneracies.

For standard affine lifts, Remark~\ref{rem:bracket-interpretation}
identifies each bracket with twice an~oriented triangle area.
Equivalently, with indices understood modulo $5$,
$$
\kappa(P)
=
-\prod_{i=1}^{5}
\frac{\sArea(p_i p_{i+1}p_{i+3})}
     {\sArea(p_i p_{i+1}p_{i+2})}.
$$
Thus $\kappa(P)$ may be viewed as a~projectively invariant signed-area
multi-ratio.  Since the~two triangles in each factor have the~same
oriented base $p_ip_{i+1}$, that factor is also the~ratio of their
signed heights.  In particular, all five ratios are positive for a~
strictly convex cyclically labeled pentagon, so the~formula immediately
gives $\kappa(P)<0$.

\begin{figure}[htbp]
    \centering
    \includegraphics[width=0.666\textwidth]
        {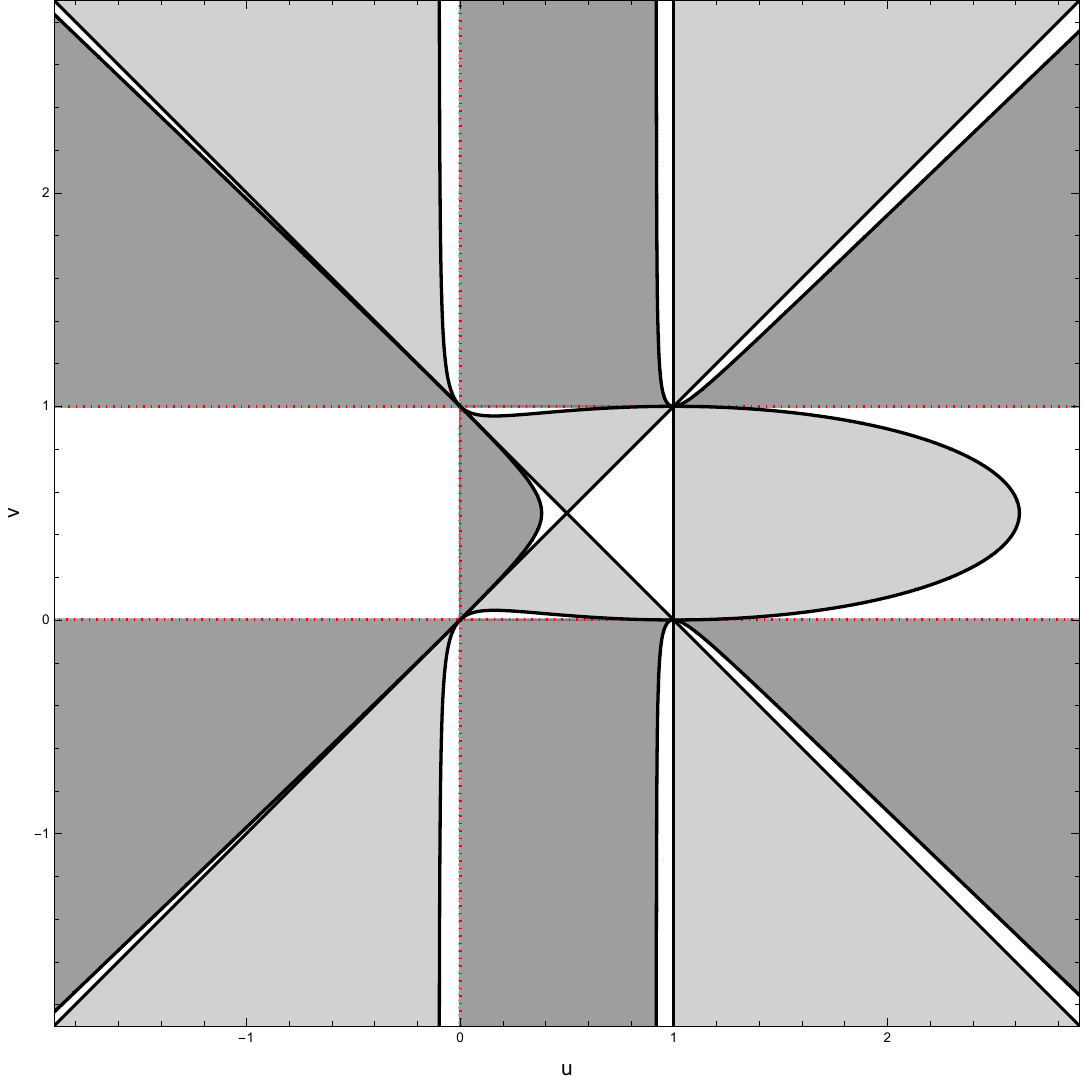}
    \caption{Pentagonal spectral regions in the~normalized chart
    $(u,v)$ for pentagon $\big((0,0),(1,0),(1,1),(0,1),(u,v)\big)$}
    \label{fig:pentagon-spectral-regions}
\end{figure}

\begin{remark}[Geometry of the~pentagonal diagram]
The~denominator of $\kappa$ vanishes on $u=0$, $v=0$, and on $v=1$.
These are genuine domain-degeneracy lines for Glick's operator.  The~
numerator vanishes on $u=1$, $u=v$, and on $u+v=1$,
and there $G_P$ is singular with spectrum $\{1,0,0\}$.  The~two
remaining black curves are the~cubics
$$
(1-u)(u-v)(u+v-1)
=
\kappa_\pm uv(1-v).
$$
Each cubic is smooth at the~four vertices of the~reference square, but
the~two cubics have a~common tangent and contact of order two there;
hence their union has an~\textit{$A_3$-tacnode}, i.e., it is locally analytically
equivalent to $y^2=x^4$, with two smooth branches tangent to the~same
line.  For example, at $(u,v)=(0,0)$ their common linear term is
$v-u$, while the~difference of their defining equations is
$$
-(\kappa_+-\kappa_-)uv(1-v),
$$
whose restriction to the~common tangent $v=u$ has a~nonzero quadratic
term.  The~other three corners follow by the~cyclic symmetry of the~
projective frame.  These corner points are base points of the~rational function
$\kappa$ and do not represent nondegenerate pentagons.
\end{remark}

\begin{proposition}[Spectral diagram for hexagons]
\label{prop:hexagon-spectral-diagram}
Let $P$ be a~labeled hexagon in the~domain of Glick's operator, and put $G_P=L_P-3I.$
Write
$$
\mathcal S(P)=\mathcal S_6(P)
=
\frac12\tr\bigl((G_P-I)^2\bigr),
\qquad
\mathcal R(P)=\mathcal R_6(P)
=
\det(G_P-I).
$$
Then $\mathcal S(P)$ and $\mathcal R(P)$ are projective invariants and
$$
\chi_{G_P}(t)
=
(t-1)^3-\mathcal S(P)(t-1)-\mathcal R(P).
$$
In the~notation of Remark~\ref{rem:bracket-interpretation},
$$
\mathcal S(P)
=
\sum_{1\leq i<j\leq6}m_{ij}-3
\quad\text{and}\quad
\mathcal R(P)
=
2\mathcal S(P)
-
\sum_{1\leq i<j<k\leq6}
\bigl(\Gamma_{ijk}+\Gamma_{ikj}\bigr).
$$
Put
$
\Delta_P
=
4\mathcal S(P)^3-27\mathcal R(P)^2.
$
Then the~following statements hold.
\begin{enumerate}[label=\textup{(\roman*)}]
\item $\Delta_P>0$ if and only if $G_P$ has three distinct real
 eigenvalues.

\item $\Delta_P=0$ if and only if $G_P$ has a~repeated real
 eigenvalue.  More precisely, there is a~unique $s\in\R$ such that
$\mathcal S(P)=3s^2,$
$\mathcal R(P)=2s^3,$
and
$\Spec(G_P)=\{1+2s,1-s,1-s\}$.
The~case $s=0$ is the~triple spectrum $\{1,1,1\}$.

\item $\Delta_P<0$ if and only if $G_P$ has one real eigenvalue and a~
non-real conjugate pair.  Let $r$ be the~unique real root of
$z^3-\mathcal S(P)z-\mathcal R(P)=0$
and put
$$
\omega
=
\sqrt{\frac{3r^2}{4}-\mathcal S(P)}.
$$
Then $\omega>0$ and
$
\Spec(G_P)
=
\left\{
1+r,
1-\frac r2+\consti\omega,
1-\frac r2-\consti\omega
\right\}.
$
Moreover,
$$
|1+r|^2
-
\left|1-\frac r2+\consti\omega\right|^2
=
3r+\mathcal S(P).
$$
Thus the~real eigenvalue is dominant, has the~same modulus as the~
non-real pair, or is subordinate according as
$3r+\mathcal S(P)$ is positive, zero, or negative.

\item Suppose that $G_P$ has three distinct real eigenvalues.  Two of
them have the~same modulus if and only if
$$
\mathcal R(P)=8-2\mathcal S(P),
\qquad
\mathcal S(P)>3,
\qquad
\mathcal S(P)\ne12.
$$
On this locus, $\Spec(G_P)
=
\left\{
3,
\sqrt{\mathcal S(P)-3},
-\sqrt{\mathcal S(P)-3}
\right\}.$
The~eigenvalue $3$ is uniquely dominant when
$3<\mathcal S(P)<12$, whereas the~opposite pair has the~common maximal
modulus when $\mathcal S(P)>12$.
\end{enumerate}
\end{proposition}

\begin{proof}
Since $n=6$,
$G_P-I=L_P-4I=\widehat L_P$.
The~characteristic-polynomial identity and projective invariance follow
from Lemma~\ref{lem:centered-spectral-invariants}, in particular from
\eqref{eq:centered-characteristic-polynomial}.  To obtain the~
specialized bracket formulas, put
$$
M=\sum_{1\leq i<j\leq6}m_{ij},
\qquad
Q=
\sum_{1\leq i<j<k\leq6}
\bigl(\Gamma_{ijk}+\Gamma_{ikj}\bigr).
$$
Since $c_{ii}=1$ and $c_{i,i\pm1}=0$, grouping ordered pairs and
ordered triples by their underlying unordered index sets gives
$$
\sum_{i,j=1}^6c_{ij}c_{ji}=6+2M
\quad\text{and}\quad
\sum_{i,j,k=1}^6c_{ij}c_{jk}c_{ki}
=
6+6M+3Q.
$$
Substitution into Lemma~\ref{lem:centered-spectral-invariants} yields
$\mathcal S=M-3$,
$\mathcal R=2\mathcal S-Q$.

Set $z=t-1$.  The~shifted characteristic polynomial is
\begin{align}
\label{eq:formulaqzcharpol}
q(z)&=z^3-\mathcal S z-\mathcal R,
\end{align}
whose discriminant is
$\Delta_P=4\mathcal S^3-27\mathcal R^2.$
The~standard discriminant criterion for a~real depressed cubic proves
\textup{(i)} and the~first assertions of \textup{(ii)} and
\textup{(iii)}.

If $\Delta_P=0$, the~roots of $q$ can be written uniquely as $2s,-s,-s$
for some $s\in\R$.  Hence
$$
q(z)=(z-2s)(z+s)^2
=z^3-3s^2z-2s^3,
$$
which proves \textup{(ii)}.

Assume that $\Delta_P<0$, and let $r$ be the~unique real root of $q$.
Since $\mathcal R=r^3-\mathcal S r$, one has
$$
q(z)
=
(z-r)\bigl(z^2+rz+r^2-\mathcal S\bigr).
$$
The~quadratic factor has non-real roots, so
$3r^2-4\mathcal S>0$.
They are therefore
$$
-\frac r2\pm\consti
\sqrt{\frac{3r^2}{4}-\mathcal S}.
$$
After shifting back by $1$, this gives the~spectrum in \textup{(iii)}.
Furthermore,
$$
\left|1-\frac r2+\consti\omega\right|^2
=
1-r+r^2-\mathcal S,
\quad\text{and hence}\quad
|1+r|^2
-
\left|1-\frac r2+\consti\omega\right|^2
=
3r+\mathcal S.
$$

Finally, suppose that three distinct real eigenvalues contain two of
the~same modulus.  They are $\eta$ and $-\eta$, and the~trace condition
forces the~third eigenvalue to be $3$.  The~shifted roots are $2$, $\eta-1$, and $-\eta-1$.
Comparison with $q(z)$ in \eqref{eq:formulaqzcharpol} gives
$\mathcal S=\eta^2+3$ and
$\mathcal R=2(1-\eta^2)=8-2\mathcal S$.
Distinctness is equivalent to $\eta\ne0$ and $\eta^2\ne9$, that is,
$\mathcal S>3$ and $\mathcal S\ne12$.  Conversely, under these
conditions the~characteristic polynomial factors as
$
\chi_{G_P}(t)
=
(t-3)\bigl(t^2-(\mathcal S-3)\bigr),
$
which proves \textup{(iv)}, including the~assertions about maximal
moduli.
\end{proof}

\begin{corollary}[Spectral restriction for convex hexagons]
\label{cor:convex-hexagon-spectral-restriction}
Let $P$ be a~strictly convex labeled hexagon whose vertices are
labeled in cyclic order.  Then
$\mathcal S(P)>0$ and
$\mathcal R(P)>0$.
Consequently, $G_P=L_P-3I$ has a~unique eigenvalue of maximal modulus.
More precisely:
\begin{enumerate}[label=\textup{(\roman*)}]
\item if $\Delta_P<0$, the~real eigenvalue is strictly dominant over
the~non-real conjugate pair;

\item if $\Delta_P=0$, then for a~unique $s>0$,
$
\Spec(G_P)=\{1+2s,1-s,1-s\},
$
and $1+2s$ is strictly dominant;

\item if $\Delta_P>0$, the~three eigenvalues are real and one of them
has strictly largest modulus.  The~two subordinate eigenvalues have the~
same modulus precisely when
$$
\mathcal R(P)=8-2\mathcal S(P),
\qquad
3<\mathcal S(P)<4,
$$
in which case
$
\Spec(G_P)
=
\left\{
3,
\sqrt{\mathcal S(P)-3},
-\sqrt{\mathcal S(P)-3}
\right\}.
$
\end{enumerate}
\end{corollary}

\begin{proof}
By projective invariance, normalize the~homogeneous vertices as
$$
\begin{aligned}
p_1&=[1:0:0],&
p_2&=[1:1:1],&
p_3&=[0:1:0],&
p_4&=[1:a:b],&
p_5&=[0:0:1],&
p_6&=[1:c:d].
\end{aligned}
$$
This is the~standard projective-frame chart obtained from
$p_1,p_2,p_3,p_5$; the~chosen signs below orient the~six consecutive
triples consistently.  Taking the~lifts
$\widetilde p_1=p_1$,
$\widetilde p_2=p_2$, and
$\widetilde p_j=-p_j$ for $j=3,\ldots,6$,
the~six consecutive orientation determinants are
$$
1,
\qquad
b-1,
\qquad
1,
\qquad
c-a,
\qquad
-c,
\qquad
c-d.
$$
Strict convexity therefore implies
$a<c<0$,
$d<c<0$, and
$b>1$.
A~direct computation gives
$$
G_P
=
\begin{pmatrix}
1 &
\dfrac{d(1-c)}{c(c-d)} &
\dfrac{d-1}{c-d}
\\[3mm]
\dfrac{b(1-a)}{b-1} &
1 &
\dfrac{a-b}{b-1}
\\[3mm]
\dfrac{a(d-b)}{a-c} &
\dfrac{bc-ad}{c(a-c)} &
1
\end{pmatrix}=:
\begin{pmatrix}
1&\alpha&\gamma'\\
\alpha'&1&\beta\\
\gamma&\beta'&1
\end{pmatrix}.
$$
The~convexity inequalities imply
$\alpha,\alpha'>0$,
$\beta,\beta',\gamma,\gamma'<0$.
Consequently,
$$
\mathcal S
=
\alpha\alpha'+\beta\beta'+\gamma\gamma'>0
\quad\text{and}\quad
\mathcal R
=
\alpha\beta\gamma+\alpha'\beta'\gamma'>0.
$$

It remains to apply Proposition~\ref{prop:hexagon-spectral-diagram}.
If $\Delta_P<0$, then the~unique real root $r$ of
$z^3-\mathcal S z-\mathcal R$ is positive because its value at $0$ is
$-\mathcal R<0$.  Hence
$3r+\mathcal S>0,$ so the~real eigenvalue is strictly dominant.  If $\Delta_P=0$, the~
identity $\mathcal R=2s^3>0$ gives $s>0$, and
$$
(1+2s)^2-(1-s)^2=3s(s+2)>0.
$$
Finally, in the~real simple-spectrum case, any equality of two moduli
is described by Proposition~\ref{prop:hexagon-spectral-diagram}.  On
that locus, positivity of $\mathcal R=8-2\mathcal S$ forces
$\mathcal S<4$.  Thus the~pair
$\pm\sqrt{\mathcal S-3}$ has modulus less than $1$, while the~third
eigenvalue is $3$.  This proves both uniqueness of the~
maximal modulus and the~final characterization.
\end{proof}

The~converse to the~positivity assertion is false.  For example, the~
nondegenerate hexagon
$$
P=
\left(
(0,0),(2,0),\left(\frac12,1\right),(2,2),(0,2),(-1,1)
\right)
$$
is not convex.  Nevertheless, a~direct computation gives
$$
G_P-I=L_P-4I
=
\begin{pmatrix}
-\frac23&0&\frac13\\[1mm]
\frac53&-\frac32&\frac{11}{3}\\[1mm]
\frac53&0&\frac{13}{6}
\end{pmatrix},
$$
and hence
$\chi_{G_P-I}(z)
=z^3-\frac{17}{4}z-3$.
Thus $\mathcal S(P)=17/4>0$ and $\mathcal R(P)=3>0$.

\begin{corollary}[Spectral flattening for convex pentagons]
\label{cor:convex-pentagon-flattening}
Let $P$ be a~strictly convex labeled pentagon whose vertices are listed
in cyclic order, and assume that $P$ is not projectively equivalent,
with its cyclic labeling, to a~regular pentagon.  Let
$\mu_1,\mu_2,\mu_3$ be the~eigenvalues of
$G_P=L_P-3I$, ordered so that
$|\mu_1|>|\mu_2|>|\mu_3|$,
and let $u$ be the~affine direction of the~line through the~eigenpoints
corresponding to $\mu_1$ and $\mu_2$.  Then
$$
\max_i
\dist\bigl((\That^k(P))_i,C(P)+\R u\bigr)
\rightarrow0
\quad\text{and}\quad
\diam\bigl(\That^k(P)\bigr)
\rightarrow\infty.
$$
More precisely, the~transverse distance satisfies
$$
\max_i
\dist\bigl((\That^k(P))_i,C(P)+\R u\bigr)
=
O\!\left(
\left|\frac{\mu_3}{\mu_2}\right|^{k/2}
\right),
$$
and there exists a~constant $D(P)>0$ such that
$$
\diam\bigl(\That^k(P)\bigr)
\sim
D(P)
\left|\frac{\mu_2}{\mu_3}\right|^{k/2}.
$$
\end{corollary}

\begin{proof}
By Theorem~\ref{thm:intrinsic-pentagonal-spectrum},
$\kappa(P)<\kappa_-<0$.
Consequently, $G_P$ is invertible and has three real eigenvalues with
pairwise distinct moduli.

Strict convexity implies that all forward pentagram iterates are
defined, remain strictly convex, have nonzero area, and converge
vertexwise to a~common finite point
\cite{GlickLimit,SchwartzPentagram}.  Lemmas~
\ref{lem:projective-periodicity-input} and~\ref{lem:cyclic-relabel}
give
$$
G_P^k(P)
=
\Sigma_k\bigl(\T^k(P)\bigr),
\qquad
k\geq0.
$$
Thus the~projective orbit of $P$ under $G_P$ satisfies the~hypotheses
of Lemma~\ref{lem:convex-projective-orbit-nondegeneracy}.  That lemma
shows that the~dominant eigenpoint is finite and that $P$ has an~
area-nondegenerate centered two-scale asymptotic in the~direction $u$,
with scales
$$
\alpha_k=\left(\frac{\mu_2}{\mu_1}\right)^k,
\qquad
\beta_k=\left(\frac{\mu_3}{\mu_1}\right)^k.
$$

All the~hypotheses of Theorem~\ref{thm:spectral-flattening} therefore
hold with $m=1$ and $s=-1$, proving the~flattening and diameter
conclusions.  Applying Proposition~\ref{prop:centered-two-scale} to
the~displayed scales gives
$$
\max_i
\dist\bigl((\That^k(P))_i,C(P)+\R u\bigr)
=
O\!\left(
\left|\frac{\mu_3}{\mu_2}\right|^{k/2}
\right)
$$
and
$$
\diam\bigl(\That^k(P)\bigr)
\sim
D(P)
\left|\frac{\mu_2}{\mu_3}\right|^{k/2}
$$
for some $D(P)>0$.
\end{proof}

\begin{remark}
The~exponentially long-and-thin behavior of constant-unsigned-area
rescalings of non-projectively-regular convex pentagons was first
proved by Schwartz
\cite[Theorem~2.2]{SchwartzPentagram}, by analyzing the~differential at
the~collapse point of the~projectivity carrying the~pentagon to its
pentagram image.  The~corollary above gives an~alternative spectral
derivation on the~entire strictly convex non-projectively-regular
locus.  It does not claim a~new flattening result; rather, it identifies
the~selected line through $L_P-3I$ and gives the~corresponding
transverse and diameter rates.

For a~nonconvex labeled pentagon,
Theorem~\ref{thm:intrinsic-pentagonal-spectrum} still gives the~complete
spectral classification, but the~convex-orbit hypotheses of
Lemma~\ref{lem:convex-projective-orbit-nondegeneracy} are no longer
automatic.  In particular, the~forward-domain and centered
nondegeneracy conditions must be checked separately, and the~chamber
$\kappa(P)>\kappa_+$, in which the~non-real conjugate pair is dominant,
is not covered by Proposition~\ref{prop:elliptic-case}.  We therefore
do not assert an~unconditional line-or-ellipse dichotomy for all
labeled pentagons.
\end{remark}

Thus the~pentagonal part of our framework recovers the~classical
normalized-shape result spectrally on the~entire strictly convex
non-projectively-regular locus and identifies the~selected direction
through $L_P-3I$.  For convex hexagons the~spectrum may instead contain
a~non-real conjugate pair.  Corollary~
\ref{cor:convex-hexagon-spectral-restriction}
shows that such a~pair is necessarily subdominant; the~corresponding
area-normalized dynamics is analyzed in the~next section.

\section{Complex subdominant spectra and elliptic asymptotics}
\label{sec:elliptic-oscillatory}

\noindent Theorem~\ref{thm:spectral-flattening} also shows how flattening can fail.  What is
needed is not merely a~large eigenvalue, but a~dominant real projective
line together with a~centered two-scale asymptotic.  If the~first
nonzero centered term is genuinely two-dimensional with equal rates,
then the~area normalization may not select a~stable line.

This phenomenon occurs already for strictly convex hexagons.  Suppose
that one real eigenvalue is strictly dominant and the~two remaining
eigenvalues form a~non-real conjugate pair:
$$
|\mu_1|>|\mu_2|=|\mu_3|,
\qquad
\mu_3=\overline{\mu_2}\notin\R.
$$
Then the~ordinary iterates may converge projectively to the~dominant
real eigenpoint, but after centering the~first nonzero term belongs to
a~two-dimensional rotating subdominant block.  Area normalization then
generically keeps a~bounded two-dimensional oscillation rather than
selecting one fixed line.  Equal-modulus configurations with a~repeated
real eigenvalue, a~nontrivial Jordan block, or two distinct real
eigenvalues of opposite sign and equal modulus require a~separate
analysis and are not covered by the~proposition below.

All displayed decimal approximations below are rounded to at most six
places after the~decimal point.

\begin{example}[A~convex hexagon with non-real spectrum]
\label{ex:convex-hexagon-complex-spectrum}
For the~strictly convex hexagon
$$
P=
((-3,-2),(-2,-3),(3,1),(3,3),(-2,3),(-3,2)),
$$
the~six consecutive turning determinants are
$9,10,10,5,4,4$, so the~displayed cyclic ordering is strictly convex.
Moreover,
the~intrinsic invariants are
$$
\mathcal S(P)=\frac{461}{5},
\qquad
\mathcal R(P)=\frac{8604}{25},
\qquad
\Delta_P=-\frac{39334412}{625}<0.
$$
Consequently,
$$
\chi_{G_P}(t)
=t^3-3t^2-\frac{446}{5}t-\frac{6324}{25}
\quad\text{and}\quad
\Spec(G_P)
\approx
\{12.099816,-4.549908\pm0.452150\consti\}.
$$
Thus the~real eigenvalue is dominant and the~subdominant block is a~
non-real conjugate pair.  This is the~convex hexagonal configuration
whose area-normalized dynamics is described by the~results below.
\end{example}

This should be contrasted with the~case of a~spectrally dominant
complex-conjugate pair,
$|\mu_2|=|\mu_3|>|\mu_1|$,
where the~pair itself is the~leading spectral block.  In that case the~
real plane spanned by the~real and imaginary parts of a~complex
eigenvector determines a~leading real projective line.

We now make the~preceding discussion more precise.  This is not part of
Theorem~\ref{thm:spectral-flattening}; it describes the~complementary spectral case in
which a~single real eigenvalue is dominant and the~first nonzero centered
term is generated by a~complex-conjugate pair.

\begin{proposition}[Elliptic asymptotics]
\label{prop:elliptic-case}
Let $P$ be a~labeled $n$-gon, and assume that the~following conditions
hold.

\begin{enumerate}[label=\textup{(\roman*)}]
\item\label{ass:elliptic-periodic}
There are an~integer $m\geq1$, a~real projectivity $G$, and a~cyclic
relabelling $\Sigma_s$ (with $s=0$ allowed) such that
$\T^m(P)=\Sigma_s(G(P))$.
All forward pentagram iterates of $P$ are defined and have nonzero
unsigned area.

\item\label{ass:elliptic-spectrum}
Choose a~real matrix
$\widetilde G\in\GL(3,\R)$ representing the~projectivity $G$.  Assume
that $\widetilde G$ has one simple real eigenvalue $\lambda_1$ and one
non-real conjugate pair
$\lambda_{2,3}
=
\rho\conste^{\pm\consti\theta}$,
$\rho>0$,
$0<\theta<\constpi$,
with
$|\lambda_1|>\rho$.
Let $v_1\in\R^3$ be a~real eigenvector for $\lambda_1$ and set
$$
X_1=[v_1]\in\RP^2.
$$
Assume that $X_1$ belongs to the~fixed affine chart
$\R^2\subset\RP^2$ in which the~polygon is considered.  Let $f$ denote
the~affine expression of $G$ on a~neighborhood of $X_1$.  Put
$$
\tau=\frac{\rho}{|\lambda_1|}<1
$$
and choose $\vartheta$ modulo $2\constpi$ so that
$$
\frac{\lambda_2}{\lambda_1}
=
\tau\conste^{\consti\vartheta}.
$$

\item\label{ass:elliptic-regular}
For each residue class $r=0,\ldots,m-1$, put
$$
Q_r=\T^r(P).
$$
Assume that $Q_r$ is in regular spectral position with respect to $G$:
for every homogeneous vertex lift, its spectral projection onto the~
eigenspace $\R v_1$, along the~real invariant plane associated with
$\lambda_2$ and $\lambda_3$, is nonzero.  Put
$$
M=\mathrm{D}f_{X_1},
\qquad
R=\tau^{-1}M,
$$
and define $\xi_{r,i}$ by the~first-order expansion
$$
f^q(q_{r,i})
=
X_1+\tau^qR^q\xi_{r,i}+O(\tau^{2q}).
$$
Set
$$
w_{r,i}
=
\xi_{r,i}-\frac1n\sum_j\xi_{r,j},
\qquad
B_r
=
\left|
\frac12\sum_i\det(w_{r,i},w_{r,i+1})
\right|,
$$
where $\det$ is the~standard determinant in the~chosen affine
coordinate plane.  Assume that
$w_{r,i}\ne0$
for every $i$, and
$B_r>0$.
\end{enumerate}

Let $P_k=\That^k(P)$, put $C_0=C(P)$ and
$A_0=\Area(P)$, and undo the~accumulated cyclic shift by setting
$$
\widetilde P_{q,r}
=
\Sigma_{-qs}(P_{qm+r}).
$$
Then, for every residue class $r=0,\ldots,m-1$, the~vertices have the~
asymptotic form
\begin{align}
\label{eq:elliptic-leading-asymptotic}
\widetilde p_{q,r,i}-C_0
&=
\sqrt{\frac{A_0}{B_r}}\,R^qw_{r,i}+o(1).
\end{align}
The~matrix $R$ has eigenvalues $\conste^{\pm\consti\vartheta}$ and
$\det R=1$.  There is a~positive definite quadratic form $\mathcal Q$,
independent of $r$, for which $R$ is a~rotation.  Consequently,
$$
\mathcal Q(\widetilde p_{q,r,i}-C_0)
\to
\frac{A_0}{B_r}\mathcal Q(w_{r,i}).
$$
Thus, for every fixed residue class $r$ and vertex $i$, the~relabelled
area-normalized vertices are asymptotic to a~level set of one positive
definite quadratic form.
\end{proposition}

\begin{proof}
Fix a~residue class $r$, and put $Q_r=\T^r(P)$.  By
Lemma~\ref{lem:cyclic-relabel},
$$
\T^{qm+r}(P)=\Sigma_{qs}\bigl(G^q(Q_r)\bigr).
$$
The~normalization commutes with cyclic relabelling because oriented
area and barycenter are unchanged by it.  Consequently,
$\widetilde P_{q,r}$ is exactly the~normalization of $G^q(Q_r)$ to the~
fixed unsigned area.  Thus the~labeled asymptotics of the~$r$-th relabelled
subsequence reduce to iterating the~single projectivity $G$.

Let $v\in\mathbb C^3$ be a~complex eigenvector for $\lambda_2$.  The~
vector $v_1$ and the~eigenpoint $X_1=[v_1]$ are those fixed in
condition~\ref{ass:elliptic-spectrum}.  By regular spectral position, a~
homogeneous lift of each vertex of $Q_r$ has the~spectral decomposition
$$
\widetilde q_{r,i}
=
c_{r,i}v_1+z_{r,i}v+
\overline{z_{r,i}}\,\overline v,
\qquad c_{r,i}\ne0.
$$
Applying $G^q$ and dividing by the~dominant coordinate gives the~affine
expansion
$$
f^q(q_{r,i})
=
X_1+
\tau^q R^q\xi_{r,i}
+O(\tau^{2q}),
$$
which also proves the~expansion used in
condition~\ref{ass:elliptic-regular}.  The~eigenvalues of
$M=\mathrm{D}f_{X_1}$ are the~projective ratios
$\lambda_2/\lambda_1$ and $\lambda_3/\lambda_1$.  Hence $R=M/\tau$
has eigenvalues $\conste^{\pm\consti\vartheta}$ and $\det R=1$.
Since these eigenvalues are non-real, $R$ is real-similar to the~
Euclidean rotation through angle $\vartheta$.  It therefore preserves a~
positive definite quadratic form $\mathcal Q$, determined by $G$ and
independent of the~residue class.

Subtracting the~barycenter gives
$$
f^q(q_{r,i})-C(f^q(Q_r))
=
\tau^q R^qw_{r,i}+O(\tau^{2q}).
$$
Because $\det R=1$, the~polygonal area formula and $B_r>0$ yield
$$
\Area(f^q(Q_r))
=
\tau^{2q}B_r+o(\tau^{2q}).
$$
Hence the~area-normalizing factor satisfies
$$
\sqrt{\frac{A_0}{\Area(f^q(Q_r))}}
=
\sqrt{\frac{A_0}{B_r}}\,\tau^{-q}+o(\tau^{-q}).
$$
Multiplying the~centered expansion by this factor gives
$$
\widetilde p_{q,r,i}-C_0
=
\sqrt{\frac{A_0}{B_r}}\,R^qw_{r,i}+o(1).
$$
Since $R$ preserves $\mathcal Q$, we obtain
$$
\mathcal Q(\widetilde p_{q,r,i}-C_0)
\to
\frac{A_0}{B_r}\mathcal Q(w_{r,i}).
$$
Its level sets are ellipses centered at $C_0$.
\end{proof}

\begin{remark}
The~level sets in Proposition~\ref{prop:elliptic-case} are concentric
ellipses in the~chosen affine plane.  The~proposition is only
asymptotic: the~normalized vertices need not lie exactly on these
ellipses for finite $q$.  When $s\ne0$, a~fixed label in the~original
sequence is cyclically permuted as $q$ changes; the~displayed
vertex-by-vertex formula applies after undoing this relabelling, while
the~family of ellipses is unchanged as a~Euclidean set.  If
$\vartheta/\constpi$ is rational, each normalized subsequence is
asymptotic to a~finite periodic leading orbit.  If
$\vartheta/\constpi$ is irrational and $w_{r,i}\ne0$, the~leading
motion of the~$i$-th vertex in the~$r$-th residue class is dense on its
corresponding ellipse.
\end{remark}

\begin{proposition}[The~asymptotic ellipses]
\label{prop:predicting-ellipses}
Assume the~hypotheses of Proposition~\ref{prop:elliptic-case}.  Suppose
that, in the~chosen affine chart, the~projectivity $G$ is written as
$$
f(x)=\frac{Ax+b}{c^{\top}x+d},
$$
where $A$ is a~real $2\times2$ matrix, $b,c\in\R^2$, and
$d\in\R$.  Let $X_1$ be the~finite attracting fixed point corresponding
to the~dominant real eigenvalue.  Then
$$
M=\mathrm{D}f_{X_1}
=
\frac{A-X_1c^{\top}}{c^{\top}X_1+d}
$$
has eigenvalues $\tau\conste^{\pm\consti\vartheta}$, where
$\tau=\rho/|\lambda_1|$.  There is a~positive definite symmetric matrix
$H$, unique up to multiplication by a~positive scalar, such that
\begin{align}
\label{eq:conformal-quadratic-form}
M^{\top}HM&=\tau^2H.
\end{align}

With the~notation of Proposition~\ref{prop:elliptic-case}, put
$$
B_r=
\left|\frac12\sum_i\det(w_{r,i},w_{r,i+1})\right|
$$
and define
$$
s_{r,i}
=
\frac{A_0}{B_r}\,w_{r,i}^{\top}H w_{r,i}.
$$
Then the~asymptotic ellipse followed by the~$i$-th vertex in the
$r$-th residue class is
$$
(y-C_0)^{\top}H(y-C_0)=s_{r,i}.
$$
Replacing $H$ by a~positive multiple multiplies every $s_{r,i}$ by the~
same factor and leaves the~geometric ellipses unchanged.  Consequently,
all the~ellipses are concentric and homothetic.  Their common principal
directions, unless they are circles, and their common axis ratio are
determined by $G$ alone, whereas their levels may depend on the~residue
class and the~vertex.

Equivalently, the~area-normalized orbit is organized into $m$ indexed
elliptic families, one for each residue class modulo $m$.  If $s\ne0$,
the~vertex index $i$ refers to the~relabelled subsequence
$$
\widetilde P_{q,r}
=
\Sigma_{-qs}(P_{qm+r});
$$
without undoing this cyclic relabelling, the~same conclusion holds
setwise.
\end{proposition}

\begin{proof}
Since $M$ has non-real conjugate eigenvalues
$\tau\conste^{\pm\consti\vartheta}$, it is real-similar to $\tau R_\vartheta$.  Thus one
may write
$M=S(\tau R_\vartheta)S^{-1}$
with $S\in \GL(2,\R)$.  Then
$H=(SS^{\top})^{-1}$
satisfies
$M^{\top} H M=\tau^2 H$.
Conversely, equation~\eqref{eq:conformal-quadratic-form} determines the~
conformal structure for which $M/\tau$ is a~rotation.  More explicitly,
if $H'$ is another positive definite solution and
$J=S^{\top}H'S$, then
$R_\vartheta^{\top}JR_\vartheta=J$.
Because $R_\vartheta$ has non-real eigenvalues, a~real symmetric matrix
invariant under this rotation is a~scalar multiple of $I$.  Thus $H'$
is a~positive scalar multiple of $H$.

The~derivative formula follows by differentiating
$$
f(x)=\frac{Ax+b}{c^{\top}x+d}
$$
and using the~fixed point identity
$
AX_1+b=(c^{\top}X_1+d)X_1.
$
The~asymptotic expansion
\eqref{eq:elliptic-leading-asymptotic} gives
$$
\widetilde p_{q,r,i}-C_0
=
\sqrt{\frac{A_0}{B_r}}\,R^q w_{r,i}+o(1).
$$
Since equation~\eqref{eq:conformal-quadratic-form} is equivalent to
$R^{\top}HR=H$, we obtain
$$
(\widetilde p_{q,r,i}-C_0)^{\top}
H(\widetilde p_{q,r,i}-C_0)
\rightarrow
\frac{A_0}{B_r}w_{r,i}^{\top}Hw_{r,i}
=
s_{r,i}.
$$
The~matrix $H$ depends only on $G$ and is therefore common to all
residue classes; only $w_{r,i}$, $B_r$, and hence $s_{r,i}$ may depend
on $r$.  The~statement concerning shifted labels follows directly from
the~definition of $\widetilde P_{q,r}$.  This proves all the~assertions.
\end{proof}

\begin{remark}
If $0<h_1\leq h_2$ are the~eigenvalues of $H$ and $e_1,e_2$ are
corresponding orthonormal eigenvectors, then $e_1,e_2$ give common
principal directions when $h_1<h_2$.  If $h_1=h_2$, the~level sets are
circles and no pair of principal directions is distinguished.  The~
common axis ratio is $a/b=\sqrt{h_2/h_1}$.

For the~level $s_{r,i}$, the~asymptotic semi-axis lengths are
$a_{r,i}=\sqrt{s_{r,i}/h_1}$ and
$b_{r,i}=\sqrt{s_{r,i}/h_2}$.
Thus $G$ determines the~conformal shape of all limiting ellipses, while
the~initial polygon and the~chosen residue class determine the~individual
levels.
\end{remark}

Before discussing the~role of residue classes, we illustrate the~
elliptic case numerically with a~nonconvex pentagon.  By
Theorem~\ref{thm:intrinsic-pentagonal-spectrum}, such a~complex
subdominant pair cannot occur for a~strictly convex pentagon.
Figure~\ref{fig:pentagon-ellipses} displays an~orbit in which the~
area-normalized iterates do not flatten to a~line.  Instead, the~
vertices oscillate around the~attracting projective point and approach
the~ellipse levels.  The~first panel shows the~initial
pentagon and the~ellipses; the~next three show iterates $65$,
$66$, and $67$, while the~last two show the~first $100$ and the~
first $700$ iterates.  The~projective dynamics contracts toward the~
dominant fixed point, while the~transverse component rotates and
produces the~visible elliptic pattern after area normalization.

\begin{remark}[The~hexagon case]
\label{rem:two-ellipse-families}
For a~labeled hexagon,
Lemma~\ref{lem:projective-periodicity-input} gives
$$
\T^2(P)=G_P(P),
\qquad
G_P=L_P-3I,
$$
so one may take $m=2$ and $s=0$.  Whenever $G_P$ satisfies the~
hypotheses of Proposition~\ref{prop:elliptic-case},
Proposition~\ref{prop:predicting-ellipses} therefore yields two indexed
families of six asymptotic ellipses: one for the~even iterates and one
for the~odd iterates.  Both families have the~same center, principal
directions, and axis ratio, although their individual levels may
differ or occasionally coincide.
\end{remark}

\begin{figure}[htbp]
    \centering
    \includegraphics[width=0.23\textwidth,angle=90]{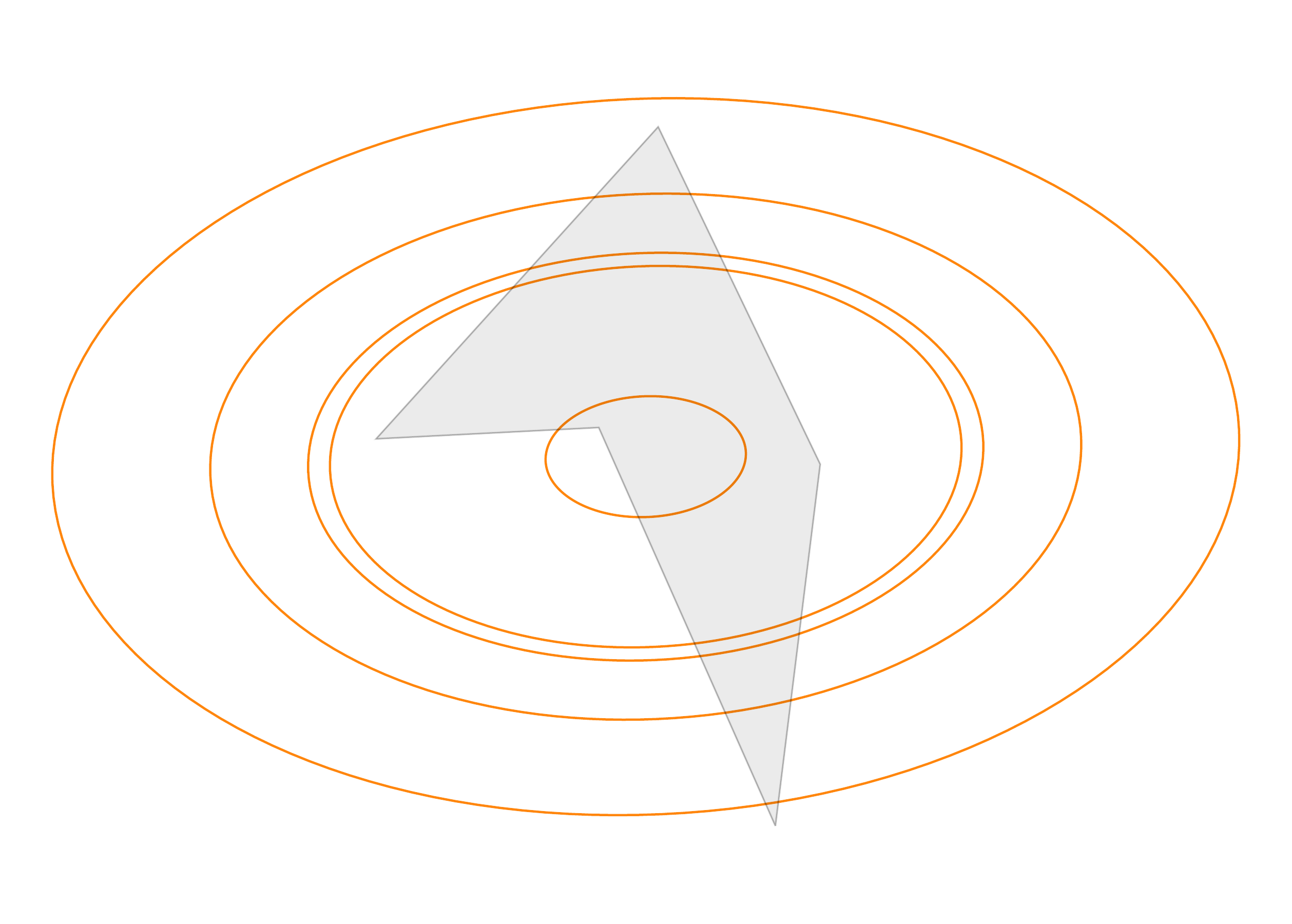}
    \hfill
    \includegraphics[width=0.23\textwidth,angle=90]{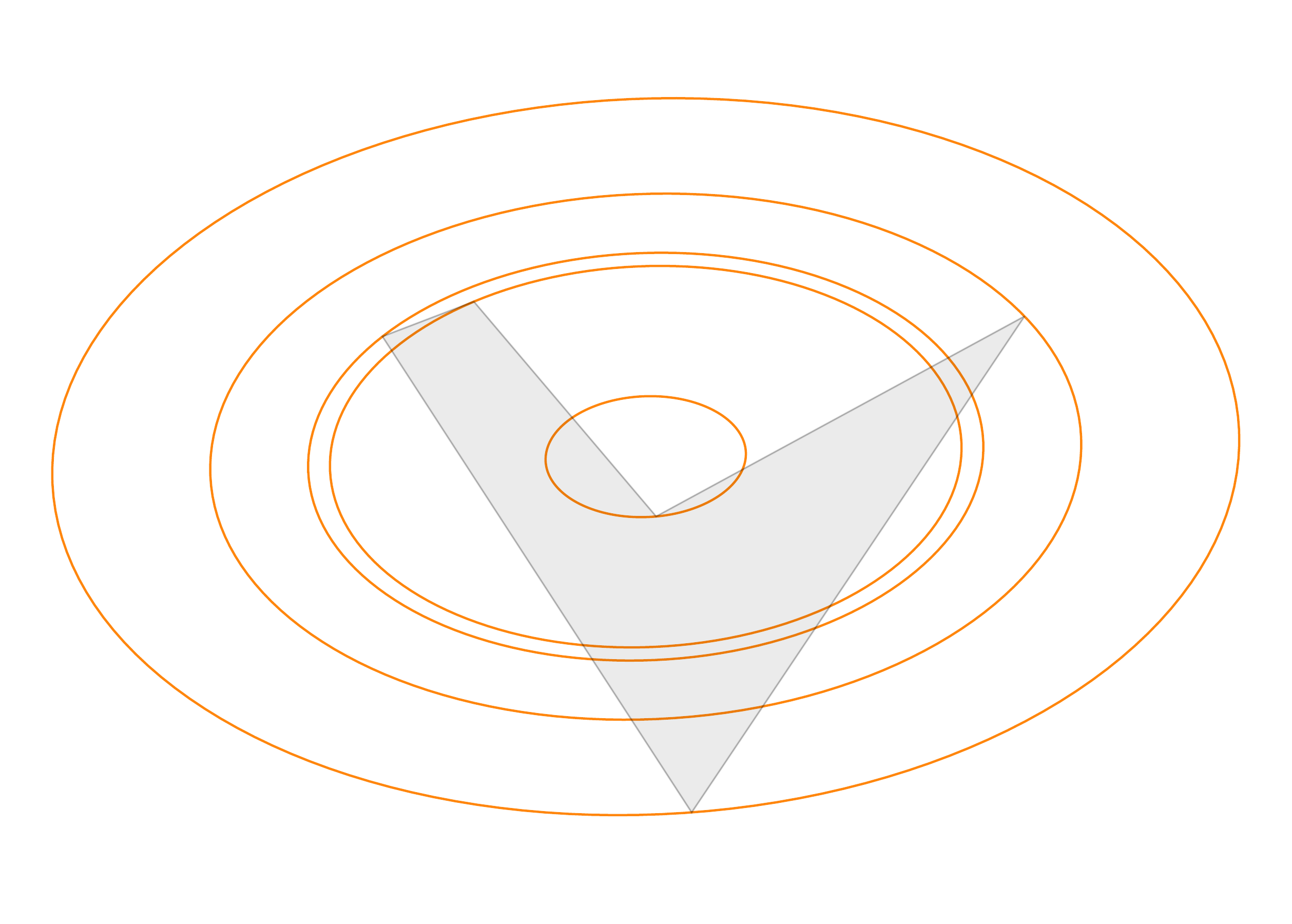}
    \hfill
    \includegraphics[width=0.23\textwidth,angle=90]{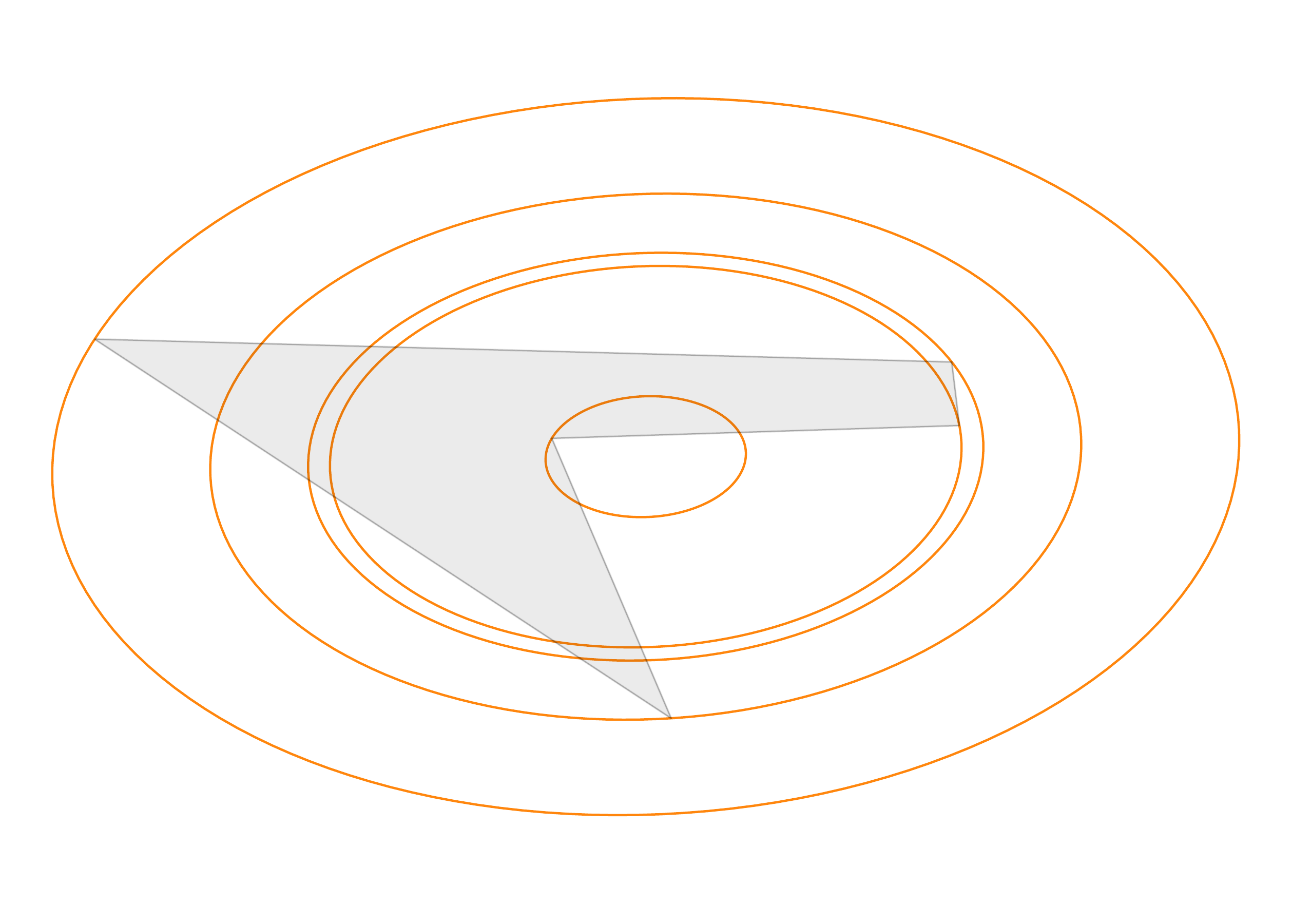}

    \vspace{0.25cm}

    \includegraphics[width=0.23\textwidth,angle=90]{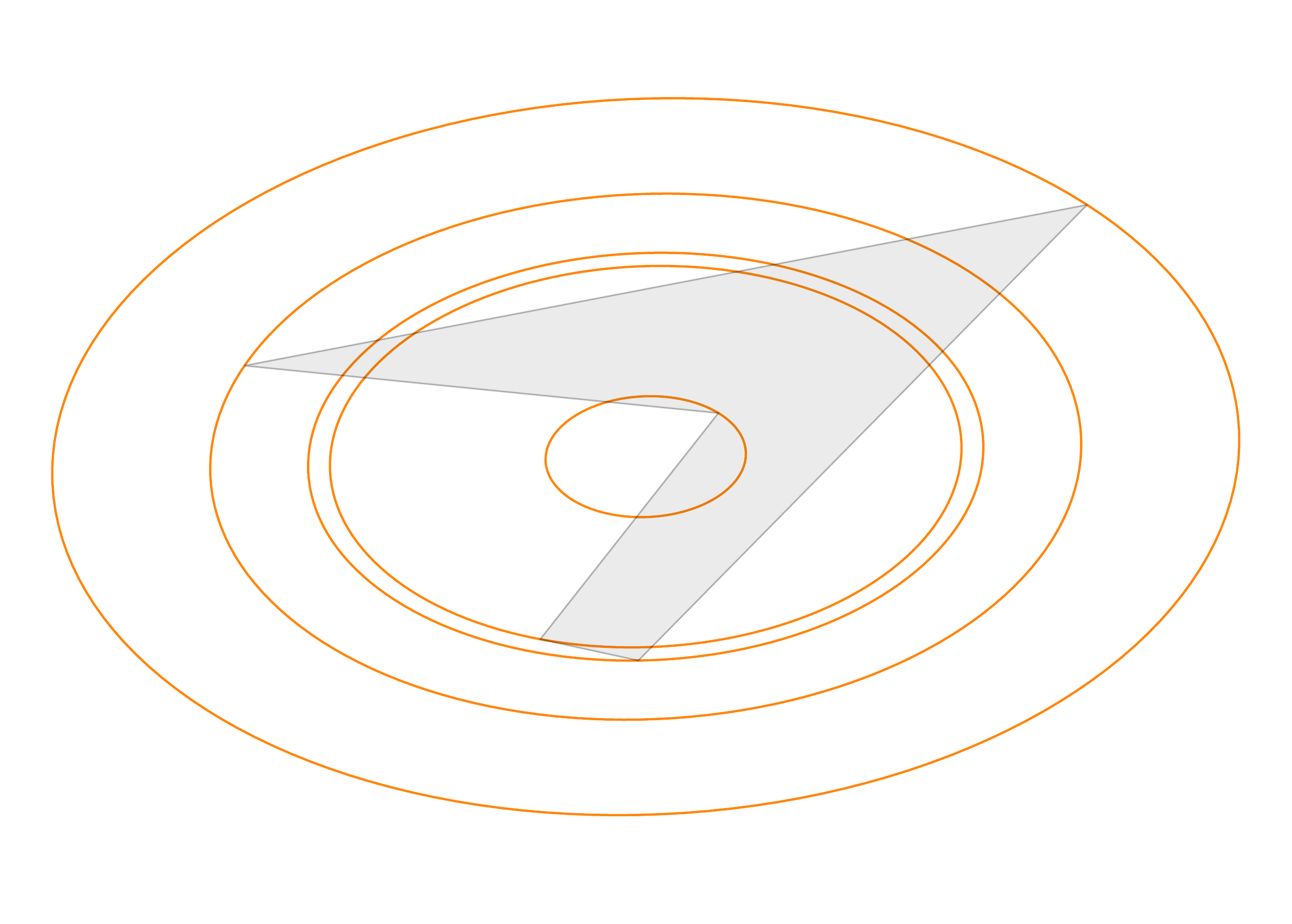}
    \hfill
    \includegraphics[width=0.23\textwidth,angle=90]{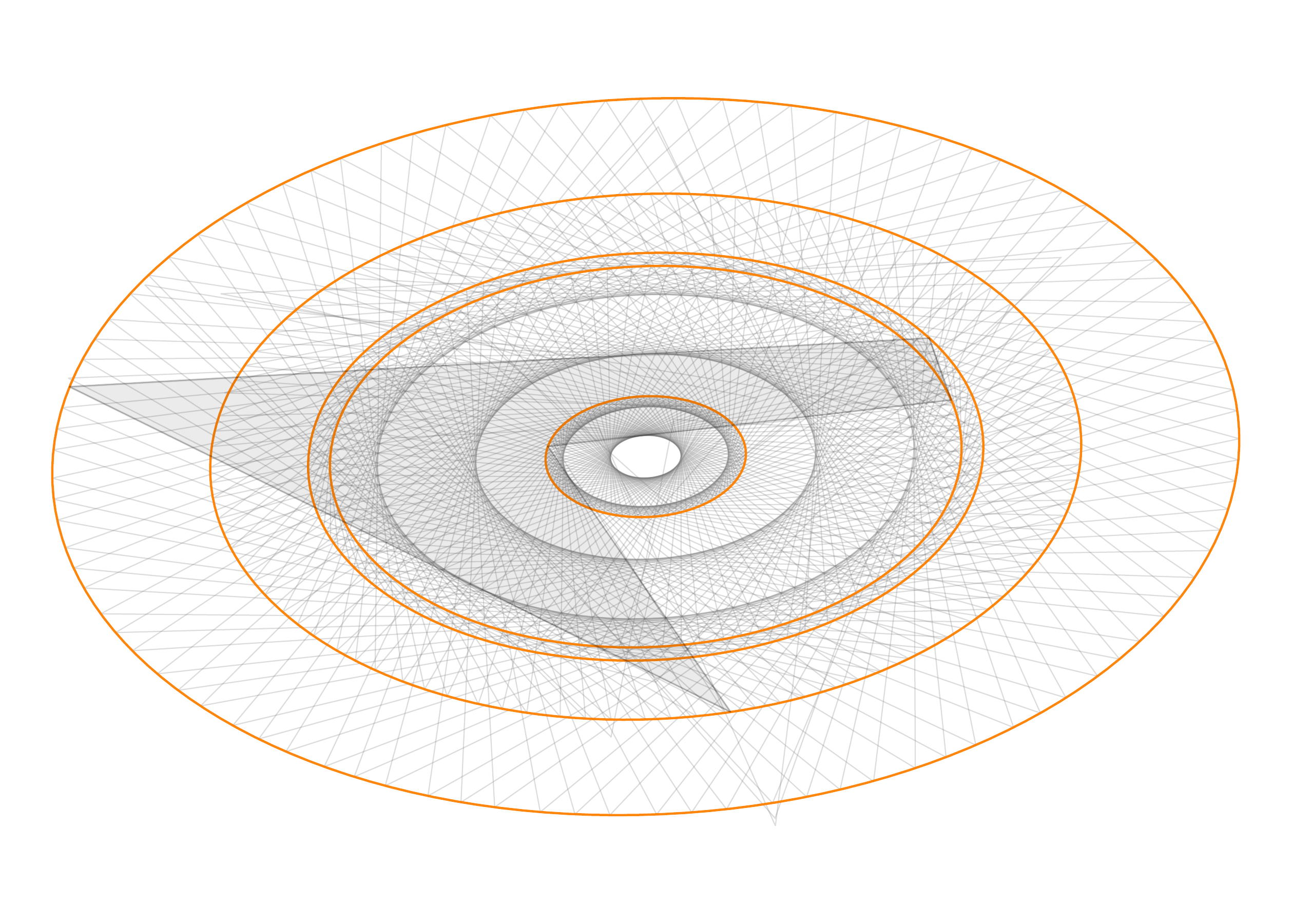}
    \hfill
    \includegraphics[width=0.23\textwidth,angle=90]{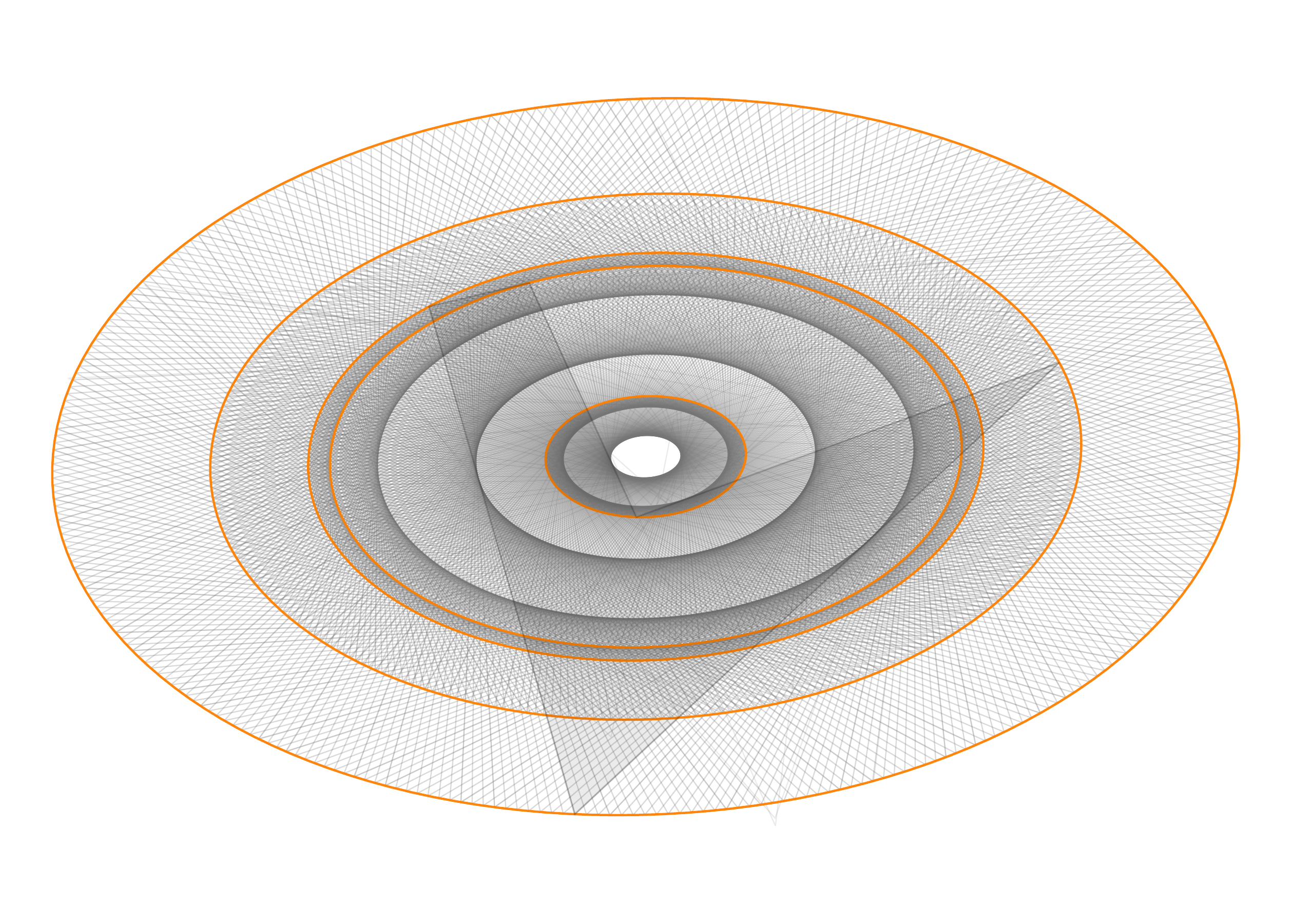}

    \caption{Elliptic asymptotics for a~nonconvex pentagon}
    \label{fig:pentagon-ellipses}
\end{figure}

\begin{figure}[htbp]
    \centering

    \includegraphics[width=0.32\textwidth]{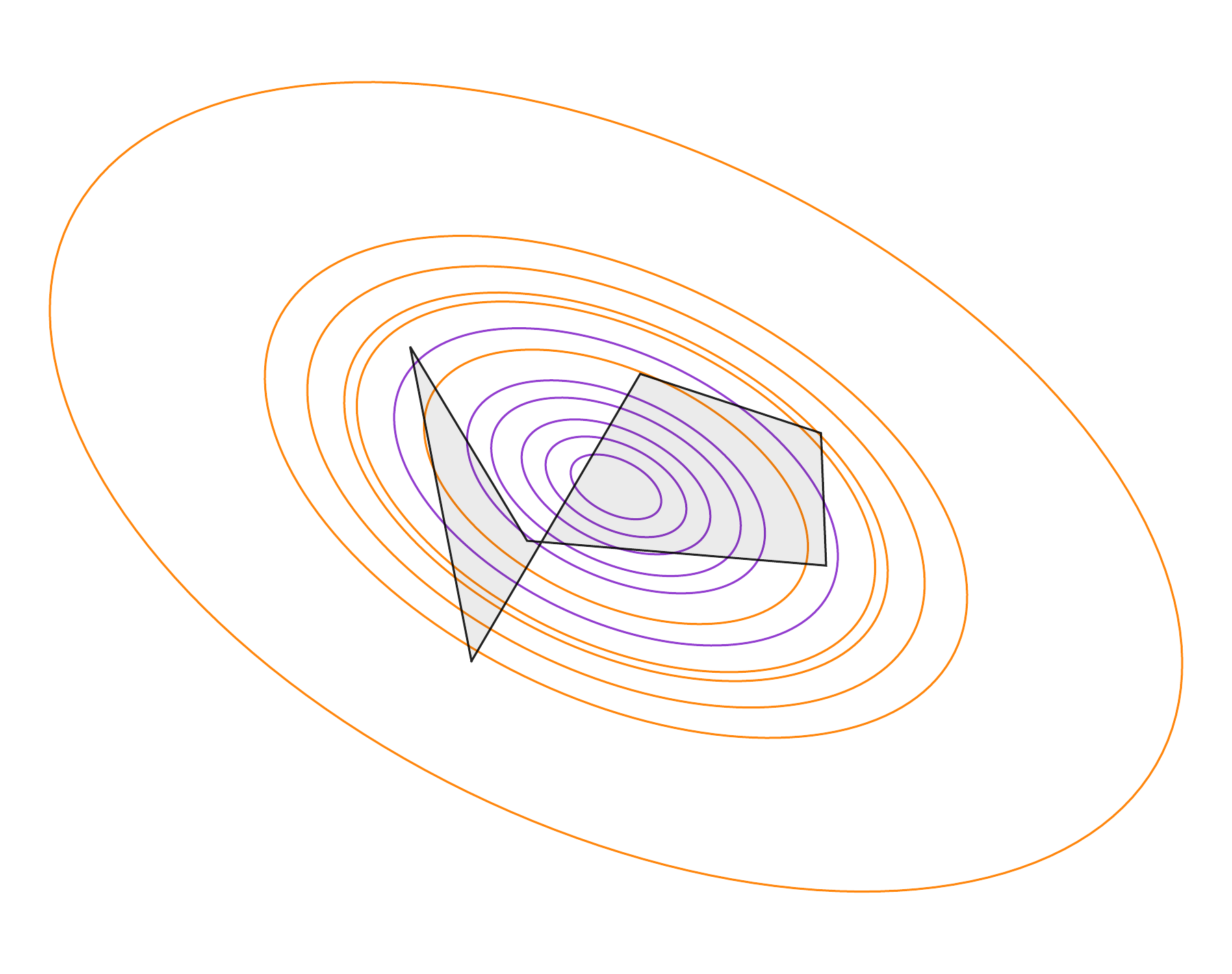}
    \hfill
    \includegraphics[width=0.32\textwidth]{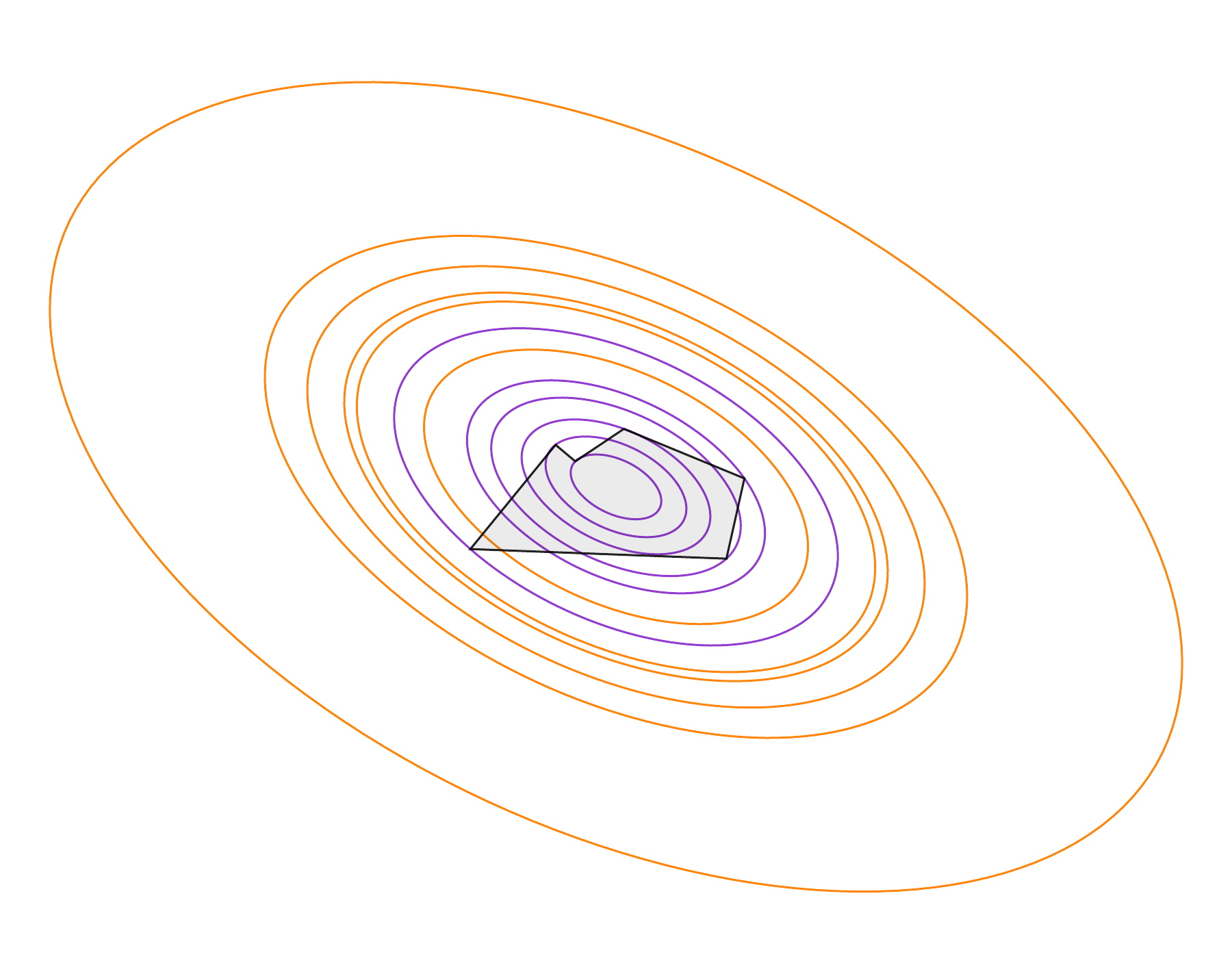}
    \hfill
    \includegraphics[width=0.32\textwidth]{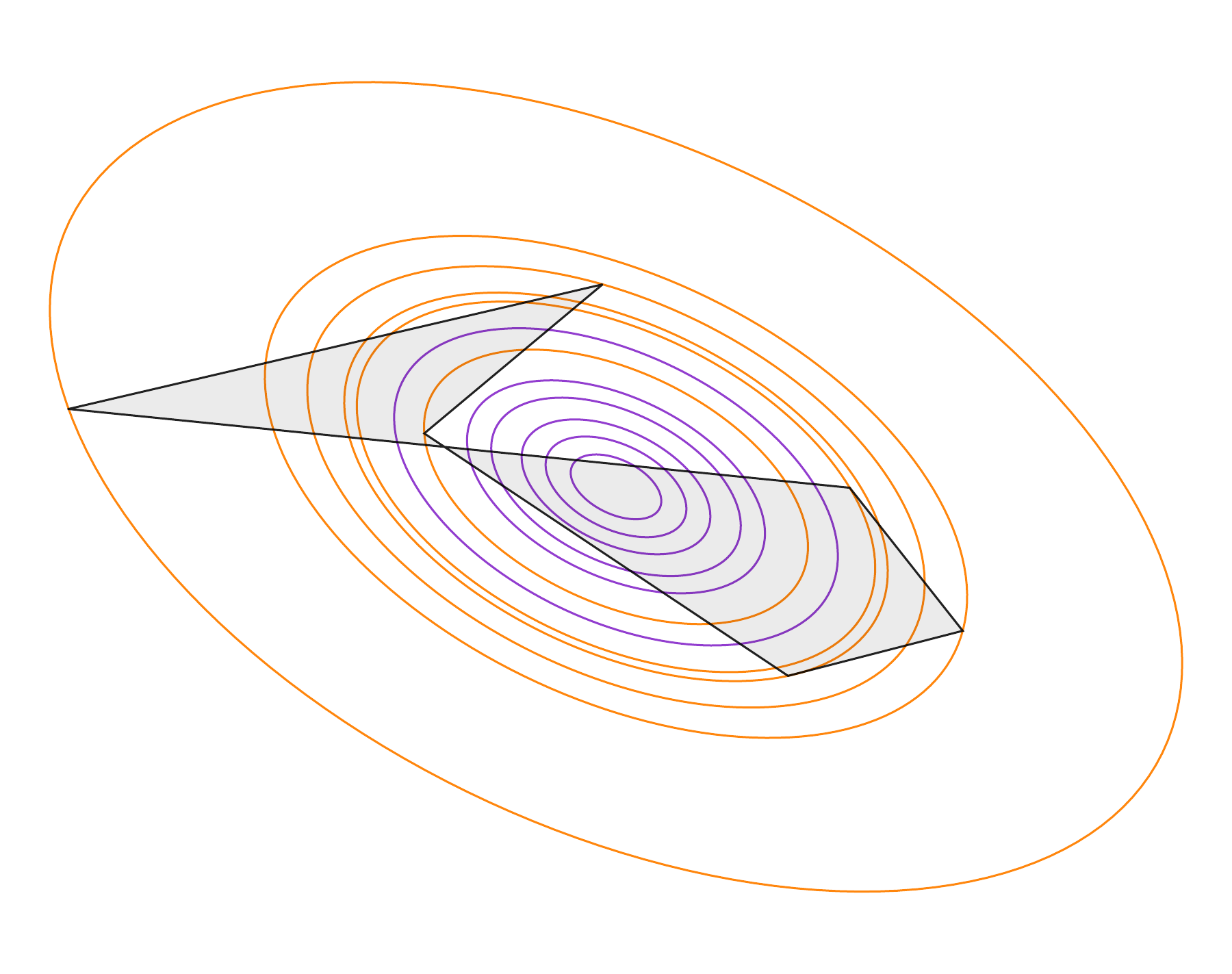}

    \vspace{0.25cm}

    \includegraphics[width=0.32\textwidth]{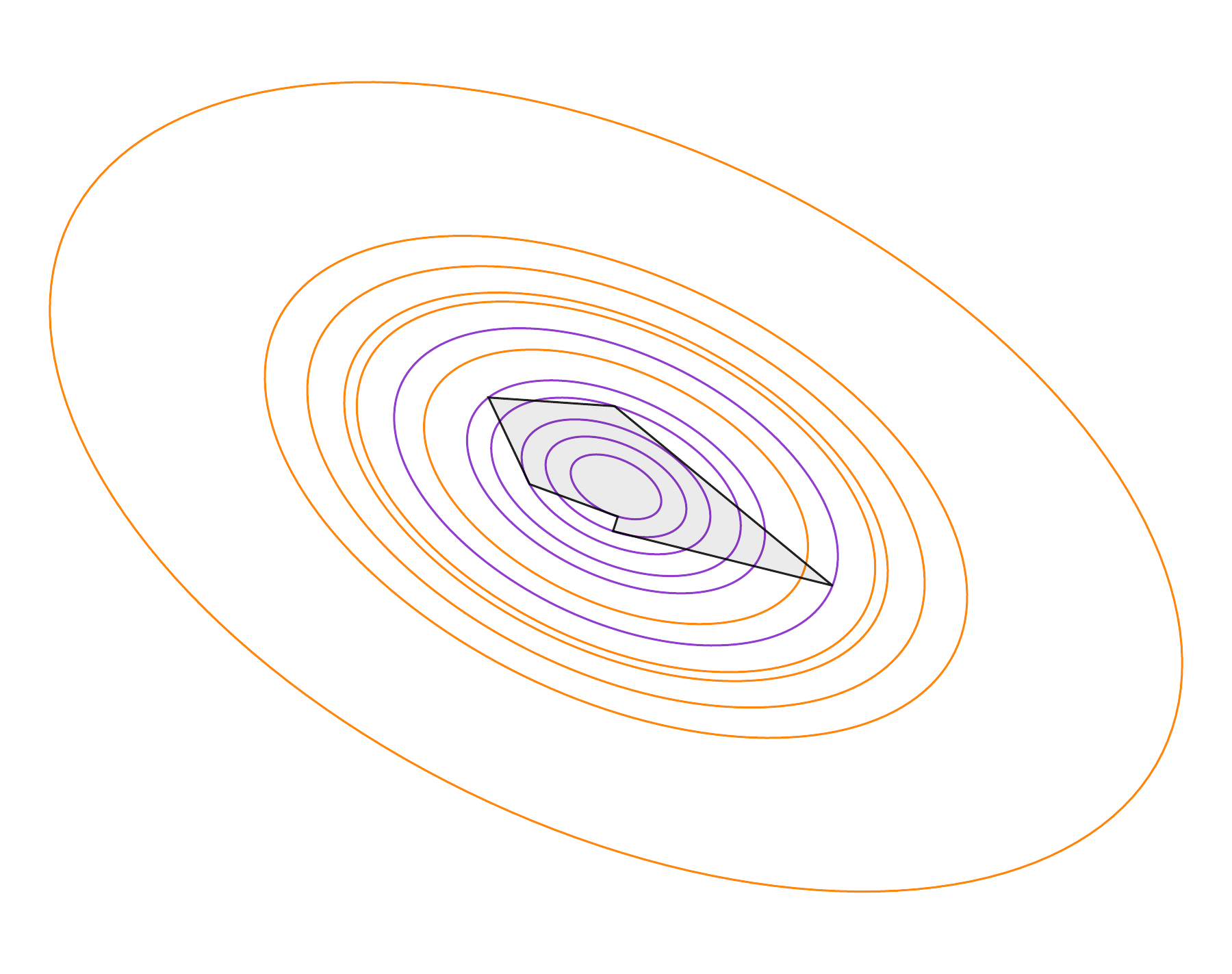}
    \hfill
    \includegraphics[width=0.32\textwidth]{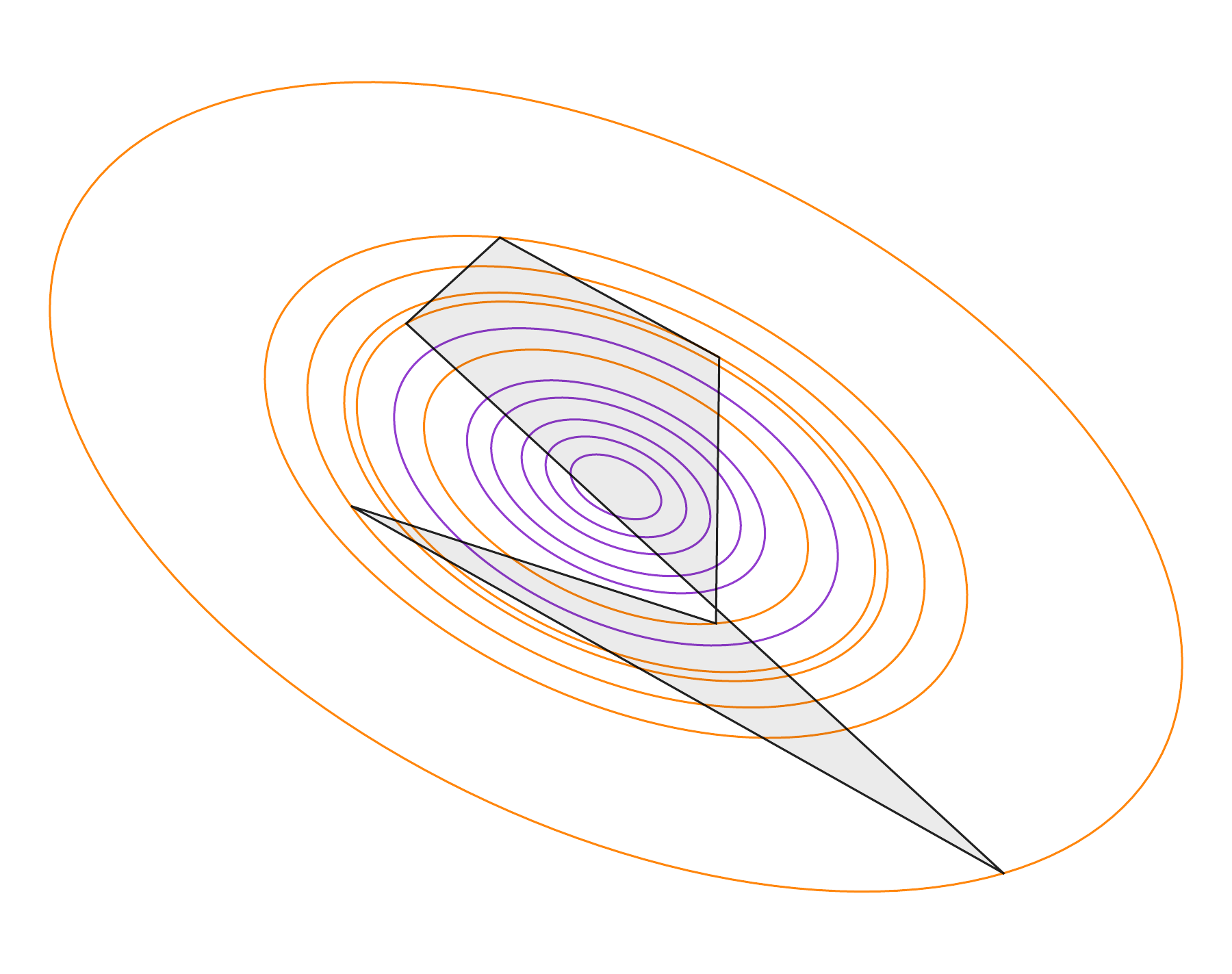}
    \hfill
    \includegraphics[width=0.32\textwidth]{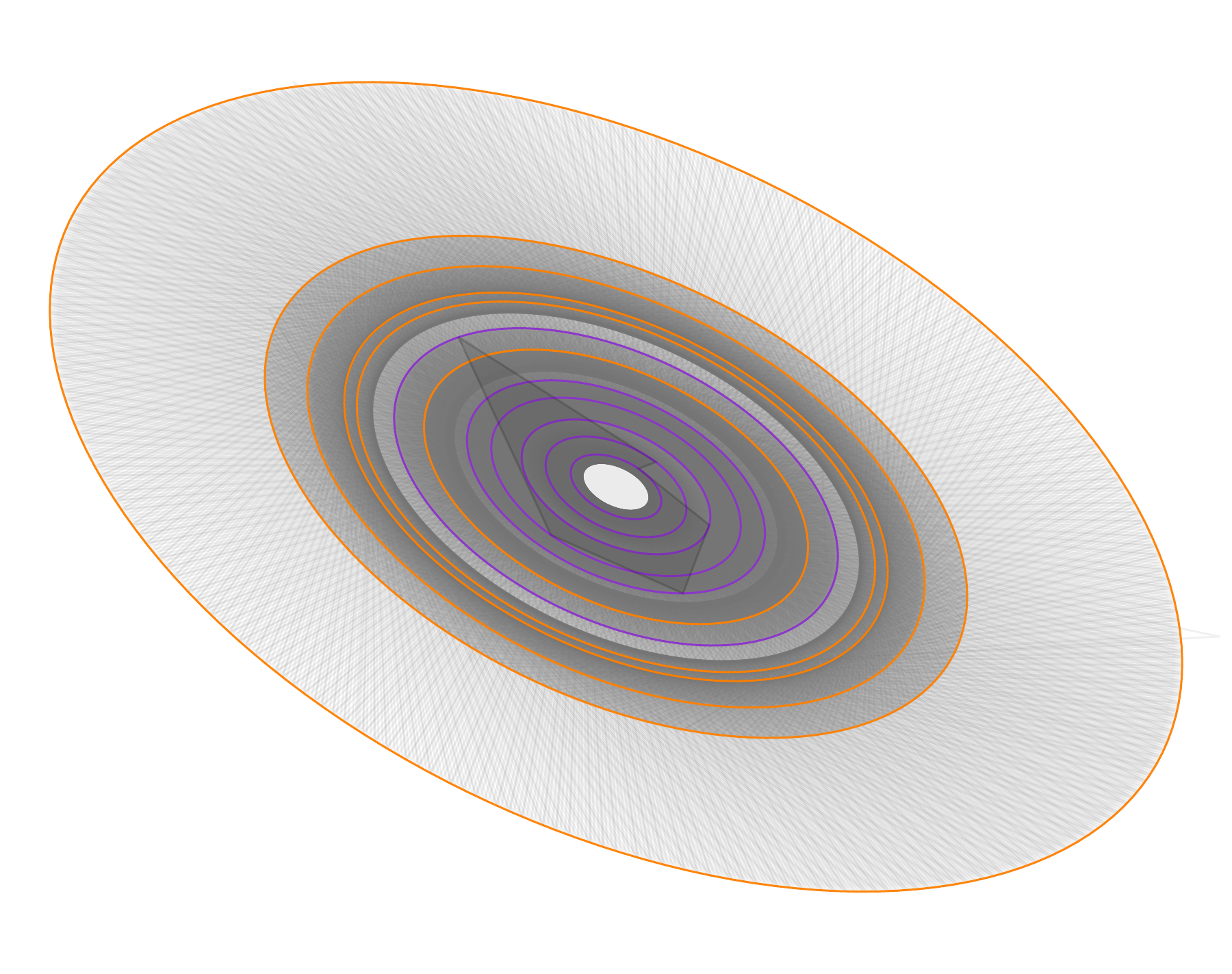}

    \caption{Two residue-class families of asymptotic ellipses for a~hexagon}
    \label{fig:hexagon-two-ellipse-families}
\end{figure}

The~heptagon in Example~\ref{ex:period-three-heptagon} has exact minimal
projective return time $m=3$, so the~same result predicts three
residue-class families whenever the~elliptic nondegeneracy hypotheses
hold.  For that example these additional hypotheses are supported
numerically.

We now illustrate the~residue-class phenomenon in the~hexagonal
case.  Since the~natural projective-periodicity relation for
hexagons involves the~second iterate, the~area-normalized dynamics
splits into two subsequences.  Figure~\ref{fig:hexagon-two-ellipse-families}
shows the~two corresponding families of ellipses.  The~even
and odd iterates share the~same asymptotic center and the~same principal
axes, but they generally lie on different levels of the~same quadratic
form.  The~first panel shows the~initial hexagon and both
families; the~next four show iterates $65$, $66$, $67$, and $68$,
while the~last panel illustrates the~first $2137$ iterates.

\begin{remark}
The~conclusion of Proposition~\ref{prop:elliptic-case} is asymptotic in
a~general projective chart.  Exact containment in fixed Euclidean ellipses
requires additional structure, for instance that the~relevant projectivity
acts affinely in the~chosen chart and preserves a~positive definite
quadratic form up to scale.  Without such an~additional assumption, the~
projective change of coordinates contributes higher-order terms, so the~
elliptic picture becomes accurate only in the~normalized limit.
\end{remark}

\begin{lemma}[Convex projective-orbit nondegeneracy]
\label{lem:convex-projective-orbit-nondegeneracy}
Let
$
Q=(q_1,\ldots,q_n),
\qquad
n\geq5,
$
be a~strictly convex labeled polygon whose vertices are listed in
cyclic order, and let $G$ be a~real projectivity.  Assume that, for
every $q\geq0$, the~polygon $G^q(Q)$ lies in the~chosen affine chart,
is strictly convex with the~induced cyclic labeling, and converges
vertexwise to a~common finite point as $q\to\infty$.

Suppose that $G$ has a~simple real eigenvalue $\lambda_1$ such that
the~moduli of all other eigenvalues are strictly smaller than
$|\lambda_1|$.  Let $X_1$ be the~eigenpoint corresponding to
$\lambda_1$.  Then the~following statements hold.
\begin{enumerate}[label=\textup{(\roman*)}]
\item The~common limit of $G^q(Q)$ is $X_1$, and every vertex of $Q$
is in regular spectral position with respect to the~dominant
eigendirection.

\item Suppose that the~remaining eigenvalues are real and can be
ordered so that
$|\lambda_1|>|\lambda_2|>|\lambda_3|$.
If $X_j$ is the~eigenpoint corresponding to $\lambda_j$, then $Q$ has
an~area-nondegenerate centered two-scale asymptotic under $G$ in the~
affine direction of the~line $X_1X_2$, with scales
$$
\alpha_q
=
\left(\frac{\lambda_2}{\lambda_1}\right)^q,
\qquad
\beta_q
=
\left(\frac{\lambda_3}{\lambda_1}\right)^q.
$$

\item Suppose that the~remaining eigenvalues form a~non-real conjugate
pair.  Then $Q$ satisfies all the~requirements in
condition~\ref{ass:elliptic-regular} of
Proposition~\ref{prop:elliptic-case}, including
$$
w_i\ne0
\quad\text{for every }i,
\qquad
\left|
\frac12\sum_i\det(w_i,w_{i+1})
\right|>0.
$$
\end{enumerate}
\end{lemma}

\begin{proof}
We first record a~sign property valid for every $n\geq5$.
Let $H$ be a~projectivity such that both $Q$ and $H(Q)$ are strictly
convex in the~chosen affine chart.  Choose a~matrix representative
$B\in\GL(3,\R)$ of $H$, put
$$
\widetilde q_i=(q_i,1)^{\top},
\qquad
h_i=e_3^{\top}B\widetilde q_i,
\quad\text{and}\quad
D_i
=
\det(\widetilde q_i,\widetilde q_{i+1},\widetilde q_{i+2}).
$$
All indices in this proof are taken modulo $n$.  The~normalized affine
lift of $H(q_i)$ is $B\widetilde q_i/h_i$, and hence
\begin{align}
\label{eq:convex-orbit-sign-product}
D_i'
&=
\det\left(
\frac{B\widetilde q_i}{h_i},
\frac{B\widetilde q_{i+1}}{h_{i+1}},
\frac{B\widetilde q_{i+2}}{h_{i+2}}
\right)
=
\frac{\det B}{h_i h_{i+1}h_{i+2}}D_i.
\end{align}
Strict convexity and cyclic labeling imply that the~signs of the~
numbers $D_i$ are independent of $i$, and the~same is true of the~
numbers $D_i'$.  It follows from
\eqref{eq:convex-orbit-sign-product} that
$\operatorname{sgn}(h_i h_{i+1}h_{i+2})$
is independent of $i$.  Comparing the~expressions corresponding to
$i$ and $i+1$ gives
$$
\operatorname{sgn}(h_i)
=
\operatorname{sgn}(h_{i+3}).
$$

If $3\nmid n$, addition by $3$ acts transitively on
$\mathbb Z/n\mathbb Z$, so all the~numbers $h_i$ have the~same sign.
If $3\mid n$, their cyclic sign pattern is $3$-periodic.  Any
nonconstant $3$-periodic sign pattern has at least
$2n/3\geq4$ cyclic sign changes.  On the~other hand, the~function
$x\mapsto e_3^{\top}B(x,1)^{\top}$
is affine linear.  Its zero set is a~line, which can cross the~boundary
of a~strictly convex polygon at most twice.  Hence its values on the~
cyclically ordered vertices can have at most two cyclic sign changes.
The~$3$-periodic pattern must therefore be constant.  Consequently,
\begin{align}
\label{eq:convex-orbit-common-denominator-sign}
\operatorname{sgn}(h_1)
=
\cdots
=
\operatorname{sgn}(h_n).
\end{align}

We apply this observation to $H=G^q$.  Let $A\in\GL(3,\R)$ represent
$G$, choose a~right eigenvector $v_1$ and a~left eigen-covector
$\varphi_1$ such that
$Av_1=\lambda_1v_1$,
$\varphi_1A=\lambda_1\varphi_1$,
$\varphi_1(v_1)=1$,
and put
$c_i=\varphi_1(\widetilde q_i)$.
The~homogeneous vertex lifts span $\R^3$, so $c_j\ne0$ for at least
one $j$.  The~projective orbit of $q_j$ therefore converges to
$[v_1]$.  Since all vertices have a~common limit, that limit is
$X_1=[v_1]$.
If $c_i=0$ for some $i$, then the~whole orbit of $q_i$ lies in the~
invariant projective line
$
\mathbb P(\ker\varphi_1).
$
This line does not contain $[v_1]$, contradicting convergence to
$X_1$.  Hence $c_i\ne0$ for every $i$.

Since $X_1$ is finite, $e_3^{\top}v_1\ne0$.  The~dominant spectral
expansion gives
\begin{align}
\label{eq:convex-orbit-dominant-denominators}
e_3^{\top}A^q\widetilde q_i
&=
\lambda_1^q(e_3^{\top}v_1)c_i
+
o(|\lambda_1|^q).
\end{align}
For every $q$, the~$n$ quantities on the~left-hand side have the~same
sign by \eqref{eq:convex-orbit-common-denominator-sign}.  Letting
$q\to\infty$ in
\eqref{eq:convex-orbit-dominant-denominators} shows that all $c_i$ have the~
same sign.  Replacing simultaneously $(\varphi_1,v_1)$ by
$(-\varphi_1,-v_1)$ if necessary, we may assume that $c_i>0$ for every $i$.

Consider the~dominant eigenchart
$$
\Psi([z])
=
\frac{z}{\varphi_1(z)}-v_1
\in\ker\varphi_1
$$
and put
$$
\zeta_i=\Psi(q_i),
\qquad
\Xi=(\zeta_1,\ldots,\zeta_n).
$$
Because the~numbers $c_i$ are positive, $\varphi_1$ is positive on
the~standard affine lifts of the~whole convex hull of $Q$.  Thus
$\Psi$ has no pole on this convex hull.  A~projective map whose affine
denominator has constant sign maps line segments to line segments.
It follows that $\Xi$ is again a~strictly convex affine $n$-gon.
In particular,
$\sArea(\Xi)\ne0$.
Moreover, its vertex barycenter
$\bar\zeta=\frac1n\sum_i\zeta_i$
lies in the~interior of $\conv(\Xi)$, and hence
\begin{align}
\label{eq:convex-orbit-eigenchart-barycenter}
\zeta_i-\bar\zeta\ne0
\quad\text{for every }i.
\end{align}

Suppose first that the~spectrum is real.  Choose eigenvectors
$v_2,v_3$ corresponding to $\lambda_2,\lambda_3$ and write
$\zeta_i=x_i v_2+y_i v_3$.
In these coordinates the~action of $G$ is exactly
$$
(x,y)\mapsto (a x,b y),
\qquad
a=\frac{\lambda_2}{\lambda_1},
\qquad
b=\frac{\lambda_3}{\lambda_1},
$$
where
$|b|<|a|<1$.
Let $\Phi=\Psi^{-1}$ be the~change from the~eigenchart to the~original
affine chart.  Choose an~affine vector $u$ parallel to $X_1X_2$ and a~
vector $v$ transverse to $u$.  Since $\Phi$ maps the~line $y=0$ to
the~line $X_1X_2$, there exist real constants
$\alpha,\beta,\gamma,c,d$, with $\alpha\gamma\ne0$, such that
\begin{align}
\label{eq:convex-orbit-real-eigenchart}
\Phi(x,y)-X_1
&=
\frac{(\alpha x+\beta y)u+\gamma yv}
     {1+cx+dy}.
\end{align}
Substituting $(a^qx_i,b^qy_i)$ into
\eqref{eq:convex-orbit-real-eigenchart} and subtracting the~vertex
barycenter gives, uniformly in $i$,
$$
G^q(q_i)-C(G^q(Q))
=
\bigl(
a^q\alpha(x_i-\bar x)+o(|a|^q)
\bigr)u
+
\bigl(
b^q\gamma(y_i-\bar y)+o(|b|^q)
\bigr)v,
$$
where
$\bar x=\frac1n\sum_i x_i$ and
$\bar y=\frac1n\sum_i y_i$.
This is a~centered two-scale asymptotic with scales
$$
\alpha_q=a^q,
\qquad
\beta_q=b^q,
\qquad
\left|\frac{\beta_q}{\alpha_q}\right|
=
\left|\frac ba\right|^q\rightarrow0.
$$
Its leading mixed-area coefficient is
$$
B_Q
=
\alpha\gamma\det(u,v)\,\sArea(\Xi),
$$
where $\sArea(\Xi)$ is computed in the~$(x,y)$-coordinates.  Every
factor on the~right-hand side is nonzero.  This proves \textup{(ii)}.

Suppose now that the~subordinate eigenvalues form a~non-real conjugate
pair.  Choose real coordinates on $\ker\varphi_1$ in which the~induced
action of $G$ is
$$
z\mapsto\tau R_0z,
\qquad
0<\tau<1,
$$
where $R_0$ is a~Euclidean rotation.  If
$
J=\mathrm{D}\Phi_0,
$
then $J$ is invertible and
$$
\Phi(z)=X_1+Jz+O(\|z\|^2).
$$
Consequently,
\begin{align}
\label{eq:convex-orbit-complex-eigenchart}
G^q(q_i)
=
\Phi(\tau^qR_0^q\zeta_i)
=
X_1+\tau^qJ R_0^q\zeta_i+O(\tau^{2q})
=
X_1+\tau^qR^q\eta_i+O(\tau^{2q}),
\end{align}
where
$R=J R_0J^{-1}$ and
$\eta_i=J\zeta_i$.
The~matrix $R$ is precisely
$
\tau^{-1}\mathrm{D}f_{X_1}
$
in the~notation of Proposition~\ref{prop:elliptic-case}, and
$\det R=1$.  After centering, the~leading vectors are
$$
w_i
=
\eta_i-\frac1n\sum_j\eta_j
=
J(\zeta_i-\bar\zeta).
$$
Equation~\eqref{eq:convex-orbit-eigenchart-barycenter} and invertibility of
$J$ give $w_i\ne0$ for every $i$.  Finally,
$$
\left|
\frac12\sum_i\det(w_i,w_{i+1})
\right|
=
|\det J|\,|\sArea(\Xi)|
>0.
$$
Thus all the~requirements in
condition~\ref{ass:elliptic-regular} hold, proving \textup{(iii)}.
\end{proof}

\begin{theorem}[Generic line-or-ellipse dichotomy for convex hexagons]
\label{thm:convex-hexagon-dichotomy}
Let $\mathcal C_6$ denote the~space of strictly convex labeled
hexagons whose vertices are listed in cyclic order.  For
$P\in\mathcal C_6$, put
$$
G_P=L_P-3I,
\qquad
F(P)=\mathcal R(P)+2\mathcal S(P)-8,
$$
and define
$$
\mathcal U_6
=
\{P\in\mathcal C_6:\Delta_P<0\}
\cup
\{P\in\mathcal C_6:\Delta_P>0,\ F(P)\ne0\}.
$$
Then $\mathcal U_6$ is open and dense in $\mathcal C_6$.  For every
$P\in\mathcal U_6$, exactly one of the~following alternatives holds.
\begin{enumerate}[label=\textup{(\roman*)}]
\item If $\Delta_P>0$, the~eigenvalues of $G_P$ are real and can be
ordered so that
$|\mu_1|>|\mu_2|>|\mu_3|$.
Let $X_j$ be the~eigenpoint corresponding to $\mu_j$, and let $u$ be
the~affine direction of the~line $X_1X_2$.  Then
$$
\max_i
\dist\bigl((\That^k(P))_i,C(P)+\R u\bigr)
\rightarrow0
\quad\text{and}\quad
\diam\bigl(\That^k(P)\bigr)
\rightarrow\infty.
$$

\item If $\Delta_P<0$, the~spectrum of $G_P$ consists of one strictly
dominant real eigenvalue and a~subdominant non-real conjugate pair.
All the~hypotheses of Proposition~\ref{prop:elliptic-case} are then
satisfied automatically with $m=2$ and $s=0$.  Consequently, the~even
and odd area-normalized subsequences satisfy the~elliptic asymptotics
of Proposition~\ref{prop:elliptic-case}; their limiting ellipses are
given by Proposition~\ref{prop:predicting-ellipses}.
\end{enumerate}
\end{theorem}

\begin{proof}
The~set $\mathcal U_6$ is open because $\Delta_P$ and $F(P)$ depend
continuously on the~vertices throughout $\mathcal C_6$.  Its complement
is contained in
$$
\{P\in\mathcal C_6:\Delta_PF(P)=0\}.
$$
The~invariants $\mathcal S$ and $\mathcal R$ are rational functions of
the~vertex coordinates, with nonvanishing denominators on
$\mathcal C_6$.  After clearing these denominators,
$\Delta_PF(P)$ therefore has a~polynomial numerator.

This numerator is not the~zero polynomial.  Indeed, for the~hexagon
in Example~\ref{ex:convex-hexagon-complex-spectrum},
$\Delta_P<0$
and
\begin{align}
F(P)
&=
\frac{8604}{25}
+
2\cdot\frac{461}{5}
-8
=
\frac{13014}{25}
\ne0.
\end{align}
The~zero set of a~nonzero real polynomial has empty interior.
Consequently, $\mathcal U_6$ is dense in $\mathcal C_6$.

We now prove the~dynamical conclusions.  Strict convexity implies that
all forward pentagram iterates are strictly convex and that their
vertices converge to a~common finite limit point
\cite{GlickLimit,SchwartzPentagram}.  Glick's homogeneous identity
\cite{GlickLimit} gives
$G_P(P)=\T^2(P)$.
Since $\T^2(P)$ is nondegenerate, this identity forces $G_P$ to have
rank three.  Thus $G_P$ represents a~projectivity.  For
$Q_0=P$,
$Q_1=\T(P)$,
Lemma~\ref{lem:cyclic-relabel} gives
\begin{align}
\label{eq:convex-hexagon-residue-orbits}
G_P^q(Q_r)
&=
\T^{2q+r}(P),
\qquad
r=0,1.
\end{align}
Hence both residue polygons satisfy the~convexity and convergence
hypotheses of
Lemma~\ref{lem:convex-projective-orbit-nondegeneracy}.

Assume first that $\Delta_P>0$.  By the~definition of
$\mathcal U_6$, one has $F(P)\ne0$.  Proposition~
\ref{prop:hexagon-spectral-diagram} and Corollary~
\ref{cor:convex-hexagon-spectral-restriction} imply that the~three
eigenvalues are real and have pairwise distinct moduli.  They may
therefore be ordered as
$|\mu_1|>|\mu_2|>|\mu_3|$.
The~dominant eigenpoint $X_1$ is finite by
Lemma~\ref{lem:convex-projective-orbit-nondegeneracy}.  Hence the~line
$X_1X_2$ is not the~line at infinity and has a~well-defined affine
direction $u$.  The~same lemma shows that both $Q_0$ and $Q_1$ have
area-nondegenerate centered two-scale asymptotics in this direction.
All the~hypotheses of Theorem~\ref{thm:spectral-flattening} are thus
satisfied with $m=2$ and $s=0$, proving \textup{(i)}.

Assume now that $\Delta_P<0$.  By
Corollary~\ref{cor:convex-hexagon-spectral-restriction}, the~unique
real eigenvalue is strictly dominant over the~non-real conjugate pair.
Applying part~\textup{(iii)} of
Lemma~\ref{lem:convex-projective-orbit-nondegeneracy} to $Q_0$ and
$Q_1$ verifies regular spectral position, the~nonvanishing of all
leading centered vectors, and the~positivity of both leading area
coefficients.  Propositions~\ref{prop:elliptic-case}
and~\ref{prop:predicting-ellipses} now give \textup{(ii)}.
\end{proof}

\begin{proposition}[Projectively regular convex Poncelet polygons]
\label{prop:projectively-regular-poncelet}
Let $n\geq 5$, and let
$$
R_n=(v_0,\ldots,v_{n-1}),
\qquad
v_j=
\left(
\cos\frac{2\constpi j}{n},
\sin\frac{2\constpi j}{n}
\right),
$$
be the~regular $n$-gon inscribed in the~unit circle.  Let $\Phi$ be a~real
projective transformation defined on a~neighborhood of the~closed unit
disk, and put $P=\Phi(R_n)$.
Set
$$
\rho_n=
\frac{\cos(2\constpi/n)}{\cos(\constpi/n)}.
$$
Then the~following statements hold.

\begin{enumerate}[label=\textup{(\roman*)}]
\item
The~polygon $P$ is a~strictly convex Poncelet polygon.  More precisely, its vertices lie
on $\Phi(S^1)$ and its sides are tangent to
$
\Phi\!\left(\cos\frac{\constpi}{n}\,S^1\right).
$
In particular, it is projectively regular in the~sense of
Definition~\ref{def:projectively-regular}.

\item
For every $k\geq0$ and every $j$, one has
$$
\T^k(P)_j
=
\Phi\!\left(
\rho_n^k R_{3k\constpi/n}v_j
\right),
$$
where $R_\theta$ denotes the~Euclidean rotation through the~angle
$\theta$.  In particular,
$$
\T^2(P)=\Sigma_3(G(P)),
$$
where
$G=\Phi D_n \Phi^{-1}$ and
$D_n[x:y:z]=[\rho_n^2x:\rho_n^2y:z]$.
Consequently, up to a~common nonzero scalar factor,
$$
\Spec(G)
=
\{1,\rho_n^2,\rho_n^2\}.
$$
Thus the~cyclically aligned Darboux--Schwartz projectivity has three real
eigenvalues, and the~two subdominant eigenvalues coincide, for every
$n\geq5$.

\item
Let
$c=\Phi(0)$,
$J_{\Phi}=\mathrm{D}\Phi(0)$,
and let
$P_k=\That^k(P)$
be the~area-normalized iterates.  Then, uniformly in $j$,
$$
p_{k,j}-C(P)
=
\kappa J_{\Phi} R_{3k\constpi/n}v_j+O(\rho_n^k),
$$
where
$$
\kappa
=
\sqrt{
\frac{\Area(P)}
{|\det J_{\Phi}|\,\Area(R_n)}
},
\qquad
\Area(R_n)
=
\frac n2\sin\frac{2\constpi}{n}.
$$
Consequently, all asymptotic conics corresponding to all vertices and
all residue classes coincide with the~single ellipse
$$
(x-C(P))^{\top}
J_{\Phi}^{-\top}J_{\Phi}^{-1}
(x-C(P))
=
\kappa^2.
$$
\end{enumerate}

\end{proposition}

\begin{proof}
The~regular $n$-gon is inscribed in the~unit circle, and each of its
sides is tangent to the~concentric circle of radius
$\cos(\constpi/n)$.  Projective transformations preserve incidence and
tangency.  Therefore $P$ is inscribed in $\Phi(S^1)$ and circumscribed
about $\Phi(\cos(\constpi/n)S^1)$.  Since the~denominator of $\Phi$
has constant sign on the~closed unit disk, $\Phi$ maps its line
segments to line segments and its boundary to a~strictly convex conic.
This proves \textup{(i)}.

Let $\delta=2\constpi/n$.
The~chord joining $v_j$ to $v_{j+2}$ has equation
$\langle x,v_{j+1}\rangle=\cos\delta$,
whereas the~chord joining $v_{j+1}$ to $v_{j+3}$ has equation
$\langle x,v_{j+2}\rangle=\cos\delta$.
Their intersection lies on the~angular bisector of the~two normal
directions, hence on the~ray making the~angle
$\left(j+\frac32\right)\delta$ with the~positive horizontal axis.  Its distance from the~origin is
therefore determined by $r\cos\frac\delta2=\cos\delta.$
Thus
$r=\rho_n$
and
$\T(R_n)_j
=
\rho_n R_{3\constpi/n}v_j$.
Iteration gives
$\T^k(R_n)_j
=
\rho_n^kR_{3k\constpi/n}v_j$.
Projective naturality of the~pentagram map now yields
$\T^k(P)_j=
\Phi\!\left(
\rho_n^kR_{3k\constpi/n}v_j
\right)$.
Since
$R_{6\constpi/n}v_j=v_{j+3}$,
we obtain
$\T^2(P)_j=\Phi(\rho_n^2v_{j+3})=G(p_{j+3})$,
which is precisely
$\T^2(P)=\Sigma_3(G(P))$.
The~matrix of $D_n$ in homogeneous coordinates is
$\operatorname{diag}(\rho_n^2,\rho_n^2,1)$,
so the~spectral statement follows by conjugacy.

It remains to determine the~area-normalized asymptotics.  Since
$0<\rho_n<1$ for $n\geq5$, Taylor expansion of $\Phi$ at the~attracting
point $0$ of $D_n$ gives
$$
\Phi\!\left(
\rho_n^kR_{3k\constpi/n}v_j
\right)
=
c+
\rho_n^k J_{\Phi} R_{3k\constpi/n}v_j
+O(\rho_n^{2k}).
$$
Because
$\sum_{j=0}^{n-1}v_j=0$,
the~barycenter satisfies
$C(\T^k(P))=c+O(\rho_n^{2k})$.
Moreover, multilinearity of the~determinant in the~polygonal area
formula gives
$$
\Area(\T^k(P))
=
|\det J_{\Phi}|\,\rho_n^{2k}\Area(R_n)
+O(\rho_n^{3k}).
$$
Using equation~\eqref{eq:normalization-classical-iterates}, we conclude that
$$
p_{k,j}-C(P)
=
\sqrt{
\frac{\Area(P)}
{\Area(\T^k(P))}
}
\bigl(
\T^k(P)_j-C(\T^k(P))
\bigr)
=
\kappa J_{\Phi} R_{3k\constpi/n}v_j+O(\rho_n^k).
$$
Finally, if
$y=\kappa J_{\Phi} R_{3k\constpi/n}v_j$,
then
$$
y^{\top}J_{\Phi}^{-\top}J_{\Phi}^{-1}y
=
\kappa^2\|v_j\|^2
=
\kappa^2.
$$
Thus every vertex and every residue class has the~same asymptotic
ellipse.
\end{proof}

\begin{corollary}[Circular asymptotics for projective images of regular polygons]
\label{cor:poncelet-circular-asymptotics}
Under the~assumptions and notation of
Proposition~\ref{prop:projectively-regular-poncelet}, the~common
asymptotic conic is a~Euclidean circle if and only if the~derivative
$J_{\Phi}=\mathrm{D}\Phi(0)$ is a~Euclidean similarity.  Equivalently, there is a~number
$\sigma>0$ such that
$J_{\Phi}J_{\Phi}^{\top}=\sigma^2I$.
In this case the~common circle is centered at $C(P)$ and has equation
$$
\|x-C(P)\|^2
=
\frac{\Area(P)}{\Area(R_n)}
=
\frac{2\Area(P)}{n\sin(2\constpi/n)}.
$$
\end{corollary}

\begin{proof}
By Proposition~\ref{prop:projectively-regular-poncelet}, the~common
asymptotic conic has quadratic matrix $J_{\Phi}^{-\top}J_{\Phi}^{-1}$.
It is a~Euclidean circle if and only if this matrix is a~positive
multiple of the~identity.  Since $J_{\Phi}$ is invertible, this is equivalent
to $J_{\Phi}J_{\Phi}^{\top}=\sigma^2I$
for some $\sigma>0$, or equivalently to $J_{\Phi}=\sigma Q$ with
$Q\in O(2)$.  In that case $|\det J_{\Phi}|=\sigma^2$, and the~ellipse equation
from Proposition~\ref{prop:projectively-regular-poncelet} becomes
$$
\sigma^{-2}\|x-C(P)\|^2
=
\frac{\Area(P)}{\sigma^2\Area(R_n)}.
$$
The~asserted circle equation follows.
\end{proof}

\begin{remark}[Shifted and unshifted spectra]
The~repeated real spectrum above belongs to the~projectivity obtained
after removing the~natural cyclic shift by three vertices.  If one
instead uses the~unshifted projectivity
$$
\widetilde G
=
\Phi\widetilde D_n\Phi^{-1},
\qquad
\widetilde D_n
=
\begin{pmatrix}
\rho_n^2R_{6\constpi/n}&0\\
0&1
\end{pmatrix},
$$
so that $\T^2(P)=\widetilde G(P)$
without relabelling, then
$$
\Spec(\widetilde G)
=
\left\{
1,
\rho_n^2\conste^{6\constpi\consti/n},
\rho_n^2\conste^{-6\constpi\consti/n}
\right\}.
$$
Thus the~two subdominant eigenvalues are non-real for every $n\geq5$
except $n=6$; in the~exceptional case they coincide with the~real value
$-\rho_6^2=-\frac13$.
After the~cyclic shift is removed, the~spectrum is $\{1,\rho_n^2,\rho_n^2\}$
for every $n\geq5$.

More generally, if one writes the~shifted relation as
$\T^2(P)=\Sigma_s(G_s(P))$, then the~rotation angle of $G_s$ is
$(6-2s)\constpi/n$.  In particular, for
$n=6$ one may take $s=1$ to obtain a~non-real conjugate pair and hence a~
nondegenerate rotating spectral representation of the~same common
asymptotic ellipse.
\end{remark}

\section{Exact projective returns and numerical spectral evidence}
\label{sec:examples-evidence}

\noindent We collect here two exact projectively periodic examples with
minimal return times three and four, followed by a~larger-polygon test case in which the~conserved operator
does not generate the~orbit.
The~first two examples illustrate the~asymptotic mechanisms analyzed
above, but at different levels of verification.  Both return identities
and their minimality are exact.  The~period-three example includes
high-precision uniform estimates, but they are reported as numerical
evidence rather than formal interval certificates.  The~period-four
example has numerically verified nondegeneracy only along the~computed
finite orbit segment.
The~last example motivates the~conjectural spectral organization for
general polygons with $n\geq7$.

\subsection{A~projective period-three heptagon}

\begin{example}[A~projective period-three heptagon]
\label{ex:period-three-heptagon}
A~fixed projectivity generates the~third pentagram iterate of the~
following heptagon, but no smaller positive iterate.  Let
$$
\begin{aligned}
p_1&=(0,0), & p_2&=(1,0), & p_3&=(1,1),&
p_4&=(0,1), & p_5&=\left(\frac35,\frac25\right),
\end{aligned}
$$
$$
\begin{aligned}
p_6&=\left(\frac{10+63\sqrt5}{359},
          \frac{170-6\sqrt5}{359}\right), &
p_7&=\left(\frac{6\sqrt5-5}{31},
          \frac{13+3\sqrt5}{31}\right)
\end{aligned}
$$
(see the first panel of Figure~\ref{fig:period-three-heptagon}).

\begin{figure}[htbp]
    \centering
    \includegraphics[width=0.31\textwidth]{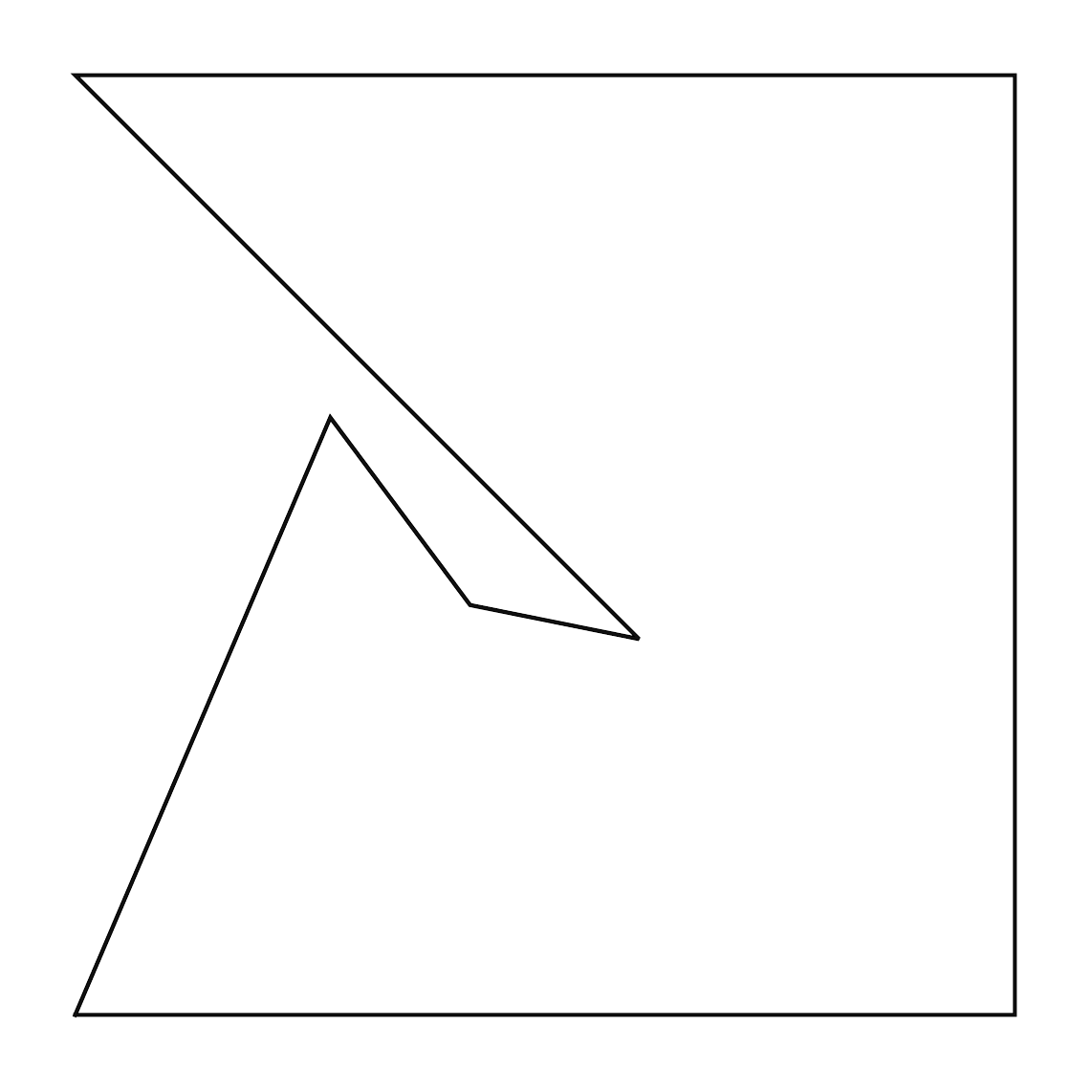}
    \hfill
    \includegraphics[width=0.31\textwidth]{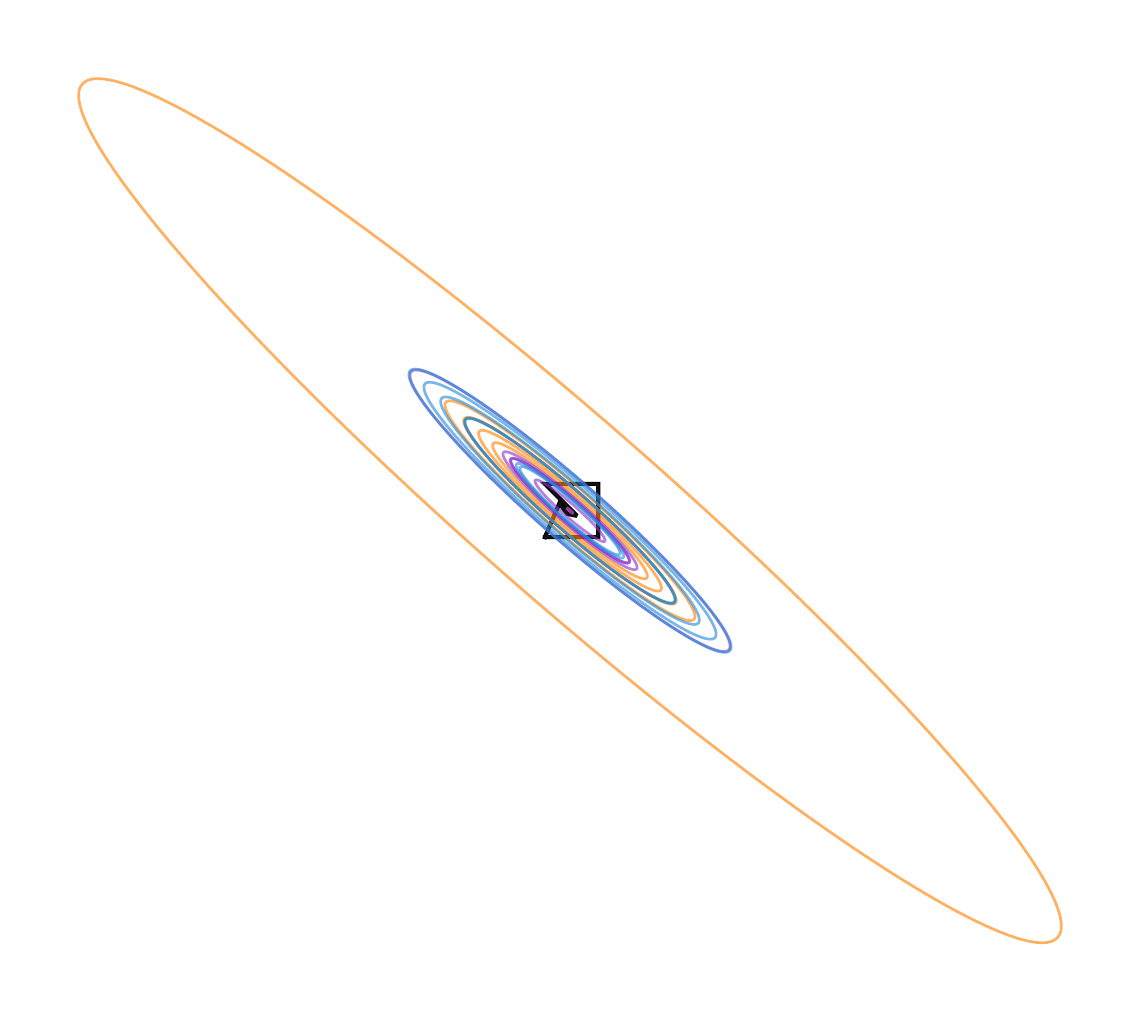}
    \hfill
    \includegraphics[width=0.31\textwidth]{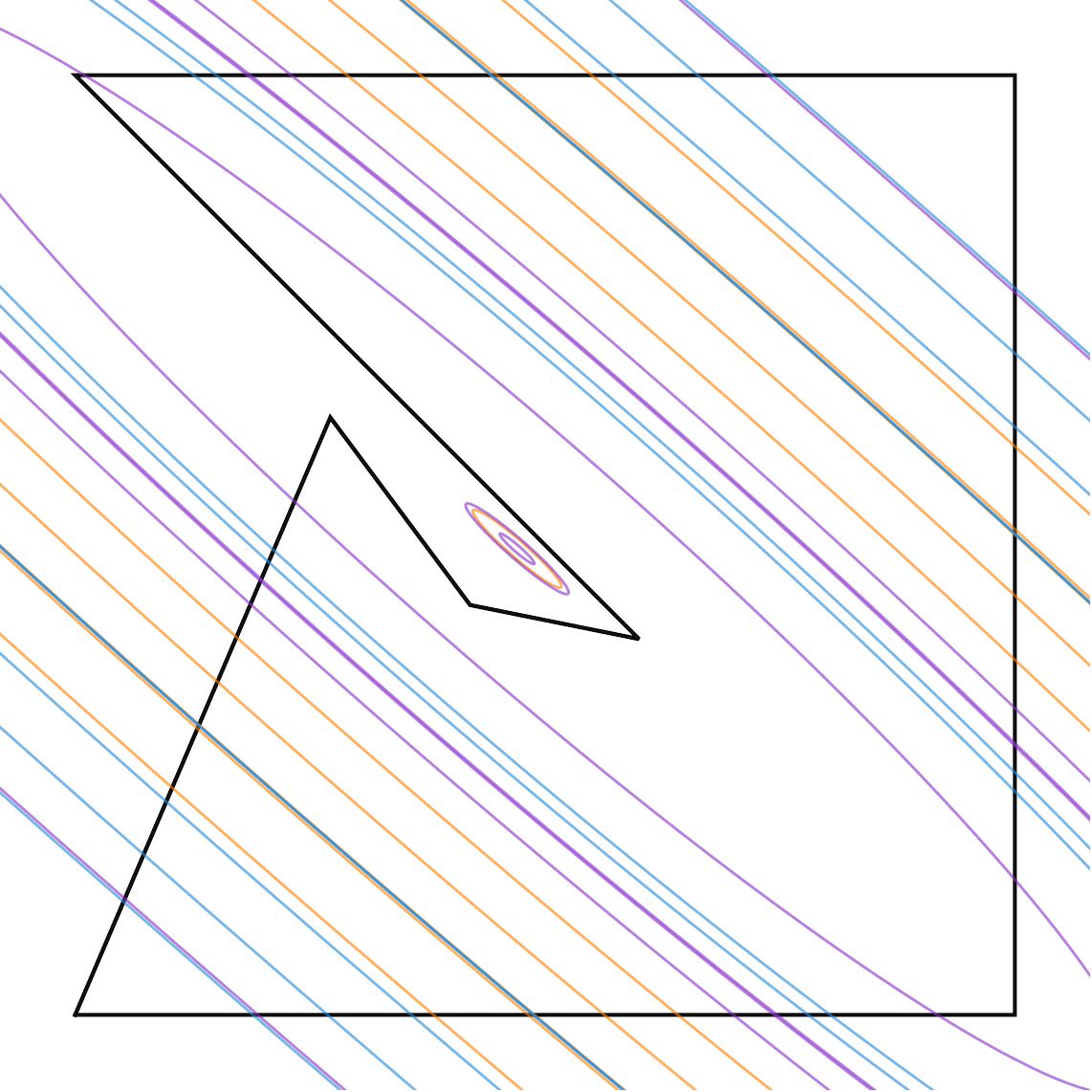}

    \caption{A~period-three heptagon and its asymptotic ellipses}
    \label{fig:period-three-heptagon}
\end{figure}

Define the~following nonsingular matrix:
$$
G=
\begin{pmatrix}
-\dfrac{19+9\sqrt5}{22}
&-\dfrac{51+67\sqrt5}{132}
& \dfrac{6+4\sqrt5}{33}\\[2mm]
-\dfrac{19+9\sqrt5}{22}
&-\dfrac{6+4\sqrt5}{11}
& \dfrac{7+\sqrt5}{22}\\[2mm]
-\dfrac{37+21\sqrt5}{22}
&-\dfrac{45+41\sqrt5}{44}
&1
\end{pmatrix}.
$$

All computations take place exactly in $\mathbb Q(\sqrt5)$.  Direct
substitution gives $\T^3(P)_i=G(p_i)$ for every $i$,
so that
$\T^3(P)=G(P)$
with the~displayed labeling and without a~cyclic shift.  The~return
time is minimal.  Indeed, write an~unknown projectivity as a
$3\times3$ matrix $K$.  For each correspondence
$p_i\mapsto\T^d(P)_{i+s}$, the~equations $K\widetilde p_i\times
\widetilde{\T^d(P)_{i+s}}=0$
give a~homogeneous linear system in the~nine entries of $K$.  Its
coefficient matrix is taken to have all $3n$ cross-product rows (only
two rows per correspondence are independent).  It has rank $9$ for
every $d=1,2$ and every
$s=0,\ldots,6$.  For $d=3$ it has rank $8$ only when $s=0$, and its
one-dimensional kernel is generated by the~displayed invertible matrix
$G$.  Hence there is no projectivity $K$ and no cyclic shift $\Sigma_s$
such that $\T^d(P)=\Sigma_s(K(P))$
for $d=1,2$, and the~minimal projective return time is $m=3$.

The~characteristic polynomial is
$$
\chi_G(t)
=
t^3+\frac{9+17\sqrt5}{22}t^2
-\frac{35+27\sqrt5}{242}t
+\frac{-17+18\sqrt5}{1331}.
$$
Hence $G$ has one strictly dominant real eigenvalue and one non-real conjugate pair:
$$
\lambda_0\approx-2.310784,
\qquad
\lambda_{\pm}
\approx0.086911
\pm0.002359\consti.
$$
An~exact solution of the~dominant real eigenvector equations subject to
the~line-at-infinity condition gives only the~zero vector, so the~
dominant real eigenpoint is finite; numerically,
$$
X_0\approx(0.486635,0.467968).
$$

We next check the~remaining hypotheses of
Proposition~\ref{prop:elliptic-case}.  Put
$Q_r=\T^r(P)$ for $r=0,1,2$,
and
$$
t_j=\left(\frac{\lambda_+}{\lambda_0}\right)^j,
\qquad
\tau=\left|\frac{\lambda_+}{\lambda_0}\right|
\approx0.037625.
$$
An~exact calculation in $\mathbb Q(\sqrt5)$ shows that the~first three
oriented areas are nonzero.  Their rounded values are
$$
\begin{aligned}
\sArea(Q_0)
& \approx 0.821178,
&
\sArea(Q_1)
&
 \approx-0.027618,
 &
\sArea(Q_2)
&
 \approx 0.000741.
\end{aligned}
$$
Choose eigenvectors
$v_0,v_+,\overline v_+$ with $(v_0)_3=1$.  The~$100$-digit spectral
decomposition gives $c_{r,i}\ne0$ for all residue vertices and allows
us to write
$$
\widetilde q_{r,i}
=c_{r,i}\bigl(v_0+\zeta_{r,i}v_+
+\overline{\zeta_{r,i}}\,\overline v_+\bigr).
$$
After $j$ powers of $G$, its affine denominator, apart from the~nonzero
factor $c_{r,i}\lambda_0^j$, is therefore
$$
D_{r,i}(t_j)
=
1+2\operatorname{Re}
\bigl(\zeta_{r,i}(v_+)_3t_j\bigr).
$$
For $j\geq1$ we have $|t_j|\leq \tau$.  A~$100$-digit computation gives
the~following uniform numerical estimates:
$$
\max_i\sup_{|t|\leq \tau}|D_{r,i}(t)-1|
<0.2110,\quad0.01540,\quad0.002397
$$
for $r=0,1,2$, respectively.  The~$j=0$ denominators are the~original
nonzero affine denominators.  If certified as interval bounds, these
inequalities would show that no forward residue vertex reaches the~line
at infinity; here they are retained as high-precision numerical
evidence.

The~signed area is a~rational function of $t$ and $\bar t$.  If either
variable is set to zero, the~complexified vertices lie on one affine
spectral line, so the~area vanishes.  After removing the~nonzero affine
denominators, the~numerator is therefore divisible by $t\bar t$, and
the~same decomposition gives
$$
\sArea(G^j(Q_r))
=|t_j|^2F_r(t_j,\overline{t_j}),
$$
where $F_r(t,\bar t)$ is continuous for $|t|\leq \tau$.  Put
$b_r=F_r(0,0)$.  The~rounded leading values are
$$
b_0\approx-0.607424,
\quad
b_1\approx-0.033920,
\quad
b_2\approx0.000737.
$$
Using the~preceding denominator bounds term by term in the~polygonal
area formula gives
$$
|F_r(t,\bar t)-b_r|
<0.2270,\quad0.000658,\quad0.000005
\qquad (|t|\leq \tau).
$$
Each reported error bound is smaller than $|b_r|$.  Certified interval
bounds of this form, together with the~exact nonzero areas at $j=0$,
would prove that every forward residue area is nonzero.  The~same
high-precision computation indicates that every leading centered vector
$w_{r,i}$ in Proposition~\ref{prop:elliptic-case} is nonzero.  Thus the~
calculation supports all hypotheses of that proposition in the~three
residue classes. 

For the~ellipse levels below, $B_r=|b_r|$ in the~notation of
Proposition~\ref{prop:predicting-ellipses}.  The~barycenter is
$$
C(P)
\approx(0.470251,0.495988).
$$
With the~normalization $\det H=1$, the~common positive definite
quadratic form from
Proposition~\ref{prop:predicting-ellipses} is
$$
H\approx
\begin{pmatrix}
3.174997 & 3.469432\\
3.469432 & 4.106133
\end{pmatrix}.
$$
Its eigenvalues are approximately $0.140035$ and $7.141096$, so the~
common major-to-minor semiaxis ratio is approximately $7.141096$.
Writing $C(P)=(c_x,c_y)$, the~twenty-one  asymptotic ellipses
have the~form
$$
\begin{aligned}
&3.174997(x-c_x)^2
+6.938865(x-c_x)(y-c_y)
+4.106133(y-c_y)^2=s_{r,i},
\end{aligned}
$$
for $r=0,1,2$,
and $i=1,\ldots,7$.
Equivalently, in the~original coordinates $(x,y)$,
$$
\begin{aligned}
&3.174997x^2
+6.938865xy
+4.106133y^2
-6.427681x
-7.336189y
+3.330640=s_{r,i}.
\end{aligned}
$$
The~rounded levels, in vertex order, are
$$
\begin{aligned}
(s_{0,i})={}&(
20.739175,
0.969678,
0.000539,
0.514275,
0.718270,
1.340759,
0.947431),\\[1mm]
(s_{1,i})={}&(
0.385919,
0.306329,
2.201430,
0.300967,
0.105859,
0.000732,
0.000083),\\[1mm]
(s_{2,i})={}&(
2.228464,
1.437116,
0.949837,
0.952697,
0.248658,
1.833281,
0.215740).
\end{aligned}
$$
Multiplying $H$ by a~positive constant multiplies all levels by the~
same constant, so the~geometric ellipses are independent of this
normalization.

The~example is illustrated in Figure~\ref{fig:period-three-heptagon},
together with its three families of asymptotic ellipses.  The~left
panel shows the~initial heptagon, while the~middle and right panels show
the~area-normalized orbit and the~ellipses at two different
scales.  Some ellipses are visually indistinguishable because their
levels are very close -- the~same near-coincidence is also apparent from
the~direct computations.

\end{example}

\subsection{A~projective period-four octagon}

\begin{example}[A~projective period-four octagon]
\label{ex:period-four-octagon}
Let $\sigma=\sqrt{3010}$.  Consider the~simple
nonconvex octagon
$$
P=\left(
(0,0),(1,0),(1,1),(0,1),\left(\frac{-104+2\sigma}{9},\frac{113-2\sigma}{9}\right),
\left(-\frac{8}{81},\frac19\right),
\left(-\frac19,0\right),
\left(1,-\frac98\right)
\right)
$$
(see the~first panel of Figure~\ref{fig:period-four-octagon}).

\begin{figure}[htbp]
    \centering
    \includegraphics[width=0.28\textwidth]{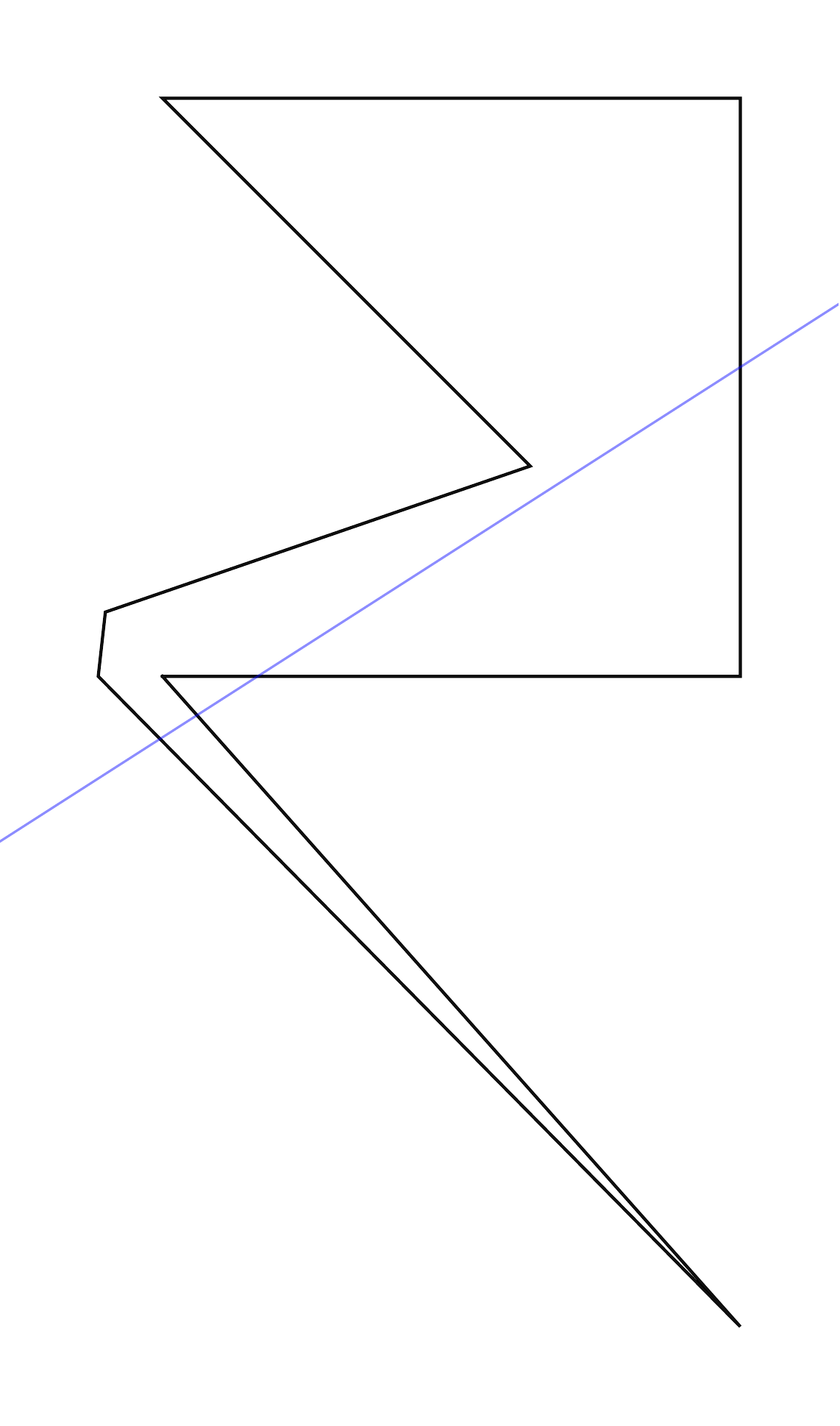}
    \hfill
    \includegraphics[width=0.70\textwidth]{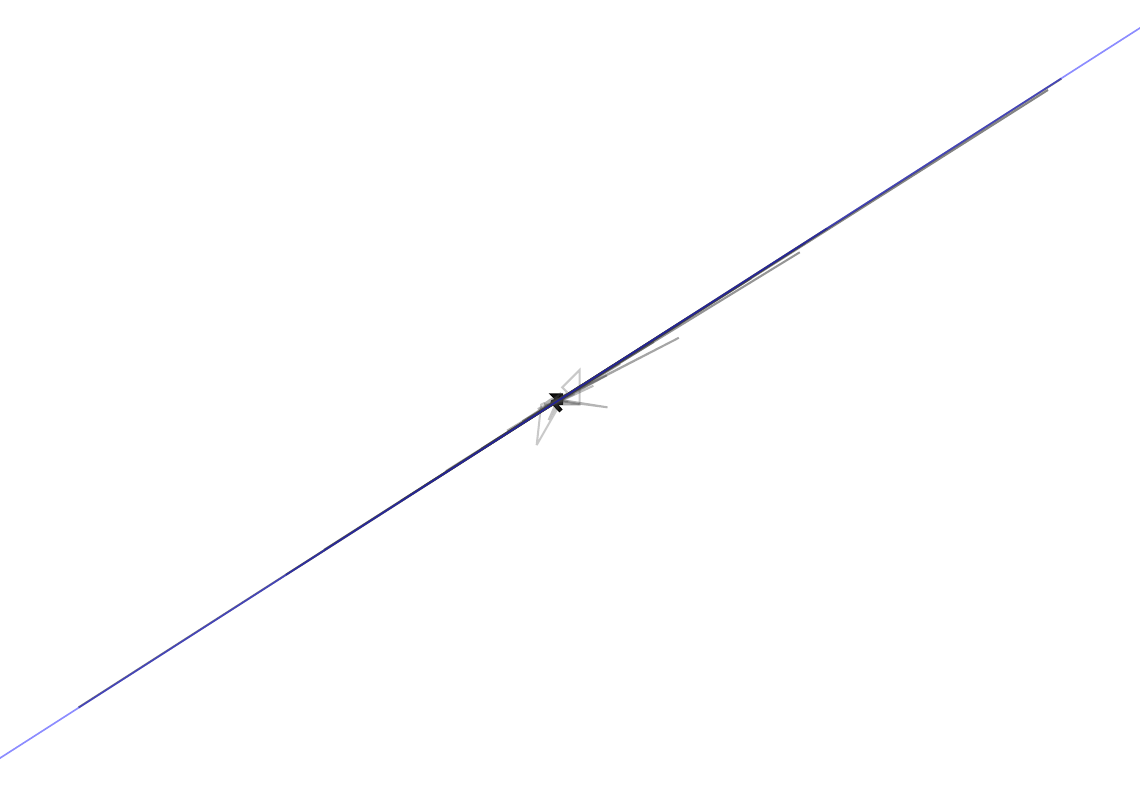}
    \caption{Numerical flattening toward the~spectral line}
    \label{fig:period-four-octagon}
\end{figure}

An~exact lift of the~return projectivity is
$$
\widetilde G=
\begin{pmatrix}
51852312+4684554\sigma & -504976378+7873054\sigma & 192104298-2612574\sigma\\
195159861-339714\sigma & 143953731-2747214\sigma & 208027359-3170826\sigma\\
247012173+4344840\sigma & -304120247+10266640\sigma & 464146857
\end{pmatrix}.
$$
Let $G$ be the~projectivity represented by $\widetilde G$.  Direct
substitution in $\mathbb Q(\sigma)$ gives $\T^4(P)=G(P)$.  Moreover, $\det\widetilde G\ne0.$
The~$3n\times9$ correspondence matrices
have rank $9$ for every cyclic shift and every $j=1,2,3$; for $j=4$
the~rank is $8$ only for the~zero shift, and the~kernel is generated by
$\widetilde G$.  Thus the~minimal projective return time is $m=4$.

With the~linear lift normalized by dividing $\widetilde G$ by its
lower-right entry, a~numerical spectral calculation gives
$$
\Spec(G)
\approx
\{1.224843,\;0.343532,\;0.082487\}.
$$
The~computed eigenvalues are positive, distinct, and have strictly
ordered moduli.  In the~chosen affine chart, after normalizing each
finite eigenvector to have third coordinate $1$, the~leading mixed-area
coefficients for the~four
residue classes are approximately
$$
0.846418,\qquad -4.661523,\qquad -0.705996,\qquad 0.052771,
$$
and all computed dominant vertex projections are nonzero.  These
checks, as well as nonvanishing of the~areas along the~computed finite
orbit segment,
provide strong numerical evidence for the~remaining hypotheses of
Theorem~\ref{thm:spectral-flattening}.  We do not claim here that the~
infinite-time forward-domain condition has been certified for this
octagon.

The~affine eigenpoints corresponding to the~two eigenvalues of largest
moduli determine the~direction $u\approx(-0.628006,-0.402221).$
Since the~area normalization preserves the~barycenter, the~spectral
line passes through $C(P)
\approx(0.428306,0.168723).$
The~spectral calculation therefore predicts, and the~computed
area-normalized orbit numerically flattens toward, $\mathcal L=C(P)+\R u,$
that is, approximately,
$$
\mathcal L:
\quad
0.539336\,x
-
0.842091\,y
-
0.088921
=0.
$$

Figure~\ref{fig:period-four-octagon} compares the~initial octagon with
the~twentieth area-normalized iterate and the~spectral line.
\end{example}

\subsection{Glick's heptagon}

Glick's paper \cite{GlickLimit} contains a~useful test example for
$n=7$.  Consider the~convex heptagon
$$
P=((2,0),(3,1),(3,2),(2,3),(1,3),(0,2),(0,1))
$$
(see the~first panel of Figure~\ref{fig:GlicksHeptagon}).  The~four panels show the~initial
configuration and the~area-normalized iterates after $5$, $20$, and
$50$ steps, each together with the~spectral tangent-line candidate
$C(P)+\R u_{12}$.

\begin{figure}[htbp]
    \centering

    \includegraphics[width=0.48\textwidth]{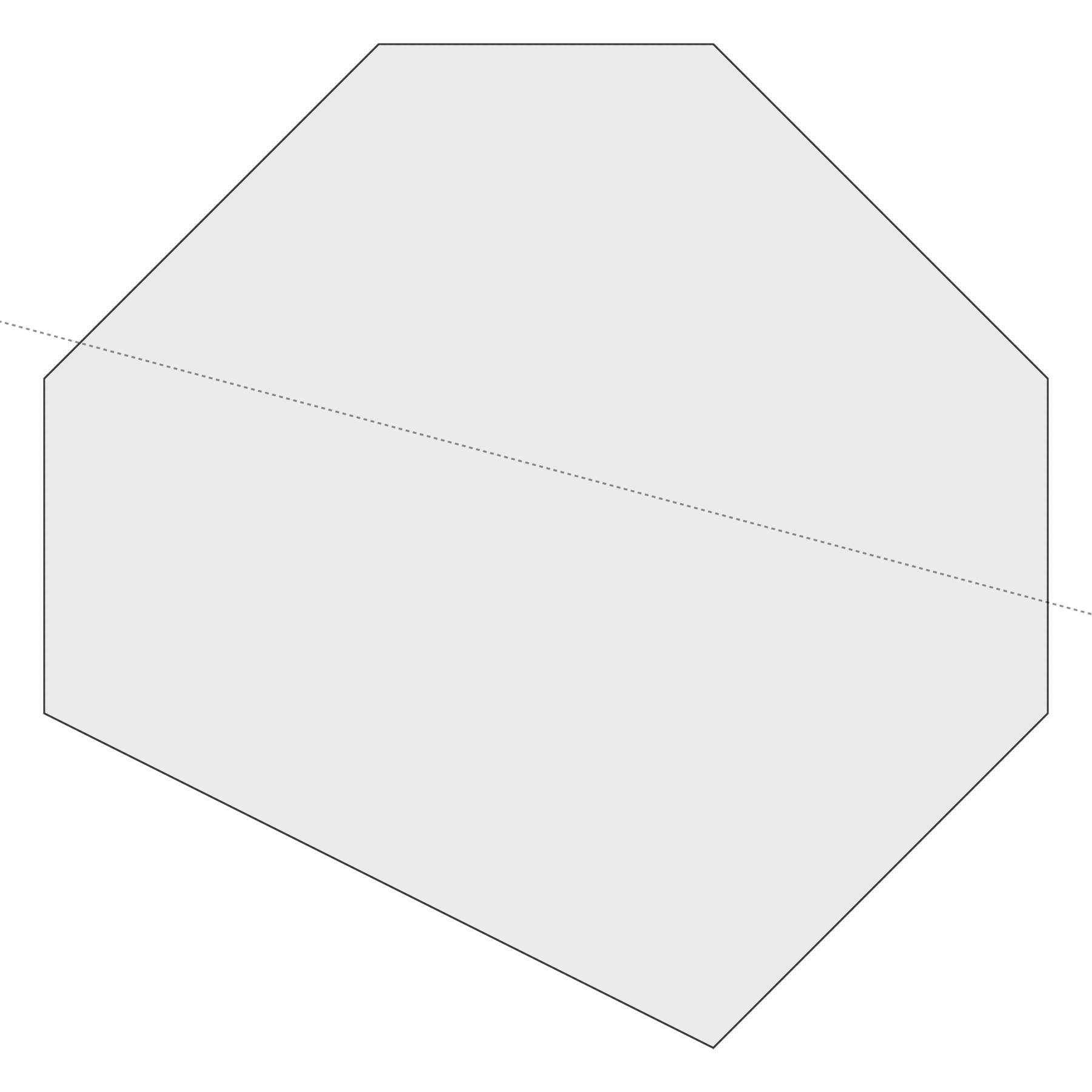}
    \hfill
    \includegraphics[width=0.48\textwidth]{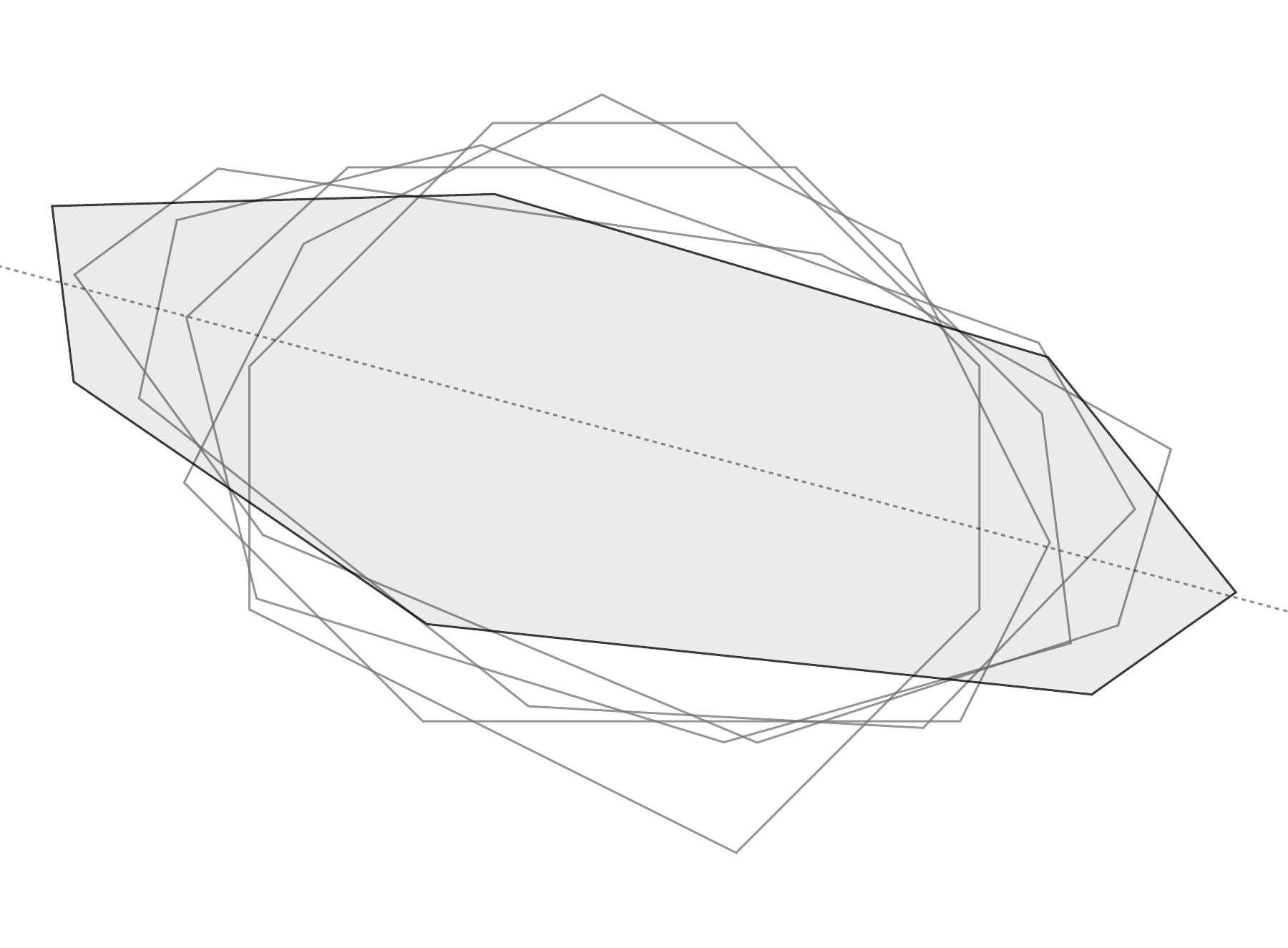}

    \vspace{0.2cm}

    \includegraphics[width=0.48\textwidth]{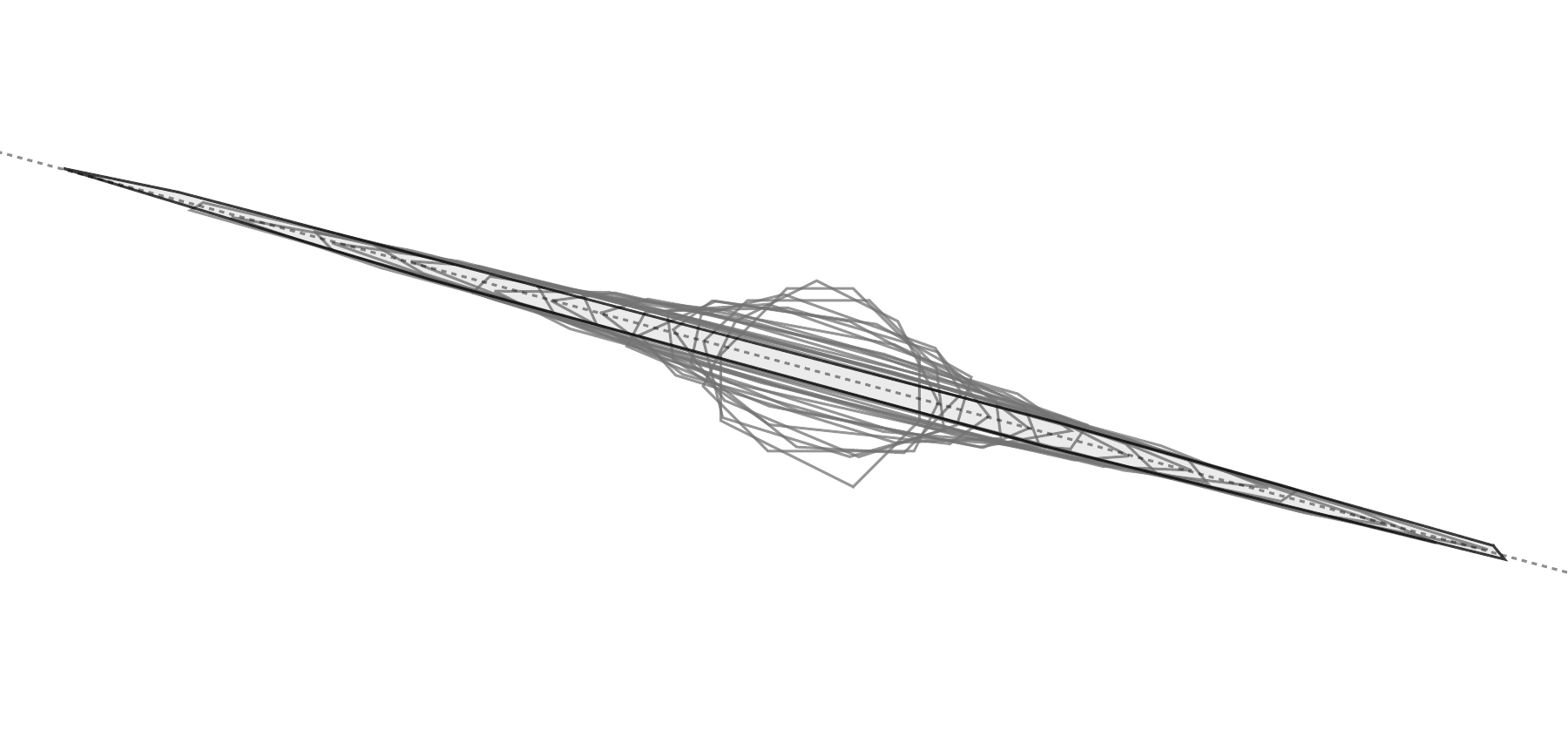}
    \hfill
    \includegraphics[width=0.48\textwidth]{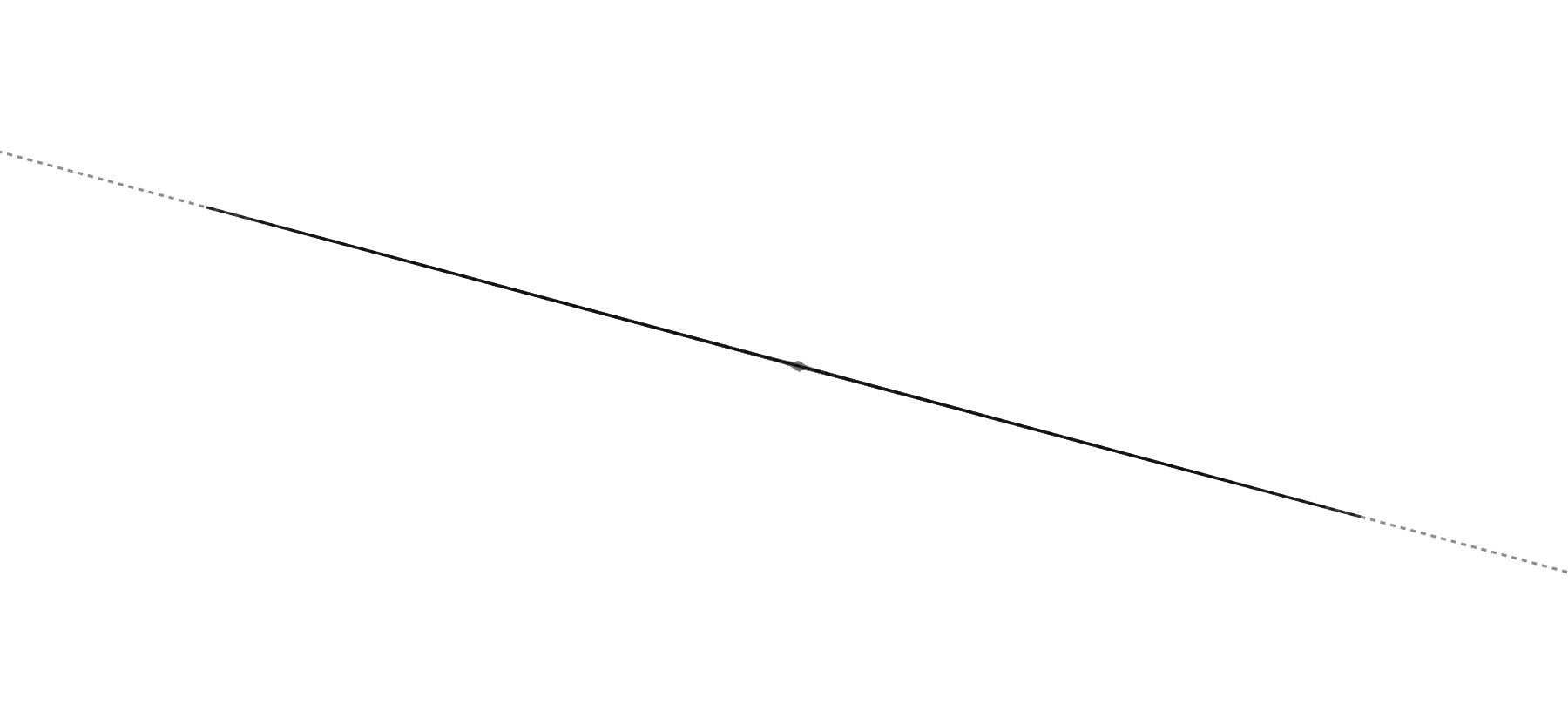}

    \caption{Area-normalized iterates of Glick's heptagon and the~
    spectral tangent-line candidate}
    \label{fig:GlicksHeptagon}
\end{figure}

For this polygon, Glick computes
$$
L_P=
\begin{pmatrix}
-6&-4&49\\
-1&-7&51\\
-1&-3&27
\end{pmatrix}.
$$
The~characteristic polynomial is
$\chi_{L_P}(\lambda)
=\lambda^3-14\lambda^2-111\lambda-116$.
Its roots, ordered by decreasing modulus, are approximately
$\lambda_1\approx 19.877720$,
$\lambda_2\approx -4.612545$, and
$\lambda_3\approx -1.265176$.
Solving
$(L_P-\lambda I)(x,y,1)^{\top}=0$
gives the~affine eigenpoint
$$
X(\lambda)=
\left(
\frac{49\lambda+139}{\lambda^2+13\lambda+38},
\frac{51\lambda+257}{\lambda^2+13\lambda+38}
\right).
$$
Thus
$$
\begin{aligned}
X_1&=X(\lambda_1)
      \approx (1.609477,\,1.837601),\\
X_2&=X(\lambda_2)
      \approx (126.564599,\,-31.650685),\\
X_3&=X(\lambda_3)
      \approx (3.325924,\,8.313084).
\end{aligned}
$$
The~eigenpoint $X_1$ is the~ordinary pentagram limit point computed by
Glick.  Since $n=7$, however, $L_P$ should not be interpreted as a~
projective generator of the~orbit.  Theorem~\ref{thm:spectral-flattening}
does not apply to this heptagon merely from the~spectral data of $L_P$.

Let us make this point explicit for the~present example.  The~heptagon is
cyclic: its vertices lie on the circle $x^2+y^2-3x-3y+2=0.$
It is not, however, a~Poncelet polygon in the~sense of
Definition~\ref{def:poncelet-polygon}.  Indeed, write the~homogeneous
side lines as $\ell_i=[a_i:b_i:c_i]$.  A~conic tangent to all sides
would be represented dually by a~nonzero symmetric matrix $Q^*$ with
$\ell_i^{\top}Q^*\ell_i=0$.  Thus its six independent coefficients
would lie in the~kernel of the~matrix whose $i$-th row is $(a_i^2,b_i^2,c_i^2,2a_ib_i,2a_ic_i,2b_ic_i).$
For the~first six sides of this heptagon, the~determinant of that
$6\times6$ matrix is $96$.  Hence $Q^*=0$ is the~only solution, so no
nondegenerate conic is tangent to all seven sides.

Separately, a~finite-time projectivity test finds no short projective
return.  For a~fixed pair $(m,s)$ one can look for a~projectivity $H$ satisfying
$$
H(p_i)=\T^m(P)_{i+s},
\qquad i=1,\ldots,7,
$$
by solving the~resulting homogeneous linear equations for the~entries of
a~$3\times3$ matrix $H$.  Exact rational rank computations for all
$1\leq m\leq100$ and all cyclic shifts $s=0,\ldots,6$ show that this
system has no invertible solution.  Thus, in the~tested range, no
relation of the~form
$\T^m(P)=\Sigma_s(G(P))$
is present.  Also, $\Sigma_s((L_P-3I)(P))$ is neither $\T(P)$ nor
$\T^2(P)$ for any cyclic shift $s$.  In particular, this heptagon should
not be treated as an~example of the~pentagon/hexagon type of projective
periodicity.

Instead, the~data suggest two natural tangent directions at the~
projective limit point $X_1$, namely the~directions of the~two lines
$X_1X_2$ and $X_1X_3$.
The~first of these has normalized direction
$u_{12}\approx
(0.965913,\,-0.258867)$,
while the~second has normalized direction
\linebreak $u_{13}\approx
(0.256220,\,0.966618)$.
The~barycenter of the~initial heptagon is $C(P)=\left(\frac{11}{7},\frac{12}{7}\right).$
Thus the~two affine tangent-line candidates visible in the~
area-normalized picture are
$$
C(P)+\R u_{12}
\qquad\text{and}\qquad
C(P)+\R u_{13}.
$$
In numerical experiments for general polygons with $n\geq7$, both types
of tangent line can occur, whereas the~complementary line $X_2X_3$ has
not appeared as a~flat limiting direction.  This motivates the~
conjectural tangent-line selection principle formulated in the~final
section.  For Glick's heptagon, the~computed iterates shown in
Figure~\ref{fig:GlicksHeptagon} align with $C(P)+\R u_{12}$.

\begin{remark}[Complex spectra outside the~projectively periodic case]
The~elliptic asymptotic proved in
Section~\ref{sec:elliptic-oscillatory} use an~actual projectivity $G$
which generates a~fixed iterate of the~orbit, possibly up to cyclic
relabelling.  Therefore they cannot be applied to the~conserved
operator $L_P$ alone when $n\geq7$.  The~distinct numerical behaviors
of the~two complex spectral configurations, and the~resulting
line-selection proposal, are stated in
Conjecture~\ref{conj:tangent-line-selection} and the~discussion that
follows it.
\end{remark}
\section{Conjectures}
\label{sec:conjectures}
\noindent
The~results proved above concern polygons whose orbits, after possibly
passing to an~iterate, are generated by a~single projectivity.  For
general polygons with $n\geq 7$, Glick's operator remains conserved but
does not generate the~pentagram orbit.  Its eigendata should therefore
be regarded as organizing data for the~observed asymptotics.

Numerical experiments suggest the~following two stable regimes.

\begin{conjecture}[Spectral-line selection for $n\geq 7$]
\label{conj:tangent-line-selection}
Let $P\in\mathcal U^\infty$ be a~labeled $n$-gon with $n\geq 7$.

\begin{enumerate}[label=\textup{(\roman*)}]

\item \textbf{\emph{Distinct real spectrum.}}
Assume that $L_P$ has three real eigenvalues
$\lambda_1,\lambda_2,\lambda_3$ satisfying
$|\lambda_1|>|\lambda_2|>|\lambda_3|$,
with corresponding eigenpoints $X_1,X_2,X_3$. Suppose that $X_1$ is
finite and that the~ordinary pentagram iterates collapse projectively
to $X_1$, in the~sense that every labeled vertex converges to $X_1$ in
$\RP^2$.

\begin{enumerate}[label=\textup{(\alph*)}]

\item
If $P$ is convex, then the~area-normalized iterates flatten to the~
affine line through $C(P)$ parallel to $X_1X_2$.

\item
If $P$ is nonconvex, then the~area-normalized iterates flatten to
exactly one of the~two affine lines through $C(P)$ parallel to
$X_1X_2$ or
$X_1X_3$.
\end{enumerate}

\item \textbf{\emph{Dominant non-real pair.}}
Assume that
$\Spec(L_P)=\{\lambda,\overline{\lambda},\mu\}$,
where $\lambda\notin\R$,
$|\lambda|>|\mu|$.
If $v$ is an~eigenvector corresponding to $\lambda$, set
$\ell_{\lambda}=
\mathbb P\left(
\Span_{\R}
\{\operatorname{Re}v,\operatorname{Im}v\}
\right)$.
Assume that $\ell_{\lambda}$ is not the~line at infinity, and denote
its affine direction by $u_{\lambda}$. Then the~area-normalized
iterates flatten to
$C(P)+\R u_{\lambda}$.
\end{enumerate}

In both cases, the~diameter tends to infinity and the~directions of the~
principal lines converge to the~asserted limiting direction.
\end{conjecture}

The~remaining generic spectral configuration should be treated
separately.  If $L_P$ has one dominant real eigenvalue and a~
subdominant non-real pair, our experiments for general $n\geq 7$ show
no stable universal pattern.  In particular, neither the~line
predictions nor the~elliptic formulae proved for projectively periodic
orbits appear to follow from the~spectrum of $L_P$ alone.

The~central open problem is to understand the~relative dynamics of the~
pentagram orbit with respect to the~conserved eigendata of $L_P$, and
in particular the~mechanism that selects between the~two spectral
tangent directions.

\bibliographystyle{pentagramdoi}
\bibliography{pentagram_map_references}

\end{document}